\documentclass[10pt, oneside]{amsart}
\usepackage[lmargin=1in,rmargin=1in, bmargin=1in, tmargin=1in]{geometry}
\usepackage{amsmath, amssymb,tikz}
\usepackage{tikz-cd}
\usepackage{tikz}
\usepackage{todonotes}
\usepackage{mathrsfs}
\usepackage{extarrows}
\usepackage{graphicx}

\usepackage[pagebackref, hidelinks]{hyperref}
\usepackage{IEEEtrantools}
\usepackage{mathrsfs}
\usepackage[utf8]{inputenc}
\usepackage[english]{babel}
 
\usepackage{comment}
\usepackage{csquotes}

\usepackage[shortlabels]{enumitem}

\usepackage{bbm}

\usepackage{amsmath,amsthm,amssymb}
\usepackage{colonequals}

 \usepackage[all,cmtip]{xy}
 
\newcommand{\defeq}{\colonequals}

\renewcommand{\tilde}{\widetilde}  

\DeclareMathOperator{\Gal}{Gal}    

\numberwithin{equation}{subsection}
\newtheorem{thm}[equation]{Theorem}
\newtheorem{prop}[equation]{Proposition}
\newtheorem{theorem}[equation]{Theorem}
\newtheorem{proposition}[equation]{Proposition}
\newtheorem{lemma}[equation]{Lemma}
\newtheorem{corollary}[equation]{Corollary}
\newtheorem*{corollary*}{Corollary}

\newtheorem{assumption}[equation]{Assumption}

\newcounter{alphalabels}

\newtheorem{theoremx}[alphalabels]{Theorem}

\theoremstyle{definition}

\newtheorem{definition}[equation]{Definition}
\newtheorem{notation}[equation]{Notation}
\newtheorem{convention}[equation]{Convention}

\theoremstyle{remark}
\newtheorem{remark}[equation]{Remark}
\newtheorem{example}[equation]{Example}

\usepackage{tikz-cd}
\usepackage{mathrsfs}

\DeclareFontFamily{U}{wncy}{}
\DeclareFontShape{U}{wncy}{m}{n}{<->wncyr10}{}
\DeclareSymbolFont{mcy}{U}{wncy}{m}{n}
\DeclareMathSymbol{\sha}{\mathord}{mcy}{"58}

\newcommand{\mbb}[1]{\mathbb{#1}}

\newcommand{\hatot}{\hat{\otimes}}

\newcommand{\opn}[1]{\operatorname{#1}}

\newcommand{\tbyt}[4]{\left( \begin{matrix} #1 & #2 \\ #3 & #4 \end{matrix} \right)}

\newcommand{\Addresses}{{
  \bigskip
  \footnotesize

  (Graham) \textsc{Mathematical Institute, University of Oxford, Woodstock Road, Oxford OX2 6GG, United Kingdom}\par\nopagebreak
  \textit{E-mail address}: \texttt{andrew.graham@maths.ox.ac.uk}\par\nopagebreak
  \textit{ORCID:} \href{https://orcid.org/0000-0003-2538-8091}{0000-0003-2538-8091}

  \medskip 

  (Pilloni) \textsc{Institut de Math\'{e}matique d'Orsay, Universit\'{e} Paris-Saclay, F-91405 Orsay Cedex, France}\par\nopagebreak
  \textit{E-mail address}: \texttt{vincent.pilloni@universite-paris-saclay.fr}\par\nopagebreak
  \textit{ORCID:} \href{https://orcid.org/0000-0002-2729-418X}{0000-0002-2729-418X}

  \medskip

  (Rodrigues Jacinto) \textsc{Institut de Mathématiques de Marseille, Université Aix-Marseille, Marseille, France}\par\nopagebreak
  \textit{E-mail address}: \texttt{joaquin.rodrigues-jacinto@univ-amu.fr}\par\nopagebreak
  \textit{ORCID:} \href{https://orcid.org/0000-0001-9571-8899}{0000-0001-9571-8899}

}}

\usepackage{stmaryrd}

\DeclareMathOperator{\Spec}{Spec}

\def\Q{\mathbb{Q}}
\def\Z{{\mathbb{Z}}}
\def\R{{\mathbb{R}}}
\def\C{{\mathbb{C}}}
\def\N{{\mathbb{N}}}

\def\zp{{\Z_p}}

\def\qp{{\Q_p}}
\def\epsilon{\varepsilon}

\def\oscr{\mathscr{O}}
\def\Spa{\mathrm{Spa}}

\title[ ]{p-adic interpolation of Gauss--Manin connections on nearly overconvergent modular forms and p-adic L-functions}
\author{Andrew Graham, Vincent Pilloni, and Joaquín Rodrigues Jacinto}

\begin{document}

\makeatletter
\def\@tocline#1#2#3#4#5#6#7{\relax
  \ifnum #1>\c@tocdepth 
  \else
    \par \addpenalty\@secpenalty\addvspace{#2}%
    \begingroup \hyphenpenalty\@M
    \@ifempty{#4}{%
      \@tempdima\csname r@tocindent\number#1\endcsname\relax
    }{%
      \@tempdima#4\relax
    }%
    \parindent\z@ \leftskip#3\relax \advance\leftskip\@tempdima\relax
    \rightskip\@pnumwidth plus4em \parfillskip-\@pnumwidth
    #5\leavevmode\hskip-\@tempdima
      \ifcase #1
       \or\or \hskip 1em \or \hskip 2em \else \hskip 3em \fi%
      #6\nobreak\relax
    \dotfill\hbox to\@pnumwidth{\@tocpagenum{#7}}\par
    \nobreak
    \endgroup
  \fi}
\makeatother

\begin{abstract}
In this paper, we give a new geometric definition of nearly overconvergent modular forms and $p$-adically interpolate the Gauss-Manin connection on this space. This can be seen as an ``overconvergent'' version of the unipotent circle action on the space of $p$-adic modular forms, as constructed by Gouv\^{e}a and Howe. This improves on results of Andreatta--Iovita and has applications to the construction of Rankin--Selberg and triple product $p$-adic $L$-functions.
\end{abstract}

\maketitle
\tableofcontents

\section{Introduction}

One of the central tools for studying complex $L$-functions associated with cuspidal modular eigenforms (or convolutions of such forms) is the Maass--Shimura differential operator, which acts on the space of nearly holomorphic modular forms (first introduced by Shimura to study the algebraicity properties of the Rankin--Selberg $L$-function \cite{Shimura76, Shimura86}). For various applications in number theory and arithmetic geometry, such as for the construction of $p$-adic $L$-functions, it is often desirable to understand how this operator varies $p$-adically for a given prime number $p$.

Since the Maass--Shimura operator is closely related with the Gauss--Manin connection (which makes sense in the $p$-adic world), one is led to understand how powers of this connection vary $p$-adically on an appropriately defined space of ``$p$-adic modular forms''. In the ordinary setting (i.e. as the modular forms vary in Hida families), the situation is well-understood -- one can use the ``unit root splitting'' to construct the Atkin--Serre operator on Katz $p$-adic modular forms, which is significantly easier to analyse than the full Gauss--Manin connection (on $q$-expansions it is simply the operator $\Theta = q \frac{d}{dq}$). Since the space of Katz $p$-adic modular forms comes equipped with an ordinary idempotent for the $U_p$-operator and satisfies a classicality theorem, this is often sufficient for the construction $p$-adic $L$-functions in the ordinary setting (e.g. see \cite{DarmonRotgerGrossZagier1}).

In the non-ordinary case the situation is much more delicate:
\begin{itemize} 
\item On the one hand, the desired space of ``$p$-adic modular forms'' (which we denote by $\mathscr{N}^{\dagger}$ and refer to as nearly overconvergent modular forms) should have a good spectral theory for the $U_p$-operator (i.e. one can define $\leq h$-slope parts for any integer $h \geq 0$). This property means that the space cannot be ``too large'' -- for example, this property is not satisfied by the space of Katz p-adic modular forms.
\item On the other hand, the space $\mathscr{N}^{\dagger}$ should carry an action of the Gauss--Manin connection which can be $p$-adically iterated, which means that $\mathscr{N}^{\dagger}$ can't be ``too small'' -- for example, the space of nearly overconvergent modular forms as defined by Urban \cite{Urban}\footnote{Denoted by $\mathcal{N}^{\infty, \rho}_{\mathfrak{U}}$ in \cite[\S 3.3]{Urban}.} has a good spectral theory for $U_p$, but is too small for $p$-adically iterating the Gauss--Manin connection.
\end{itemize}
To put this another way, one looks for an optimal intermediate space between classical modular forms and Katz $p$-adic modular forms, which has a good spectral theory for $U_p$, and on which the Gauss--Manin connection can be $p$-adically iterated. If such a space exists, it is likely that the method for constructing $p$-adic $L$-functions in the ordinary setting generalises to the finite-slope setting.

In \cite{AI_LLL} the authors gave a construction of such a space of nearly overconvergent modular forms and $p$-adically iterated powers of the Gauss-Manin connection. However there is an analyticity assumption on the weight, and hence on the $p$-adic $L$-functions that they construct. We also point out that in \cite{HarronXiao}, the authors propose a construction of families of nearly overconvergent Siegel modular forms, but they do not consider $p$-adic iterations of the Gauss-Manin connection. The aim of this paper is to provide a new approach for the construction of nearly overconvergent modular forms which is optimal in the aforementioned sense, i.e., which has a good spectral theory for the $U_p$ operator and over which we have a full $p$-adic interpolation of the Gauss-Manin connection (where in particular the analyticity assumption of \cite{AI_LLL} is removed). As an application, we describe how this can be applied to the construction of triple product and Rankin-Selberg $p$-adic $L$-functions in families of finite slope modular forms.

\subsection{The main result} \label{TheMainResultIntroSubSec}

Throughout the main body of the article we will take $G = \opn{GL}_2$ and $P=B \subset G$ the upper triangular Borel. We let $T \subset B$ denote the diagonal torus with character group $X^*(T)$ and we identify characters $\kappa \in X^*(T)$ with pairs $(k; w) \in \Z^2$ via the rule $\opn{diag}(t, t^{-1}z) \mapsto t^k z^w$. We let $X^\star(T)^+$ be the cone of dominant weights given by the condition $k \geq 0$. In particular, the sum of positive roots, denoted $2\rho$, corresponds to the character $(2; -1)$. We let $\overline{P} = \overline{B} \subset G$ denote the lower triangular Borel subgroup.

Fix a neat compact open subgroup $K^p \subset G(\mbb{A}_f^p)$ and let $P_{\opn{dR}}$ denote the $\overline{B}$-torsor over the modular curve $X_K$ of level $K=\opn{GL}_2(\mbb{Z}_p) K^p$ parameterising frames of the first de Rham homology respecting the Hodge filtration. Then one can consider:
\begin{itemize}
    \item The space of nearly holomorphic modular forms $\mathscr{N}^{\opn{hol}}$ of level $K$, obtained as sections of the de Rham torsor $P_{\opn{dR}}$. This space carries an action of $T$ (through the torsor structure), and the isotypic pieces under this action are nearly holomorphic forms of a prescribed weight. One has an ascending filtration $\opn{Fil}_r\mathscr{N}^{\opn{hol}}$ given by nearly holomorphic modular forms of degree $\leq r$ (see Proposition \ref{PropNearlyOverconvergent}). Furthermore, the Gauss--Manin connection induces an operator $\nabla \colon \mathscr{N}^{\opn{hol}} \to \mathscr{N}^{\opn{hol}}$ which satisfies Griffiths transversality and shifts the weight by $2\rho$.
    \item The space of $p$-adic modular forms $\mathscr{M}$ of level $K$, obtained as sections of the Igusa tower $\mathcal{IG}$ (which is a pro\'{e}tale $T(\mbb{Z}_p)$-torsor) over the ordinary locus inside the modular curve $X_K$. The space $\mathscr{M}$ therefore comes equipped with an action of $T(\mbb{Z}_p)$ and, taking isotypic subspaces for this action, one can consider the space $\mathscr{M}_{\kappa}$ of weight $\kappa$ $p$-adic modular forms, for any locally analytic character of $T(\mbb{Z}_p)$. The Gauss--Manin connection composed with the unit root splitting induces the Atkin--Serre differential operator $\Theta \colon \mathscr{M} \to \mathscr{M}$.
\end{itemize}

There is a commutative diagram
\begin{equation} \label{NholToMpadicDiagram}
\begin{tikzcd}
\mathscr{N}^{\opn{hol}} \arrow[r] \arrow[d, "\nabla"'] & \mathscr{M} \arrow[d, "\Theta"] \\
\mathscr{N}^{\opn{hol}} \arrow[r]                      & \mathscr{M}                    
\end{tikzcd}
\end{equation}
where the horizontal arrows are induced from restricting to the ordinary locus and applying the unit root splitting. In addition to this, the space of $p$-adic modular forms comes equipped with a $U_p$-operator, a Frobenius operator $\varphi$ and the diamond operator $S_p$, and, as explained in \cite{howe2020unipotent}, the space $\mathscr{M}$ carries an action of continuous functions $C_{\opn{cont}}(\mbb{Z}_p, \mbb{Q}_p)$ such that the identity function acts as $\Theta$. The main result of this article is the following overconvergent version of this picture.

\begin{theoremx} \label{TheoIntro1}
There exists a space $\mathscr{N}^{\dagger}$ of nearly overconvergent modular forms, which is an $LB$-space of compact type and which comes equipped with actions of $U_p$, $\varphi$, $S_p$, $T(\Z_p)$, locally analytic functions $C^{\opn{la}}(\Z_p, \Q_p)$, and with an ascending filtration $\opn{Fil}_r \mathscr{N}^{\dagger}$. Moreover, there is a natural map $\mathscr{N}^{\dagger} \to \mathscr{M}$ which is compatible with all these actions and $\mathscr{N}^{\dagger}$ satisfies the following additional properties:
\begin{enumerate}
\item The filtration is stable under $U_p$, $S_p$, $\varphi$, and $T(\Z_p)$. Moreover, if $\kappa \colon T(\Z_p) \rightarrow \C_p^\times$ is the character given by $\opn{diag}(t, t^{-1}z) \mapsto t^{k} z^w$, then $\opn{Fil}_r \mathscr{N}^\dagger_\kappa$ (the $r$-th filtered piece of nearly overconvergent modular forms of weight $\kappa$) is the space of nearly overconvergent modular forms of weight $k$ and degree $r$ as constructed in \cite{Urban}. In particular, the composition of the inclusion $\opn{Fil}_r\mathscr{N}^{\dagger}_{\kappa} \subset \mathscr{N}^{\dagger}$ with the map $\mathscr{N}^\dagger \to \mathscr{M}$ induces the usual inclusion of nearly overconvergent modular forms into $p$-adic modular forms.

\item For any locally analytic character $\kappa \colon T(\Z_p) \rightarrow \C_p^\times$, the weight $\kappa$ nearly overconvergent forms $\mathscr{N}^{\dagger}_{\kappa}$ admit slope decompositions with respect to $U_p$, and the finite-slope part agrees with the finite-slope part of $\bigcup_{r \geq 0} \opn{Fil}_r\mathscr{N}^{\dagger}_{\kappa}$. 

\item The operator $S_p$ is invertible, commutes with $U_p$ and $\varphi$, and satisfies $U_p \circ \varphi = p S_p$. Furthermore, for any $\phi \in C^{\opn{la}}(\mbb{Z}_p, \mbb{Q}_p)$ and $t = \opn{diag}(t_1, t_2) \in T(\mbb{Z}_p)$, we have
\[
t \circ \phi = \phi(t_2^{-1} \cdot - \cdot t_1) \circ t, \quad U_p \circ \phi = \phi(p \cdot -) \circ U_p, \quad \varphi \circ \phi(p \cdot -) = \phi \circ \varphi, \quad S_p \circ \phi = \phi \circ S_p
\]
as endomorphisms of $\mathscr{N}^{\dagger}$.

\item One has a $T(\mbb{Z}_p)$-equivariant factorisation $\mathscr{N}^{\opn{hol}} \hookrightarrow \mathscr{N}^{\dagger} \to \mathscr{M}$ of the horizontal map in (\ref{NholToMpadicDiagram}). In particular, the action of the identity function in $C^{\opn{la}}(\mbb{Z}_p, \mbb{Q}_p)$ extends the Gauss--Manin connection $\nabla \colon \mathscr{N}^{\opn{hol}} \to \mathscr{N}^{\opn{hol}}$, and the action of the indicator function $1_{\mbb{Z}_p^{\times}} \in C^{\opn{la}}(\mbb{Z}_p, \mbb{Q}_p)$ is given by the $p$-depletion operator $1 - p^{-1}S_p^{-1} \varphi U_p$.
\end{enumerate}
\end{theoremx}

\begin{remark}
The existence of the action of $C^{\opn{la}}(\mbb{Z}_p, \mbb{Q}_p)$ in Theorem \ref{TheoIntro1} allows one to $p$-adically interpolate powers of the Gauss--Manin connection. Indeed for any locally analytic character $\chi \colon \mbb{Z}_p^{\times} \to \mbb{C}_p^{\times}$ one can define 
\[
\nabla^{\chi} f^{[p]} \defeq (\chi \cdot 1_{\mbb{Z}_p^{\times}} ) \star f, \quad \quad f \in \mathscr{N}^{\dagger} \hatot_{\mbb{Q}_p} \mbb{C}_p
\]
where $\star$ denotes the action of $C^{\opn{la}}(\mbb{Z}_p, \mbb{Q}_p) \hatot \mbb{C}_p$ and $\chi \cdot 1_{\mbb{Z}_p^{\times}} \colon \mbb{Z}_p \to \mbb{C}_p$ denotes the extension by zero of the character $\chi$.
\end{remark}

\begin{remark}
It is possible to prove an analogue of Theorem \ref{TheoIntro1} for any level at $p$ (see Theorem \ref{SheafVersionMainThm} and \S \ref{RedToHypLevelSSec}). 
\end{remark}

The idea for constructing $\mathscr{N}^{\dagger}$ is as follows. Let $M_{\opn{dR}} = P_{\opn{dR}} \times^{\overline{B}} T$ denote the pushout of $P_{\opn{dR}}$ along the map $\overline{B} \twoheadrightarrow T$. This is a scheme whose global sections are equipped with an action of $T$, and the isotypic pieces under this action are holomorphic forms of a prescribed weight and level $K$. There is a natural map $\mathcal{IG} \to M_{\opn{dR}}^{\opn{an}}$ from the Igusa tower to the analytification of $M_{\opn{dR}}$ providing a reduction of structure to a $T(\mbb{Z}_p)$-torsor over the ordinary locus. One can interpret the space of overconvergent modular forms $\mathscr{M}^{\dagger}$ as overconvergent functions on the closure of $\mathcal{IG}$ inside $M_{\opn{dR}}^{\opn{an}}$ (see \cite{Pilloni, AIS}). To define $\mathscr{N}^{\dagger}$, we follow a similar approach, but now we make use of the unit root splitting. More precisely, the unit root splitting and the universal trivialisation give rise to a natural map $\mathcal{IG} \to P_{\opn{dR}}^{\opn{an}}$ extending the morphism $\mathcal{IG} \to M_{\opn{dR}}^{\opn{an}}$, which can be encoded in the following important diagram:
\begin{equation*}
\begin{tikzcd}
& P_{\opn{dR}}^{\opn{an}} \arrow[d] \\
\mathcal{IG} \arrow[r] \arrow[d] \arrow[ur] & M_{\opn{dR}}^{\opn{an}} \arrow[d] \\
\mathcal{X}_{\opn{ord}} \arrow[r] & \mathcal{X}.
\end{tikzcd}
\end{equation*}
where $\mathcal{X}$ (resp. $\mathcal{X}_{\opn{ord}}$) denotes the adic generic fibre of the $p$-adic formal completion $\mathfrak{X}$ of $X_{\mbb{Z}_p}$ (resp. the ordinary locus $\mathfrak{X}_{\opn{ord}} \subset \mathfrak{X}$). Inspired by the above interpretation of overconvergent modular forms, we define the space of nearly overconvergent modular forms $\mathscr{N}^{\dagger}$ as the space of overconvergent functions on the closure of $\mathcal{IG}$ inside $P_{\opn{dR}}^{\opn{an}}$. The existence of the Hecke and $T(\mbb{Z}_p)$ actions on $\mathscr{N}^{\dagger}$ follow almost immediately from this description.

The main technical part of this article is the construction of the action of locally analytic functions $C^{\opn{la}}(\mbb{Z}_p, \mbb{Q}_p)$ on $\mathscr{N}^{\dagger}$ whose proof, occupying \S \ref{SectionNOCMFs}--\S \ref{GMinterpolationSection}, we now sketch. We hope that this sketch will also serve as guide for the reader interested in the details of the proof. We first note that $\nabla$ extends to a bounded linear derivation on $\mathscr{N}^{\dagger}$. Let $f \in C^{\opn{la}}(\Z_p, \Q_p)$ be a locally analytic function. By a classical theorem of Amice, there exists a real number $\varepsilon > 0$ such that the Mahler expansion $f(x) = \sum_{k \geq 0} a_k \binom{x}{k}$ of the function $f$ satisfies the condition: $p^{\varepsilon k}|a_k|_p \to 0$ as $k \to +\infty$. We will denote by $C_{\varepsilon}(\Z_p, \Q_p) \subset C_{\opn{cont}}(\mbb{Z}_p, \mbb{Q}_p)$ the subspace of continuous functions whose coefficients satisfy this growth condition for some fixed $\varepsilon > 0$. Then, in order to make sense of the expression $f(\nabla) = \sum_{k \geq 0} a_k \binom{\nabla}{k}$, it suffices to show that the norm of the operators
\[
p^{k \varepsilon} \binom{\nabla}{k} = p^{k \varepsilon} \frac{\nabla (\nabla - 1) \cdots (\nabla - k+1)}{k!}, \quad \quad k \geq 1
\]
are uniformly bounded in $k$, for any $\varepsilon > 0$.

We first consider the local setting over a quasi-compact open affinoid $U = \Spa(A, A^+) \subseteq \mathcal{X}$ arising as the adic generic fibre of an affine open subscheme $\opn{Spf}(A^+) \subset \mathfrak{X}$. Set $\Spa(A_{\opn{ord}}, A^+_{\opn{ord}}) = U \times_{\mathcal{X}} \mathcal{X}_{\opn{ord}}$,  $\mathcal{IG}_{A} \defeq \mathcal{IG} \times_{\mathcal{X}} \Spa(A, A^+) = \Spa(A_{\opn{ord}, \infty}, A^+_{\opn{ord}, \infty})$, and \[ \mathscr{N}^{\dagger}_A \defeq \opn{colim}_V \opn{H}^0\left( V, \mathcal{O}_V \right), \] where the colimit is over all open subsets of $P_{\opn{dR}, A}^{\opn{an}} \defeq P_{\opn{dR}}^{\opn{an}} \times_{\mathcal{X}} U$ containing the closure of $\mathcal{IG}_A$. The strategy is to analyse the Gauss--Manin connection first over the ordinary locus, and then on certain overconvergent neighbourhoods.

More precisely, in \S \ref{OrdinaryNHoodsSubSec} we construct, using the moduli space interpretation of the Igusa tower $\mathcal{IG}_{A}$, a cofinal system of strict quasi-compact open neighbourhoods $\{ \mathcal{U}_{\opn{HT}, n} \}_{n \geq 1}$ of $\mathcal{IG}_{A}$ inside $P_{\opn{dR}, A_{\opn{ord}}}^{\opn{an}} = P_{\opn{dR}}^{\opn{an}} \times_{\mathcal{X}} \Spa (A_{\opn{ord}}, A_{\opn{ord}}^+)$ which (after pulling back to $\mathcal{IG}_{A}$) is simply a disjoint union over $\lambda \in T(\mbb{Z}/p^n\mbb{Z})$ of finitely many rigid balls $\mathcal{B}^{\lambda} \cong \left(
\begin{smallmatrix}
    1 + p^n \mathbb{G}_{a}^{\opn{an}} & \\ p^n \mathbb{G}_a^{\opn{an}} & 1 + p^n \mathbb{G}_a^{\opn{an}} \end{smallmatrix} \right)$
of radius $p^{-n}$, i.e. we have
\[ 
\mathcal{B}^\lambda \cong \Spa \left( A_{\opn{ord}, \infty} \langle \frac{X-1}{p^n}, \frac{Y}{p^n}, \frac{Z - 1}{p^n}\rangle, A^+_{\opn{ord}, \infty} \langle \frac{X-1}{p^n}, \frac{Y}{p^n}, \frac{Z - 1}{p^n} \rangle \right). 
\]
We then show (Proposition \ref{PropExplicitForm}) that the restriction of the Gauss-Manin connection to each of these balls has a very simple description modulo $p^n$, namely, it extends the Atkin-Serre operator on $A^+_{\opn{ord}, \infty}$ (which we know is integral and extends to an action of $C_{\opn{cont}}(\mbb{Z}_p, \mbb{Z}_p)$) and is ``nilpotent'' in the sense that $\nabla(X) = Y$, $\nabla(Y) = \nabla(Z) = 0$. One can show (Propositions \ref{Lemmaepsan1} and \ref{Proplocanextension1}) that any derivation satisfying these properties extends to an action of $C_{\varepsilon}(\Z_p, \Q_p)$, for any $\epsilon > 0$, provided that the integer $n$ is sufficiently large. As a consequence of this, we show (Proposition \ref{GMinterpolationOverOrdProp}) that, for any integer $\varepsilon > 0$ there exists some $n(\varepsilon) \geq 1$ such that the Gauss--Manin connection extends to an action
\begin{equation} \label{EquationIntroepsilonan}
C_{\varepsilon}(\Z_p, \Q_p) \times \mathscr{N}_{\mathcal{U}_{\opn{HT}, n}} \to \mathscr{N}_{\mathcal{U}_{\opn{HT}, n}}
\end{equation}
for any $n \geq n(\epsilon)$. Here we use the notation $\mathscr{N}_{\mathcal{U}_{\opn{HT}, n}} \defeq \mathcal{O}_{\mathcal{U}_{\opn{HT}, n}}(\mathcal{U}_{\opn{HT}, n})$.

We then establish an overconvergent version of this result (Proposition \ref{GeneralUVepsilonaction}). For this, let $(\mathcal{X}_r)_{r \in \N_{\geq 1}}$ denote the usual system of neighbourhoods of the ordinary locus defined as the loci whose rank one points $| \cdot |_x$ satisfy $|h|_x \geq p^{-1/p^{r+1}}$ for any local lift $h$ of the Hasse invariant. We set $\Spa(A_{r}, A^+_r) = \mathcal{X}_{r} \times_\mathcal{X} \Spa(A, A^+)$. Since $A^+_{\opn{ord}} = A^+_{r}\langle 1/h \rangle$, any section defined over the ordinary locus can be approximated modulo arbitrary large powers of $p$ by an overconvergent section. This allows us (Proposition \ref{PropositionOCneighbourhoods}) to overconverge the system $\{ \mathcal{U}_{\opn{HT}, n} \}_{n \geq 1}$ of neighbourhoods of the Igusa tower and build a cofinal system $\{ \mathcal{U}_{n, r} \}$ of quasi-compact open strict neighbourhoods of $\mathcal{IG}_A$ inside $P_{\opn{dR}, A}^{\opn{an}}$ such that one can control the norms of the restriction maps $\mathcal{O}(\mathcal{U}_{n, r}) \to \mathcal{O}(\mathcal{U}_{n, s})$ for $s \geq r$ (see Corollary \ref{OCcorollary}). Since the system $\{ \mathcal{U}_{n, r} \}$ is cofinal, we can locally describe the space of nearly overconvergent modular forms as
\[ 
\mathscr{N}^{\dagger}_A = \varinjlim_{n, r} \mathcal{O}(\mathcal{U}_{n, r}). 
\]
Combining the above key property on the norms of the restriction maps with the fact that the Gauss--Manin connection already exists on $\mathscr{N}^{\dagger}_A$, we obtain (Proposition \ref{GeneralUVepsilonaction}) an action
\[ 
C_{\varepsilon}(\Z_p, \Q_p) \times \mathscr{N}^\dagger_A \to \mathscr{N}^\dagger_A.
\]
extending the Gauss--Manin connection, for any $\varepsilon > 0$. In other words, the Gauss--Manin connection on $\mathscr{N}^{\dagger}_A$ extends to an action of $C^{\opn{la}}(\mbb{Z}_p, \mbb{Q}_p)$. Since this action is uniquely determined by its interpolation property, it is functorial in $A$ and glues to the desired action in Theorem \ref{TheoIntro1}. We expect that this construction will readily extend to higher dimensional Shimura varieties (at least when one has a non-empty ordinary locus) -- for example, such a generalisation appears in \cite{UFJII} in the setting of unitary Shimura varieties.

Finally, we note that Theorem \ref{TheoIntro1} has immediate applications to the construction of $p$-adic $L$-functions for the triple product of modular forms (Theorem \ref{TripleProductJinterpolation}) and for the Rankin-Selberg product of two modular forms (Theorem \ref{RankinSelberginterpolation}). In particular, one can use Theorem \ref{TheoIntro1} to construct three-variable versions of these $p$-adic $L$-functions without the restriction on the weight appearing in \cite{AI_LLL}, and without appealing to the Beilinson--Flach Euler system (as in \cite{LoefflerRSNote}). We refer the reader to \S \ref{SectpadicLfunctions} for the relevant statements. Theorem \ref{TheoIntro1} is also a key ingredient in the construction of $p$-adic $L$-functions for $\opn{GSp}_4 \times \opn{GL}_2$ and $\opn{GSp}_4 \times \opn{GL}_2 \times \opn{GL}_2$ appearing in the work \cite{2024arXiv241104559G}.

\subsection{Structure of the article}

The article is organised as follows. In \S \ref{PadicMFsMainresultSec}, we introduce the main objects in the paper ($p$-adic modular forms, nearly overconvergent modular forms, etc.) and state the main theorem, which is a slightly more general version of Theorem \ref{TheoIntro1}. We also explain why it is sufficient to prove the main theorem at hyperspecial level at $p$. In \S \ref{SectionNOCMFs}, we describe certain ``explicit'' neighbourhoods of the Igusa tower in $P_{\opn{dR}}$ that are useful for constructing the action of $C^{\opn{la}}(\mbb{Z}_p, \mbb{Q}_p)$ on $\mathscr{N}^{\dagger}$, which takes place in \S \ref{GMinterpolationSection} and uses the abstract theory of actions on Banach spaces developed in \S \ref{SectionPADICINTCONT}. Finally, we describe the additional structures on nearly overconvergent modular forms (such as the filtration and $T(\mbb{Z}_p)$ and Hecke actions) in \S \ref{AdditionalStructuresSection} and the application to $p$-adic $L$-functions in \S \ref{SectpadicLfunctions}. When discussing the action of differential operators throughout the article, we find it helpful to use the language of $\mathcal{D}$-modules. We have therefore summarised the main constructions regarding this in Appendix \ref{AppendixClassicalNHMFs}. Furthermore, since the notation in \S \ref{PadicMFsMainresultSec} differs from the rest of the article, we have provided a glossary of notation in Appendix \ref{GlossaryNotAppendix}.

\subsection{Notation and conventions} \label{NotationAndConventionsSec}

We fix the following notation and conventions throughout.

\begin{itemize}
    \item We say a $\mbb{Z}_p$-algebra $S$ is \emph{admissible} if it is $p$-adically complete and separated, $p$-torsion free, and topologically of finite type over $\mbb{Z}_p$. 
    \item For any $p$-adic manifold $X$ and $\opn{LB}$-space $V$, we let $C_{\opn{cont}}(X, V)$ and $C^{\opn{la}}(X, V)$ denote the spaces of continuous and locally analytic maps $X \to V$ respectively.
    \item For an integer $n \geq 0$ and a partition $m_1 + \cdots + m_r = n$ by non-negative integers, we let 
    \[
    \genfrac{(}{)}{0pt}{}{n}{m_1, \cdots, m_r} = \frac{n!}{m_1! \cdots m_r!}
    \]
    denote the associated multinomial coefficient.
    \item All adic spectra we consider are with respect to complete Huber pairs.
    \item All $p$-adic valuations are normalised so that $|p|_x = p^{-1}$.
    \item We let $w_0$ denote the longest Weyl element of $\opn{GL}_2$.
\end{itemize}

We note that the ``de Rham'' torsor $P_{\opn{dR}}$ that we consider throughout the whole article is a torsor for the \emph{lower-triangular} Borel subgroup $\overline{P} = \overline{B}$ (so, in the language of \cite[\S 2]{CS17}, we take $\mu \colon \mbb{G}_m \to \opn{GL}_2$ to be the fixed choice of Hodge cocharacter given by $\mu(z) = \left( \begin{smallmatrix} z & \\ & 1 \end{smallmatrix} \right)$). On the other hand, our convention is that the positive root for $\opn{GL}_2$ lies in $B$ (and the notion of dominant or highest weight is with respect to this positive root). In addition to this, our weights are characters of the standard torus inside $\opn{GL}_2$ (rather than $\opn{SL}_2$), so when defining ``weight $\kappa$'' (nearly overconvergent) modular forms there is often a twist by the longest Weyl element $w_0$. Our conventions are arranged so that one can simply pass to the $\opn{SL}_2$-setting (as found in the literature) by restricting weights to the torus in $\opn{SL}_2$ and identifying elliptic curves with their dual via the principal polarisation.

\subsection{Acknowledgements}

The authors would like to thank Fabrizio Andreatta, Ananyo Kazi, Yuanyang Jiang, David Loeffler and Rob Rockwood for helpful discussions related to this article. The authors are also grateful to the anonymous referee whose comments and suggestions have significantly improved the presentation of this article. This work was supported by the grant ERC-2018-COG-818856-HiCoShiVa. AG was (partly) funded by UK Research and Innovation grant MR/V021931/1. For the purpose of Open Access, the authors have applied a CC BY public copyright licence to any Author Accepted Manuscript (AAM) version arising from this submission.

\section{Nearly overconvergent modular forms and the main result} \label{PadicMFsMainresultSec}

To state the main result, we first introduce some notation for the space of $p$-adic modular forms. We warn the reader that, in order to present the most general result possible, the notation in this section differs slightly from that in the rest of the article (as explained at the end of this section, it will be enough to establish the result with hyperspecial level at $p$, so we will often omit the level subgroup from the notation). We hope this doesn't cause any confusion; we have provided a glossary of notation in Appendix \ref{GlossaryNotAppendix} outlining the differences.

\subsection{Infinite level Igusa towers and $p$-adic modular forms}

Set $G = \opn{GL}_2$ and let $P \subset G$ denote the upper triangular Borel subgroup with diagonal torus $T$. Fix a prime number $p$ and a neat compact open subgroup $K^p \subset G(\mbb{A}^p_f)$. 

In this section, we will define the relevant spaces of $p$-adic modular forms. In order to do this, we need to introduce the infinite level Igusa varieties as constructed in \cite[\S 4.3]{CS17} (see also \cite{howe2020unipotent} and \cite[\S 3.4]{BPHHT} for some complementary details). Consider the $p$-divisible group over $\opn{Spf}\mbb{Z}_p$ given by $\mbb{X}_{\opn{ord}} = \mu_{p^{\infty}} \oplus \mbb{Q}_p/\mbb{Z}_p$ and define the following group schemes over $\opn{Spf}\mbb{Z}_p$:
\[
J_{\opn{ord}} = \mathrm{QIsog}( \mathbb{X}_{\opn{ord}}) = \begin{pmatrix} 
      \Q_p^\times & \tilde{\mu_{p^\infty}} \\
      0 & \Q_p^\times \\
   \end{pmatrix}, \quad \quad J_{\opn{ord}}^{\opn{int}} = \opn{Aut}(\mbb{X}_{\opn{ord}}) = \begin{pmatrix} 
      \mbb{Z}_p^\times & T_p(\mu_{p^\infty}) \\
      0 & \mbb{Z}_p^\times \\
   \end{pmatrix},
\]
where $\opn{QIsog}(-)$ (resp. $\opn{Aut}(-)$) denotes the self quasi-isogenies (resp. automorphisms) of a $p$-divisible group and $\tilde{\mu_{p^\infty}} = \opn{lim}_{\times p} \mu_{p^{\infty}}$ is the universal cover of $\mu_{p^{\infty}}$. Recall that the universal cover sits in an exact sequence
\begin{equation*} \label{EqExactSeqmuinfty}
0 \to T_p(\mu_{p^\infty}) \to \widetilde{\mu_{p^\infty}} \to \mu_{p^\infty} \to 0
\end{equation*}
and note that $\Q_p^\times =  \mathrm{QIsog}( \mu_{p^\infty}) =  \mathrm{QIsog}( \Q_p/\Z_p)$ (see \cite[\S 3.1]{ScholzeWeinstein}). We let 
\[
P' =    \begin{pmatrix} 
      \Q_p^\times & \Q_p(1) \\
      0 & \Q_p^\times \\
   \end{pmatrix}, \quad \quad (P')^{\opn{int}} = J_{\opn{ord}}^{\opn{int}} \times_{\opn{Spa}(\mbb{Z}_p, \mbb{Z}_p)} \opn{Spa}(\mbb{Q}_p, \mbb{Z}_p) = \begin{pmatrix} 
      \mbb{Z}_p^\times & \mbb{Z}_p(1) \\
      0 & \mbb{Z}_p^\times \\
   \end{pmatrix}
\]
which are (locally) profinite pro\'{e}tale group schemes over $\Spa(\Q_p,\Z_p)$. We let $U_{P'} \subset P'$ and $U_{P'}^{\opn{int}} \subset (P')^{\opn{int}}$ denote the unipotent radicals. There is a natural map $P' \hookrightarrow J_{\opn{ord}}^{\opn{an}} \defeq J_{\opn{ord}} \times_{\Spa (\Z_p, \Z_p)} \Spa (\Q_p, \Z_p)$ from $P'$ to the adic generic fibre of $J_{\opn{ord}}$, induced by $\Q_p(1) \rightarrow  \tilde{ \mu_{p^\infty}}$ (but note that $J_{\opn{ord}}^{\opn{an}}$ is a $1$-dimensional group scheme, and therefore much larger than $P'$). We will often write $\widehat{\mbb{G}}_m$ for the $p$-divisible group $\mu_{p^{\infty}}$ over $\opn{Spf}\mbb{Z}_p$ to emphasise the fact that $\widehat{\mbb{G}}_m(\opn{Spf}R) = 1 + R^{00}$, where $R^{00} \subset R$ denotes the ideal of topologically nilpotent elements (cf. \cite[Remark 2.1.3]{howe2020unipotent}).
 
We now recall the construction of the perfect Igusa tower $\mathfrak{IG}_{K^p}$. Let ${X}_{\mathrm{GL}_2(\Z_p)K^p} \rightarrow \Spec \Z_p $ denote the (compact) modular curve of level $\mathrm{GL}_2(\Z_p)K^p$, and $\mathfrak{X}_{\mathrm{GL}_2(\Z_p)K^p} \rightarrow \opn{Spf} \Z_p$ its completion along the special fibre. Set $\mathfrak{IG}_{(P')^{\opn{int}}K^p} \defeq \mathfrak{X}^{\opn{ord}}_{\mathrm{GL}_2(\Z_p)K^p}$ to be the ordinary locus. We also let $\mathfrak{IG}_{U_{P'}^{\opn{int}}K^p} \to \mathfrak{IG}_{(P')^{\opn{int}}K^p}$ denote the pro-\'{e}tale $T(\mbb{Z}_p)$-torsor parameterising trivialisations $\mu_{p^{\infty}} \xrightarrow{\sim} E[p^\infty]^{\circ}$ and $\Q_p/\Z_p \xrightarrow{\sim} E[p^\infty]^{\mathrm{et}}$ of the connected and \'{e}tale parts of the $p$-divisible group associated with the universal (ordinary) elliptic curve (since $E[p^\infty]^{\mathrm{et}} \cong (E[p^\infty]^{\circ})^D$ the \'{e}tale part extends to a $p$-divisible group over the boundary). Note that there is a natural lift of Frobenius $\varphi$ on $\mathfrak{IG}_{U_{P'}^{\opn{int}}K^p}$.
 
\begin{definition}
We define $\mathfrak{IG}_{K^p} = \lim_{\varphi} \mathfrak{IG}_{U_{P'}^{\opn{int}}K^p}$.  If $\mathfrak{Y}_{\mathrm{GL}_2(\Z_p)K^p} \hookrightarrow \mathfrak{X}_{\mathrm{GL}_2(\Z_p)K^p}$ denotes the good reduction locus, then $\mathfrak{IG}_{K^p} \times_{\mathfrak{X}_{\mathrm{GL}_2(\Z_p)K^p}}  \mathfrak{Y}_{\mathrm{GL}_2(\Z_p)K^p}$ parameterises elliptic curves $E$ with a $K^p$-level structure and an isomorphism $\mathbb{X}_{\opn{ord}} \xrightarrow{\sim} E[p^\infty]$ (see \cite[\S 4.3]{CS17}).
\end{definition}

\begin{prop} 
The group $J_{\opn{ord}}$ acts on $\mathfrak{IG}_{K^p}$.
\end{prop}
\begin{proof} 
By \cite[Corollary 4.3.5]{CS17}, the group $J_{\opn{ord}}$ acts on $\mathfrak{IG}_{K^p} \times_{\mathfrak{X}_{\mathrm{GL}_2(\Z_p)K^p}}  \mathfrak{Y}_{\mathrm{GL}_2(\Z_p)K^p}$. Indeed one can formulate the moduli problem for $\mathfrak{IG}_{K^p} \times_{\mathfrak{X}_{\mathrm{GL}_2(\Z_p)K^p}}  \mathfrak{Y}_{\mathrm{GL}_2(\Z_p)K^p}$ equivalently as the space of elliptic curves $E$ with a $K^p$-level structure and a quasi-isogeny $\mathbb{X}_{\opn{ord}} \rightarrow E[p^\infty]$, up to quasi-isogeny (see \cite[Lemma 4.3.4]{CS17}). We claim that the action extends to an action over $\mathfrak{IG}_{K^p}$. Since we already have an action of $T(\Q_p)$, it suffices to prove that the action of $\tilde{\mu_{p^\infty}}$ extends. By $p$-adic Fourier theory, this action amounts to an action of $C_{\opn{cont}}(\Q_p, \Z_p)$ on $\mathcal{O}_{\mathfrak{IG}_{K^p}}$ (see, e.g., \cite[\S 7]{howe2020unipotent}). On a cusp, this action is given by 
$\phi(\sum a_n q^n)= \sum \phi(n) a_n q^n$ by \cite[Theorem 7.1.1]{howe2020unipotent}, and thus preserves the regular functions at the cusp. 
\end{proof}

For any compact open subgroup $K_{p,P} \subseteq P'$ (i.e. one which is commensurable to $(P')^{\opn{int}}$), we let $\overline{K_{p, P}}$ denote its schematic closure in $J_{\opn{ord}}$. By \cite[Lemma 3.4.8]{BPHHT} the group scheme $\overline{K_{p, P}}$ is a profinite flat group scheme over $\Spec \Z_p$ with generic fibre $K_{p, P}$. We let $\mathfrak{IG}_{K_{p,P}K^p}$ be the flat formal scheme equal to the categorical quotient of $\mathfrak{IG}_{K^p}$ by $\overline{K_{p,P}}$. This is the affine formal scheme whose ring of functions is $(\mathcal{O}_{\mathfrak{IG}_{K^p}})^{\overline{K_{p,P}}}$. Similarly, for any compact open
\[
U_{K_{p,P}} \subseteq U_{P'} = \tbyt{1}{\mbb{Q}_p(1)}{}{1}
\]
we let $\mathfrak{IG}_{U_{K_{p,P}}K^p}$ be the flat formal scheme equal to the categorical quotient of $\mathfrak{IG}_{K^p}$ by $\overline{U_{K_{p,P}}}$. 
If $K_{p,P}$ is a compact open subgroup of $P'$ and if we set $U_{K_{p,P}} = K_{p, P} \cap U_{P'}$, then there is a short exact sequence 
\[
0 \rightarrow U_{K_{p,P}} \rightarrow K_{p,P} \rightarrow M_{K_{p,P}} \rightarrow 0.
\]
The map $\mathfrak{IG}_{U_{K_{p,P}}K^p}\rightarrow \mathfrak{IG}_{{K_{p,P}}K^p}$ is a pro-\'etale torsor under the group $\overline{M_{K_{p,P}}}$. 

\begin{remark} 
Let $\mbb{Q}_p^{\opn{cycl}}$ denote the $p$-adic completion of $\mbb{Q}_p(\mu_{p^{\infty}})$ with ring of integers $\mbb{Z}_p^{\opn{cycl}}$. If we base change to $\opn{Spa}(\Q_p^{\opn{cycl}}, \Z_p^{\opn{cycl}})$, then $P' \cong \underline{P(\Q_p)}$ (i.e. it is isomorphic to the constant \'{e}tale group scheme associated with $P(\mbb{Q}_p)$), and we may consider spaces $\mathfrak{IG}_{K_{p,P} K^p, \Z_p^{\opn{cycl}}}$ for any compact open $K_{p, P}$ of $P(\mbb{Q}_p)$. 
\end{remark}

Recall that the space of $p$-adic modular forms can be defined as (isotypic parts of) sections of the classical Katz moduli space parameterising elliptic curves (modulo prime-to-$p$ quasi-isogenies) equipped with a trivialisation of the connected part of its associated $p$-divisible group. Let $\overline{U_{P'}^{\opn{int}}}$ denote the schematic closure of $U_{P'}^{\opn{int}}$ in $J_{\opn{ord}}$ (which is equal to the unipotent part of $J_{\opn{ord}}^{\opn{int}}$). By \cite[Lemma 5.1.1]{howe2020unipotent}, the Igusa tower $\mathfrak{IG}_{K^p}$ is an fpqc $\overline{U_{P'}^{\opn{int}}}$-torsor over the above Katz moduli space, which explains the following definition. 

\begin{definition}
The space of $p$-adic modular forms\footnote{The Katz Igusa tower often considered in the literature is a $\mbb{Z}_p^{\times}$-torsor which parameterises isomorphisms $\mu_{p^{\infty}} \xrightarrow{\sim} E[p^\infty]^{\circ}$. The version of the Igusa tower we use has better functoriality properties with respect to Hecke operators.} of tame level $K^p$ is defined as
\[ 
\mathscr{M}^+ \defeq \mathrm{H}^0(\mathfrak{IG}_{U_{P'}^{\opn{int}} K^p}, \mathcal{O}_{{\mathfrak{IG}}_{U_{P'}^{\opn{int}}K^p}}).
\]
\end{definition}

The following theorem summarises some of its key structures.

\begin{thm} \label{PropertiesOfpadicmfsThm}
Let $C_{\opn{cont}}(\mbb{Z}_p, \mbb{Z}_p)$ denote the $\mbb{Z}_p$-algebra of continuous functions $\mbb{Z}_p \to \mbb{Z}_p$.
\begin{enumerate} 
\item The space $\mathscr{M}^+$ carries a $\mbb{Z}_p$-algebra action of $C_{\opn{cont}}(\Z_p, \Z_p)$, an action of $T(\mbb{Z}_p)$ and the Hecke operators
\[
U_p \defeq [\overline{U_{P'}^{\opn{int}}} \cdot \mathrm{diag}(1, p^{-1}) \cdot \overline{U_{P'}^{\opn{int}}}], \quad \varphi \defeq [\overline{U_{P'}^{\opn{int}}} \cdot \mathrm{diag}(p^{-1}, 1) \cdot \overline{U_{P'}^{\opn{int}}}], \quad S_p \defeq [\opn{diag}(p^{-1}, p^{-1}) \cdot \overline{U_{P'}^{\opn{int}}}],
\]
and an action of the prime-to-$p$ Hecke algebra $\mathcal{H}_{K^p}$. The action of $T(\Z_p)$ commutes with that of $U_p$, $\varphi$, $S_p$ and we have the following relations 
\[
U_p \circ \varphi = p S_p, \quad S_p \circ U_p = U_p \circ S_p, \quad \varphi \circ S_p = S_p \circ \varphi.
\] 
\item Let $c \colon \opn{Spf} \Z_p(\zeta_N)[[q^{1/N}]] \rightarrow \mathfrak{IG}_{U_{P'}^{\opn{int}}K^p}$ be a cusp, where $N$ is a sufficiently large integer with $(p, N) = 1$, which gives rise to a $q$-expansion map $c \colon \mathscr{M}^+ \to \Z_p(\zeta_N)[[q^{1/N}]]$. Then for any $\phi \in C_{\opn{cont}}(\mbb{Z}_p, \mbb{Z}_p)$ and $f \in \mathscr{M}^+$ one has
\[
c(\phi \cdot f) = \sum_{k \in \frac{1}{N}\mbb{Z}} \phi(k) a_k q^k, \quad \quad \text{ where } \; c(f) = \sum_{k \in \frac{1}{N}\mbb{Z}} a_k q^k .
\]
\item Let $\phi \in C_{\opn{cont}}(\mbb{Z}_p, \mbb{Z}_p)$. Then for any $t = \opn{diag}(t_1, t_2) \in T(\Z_p)$, we have $t ~ \circ ~ \phi = \phi( t_2^{-1} \cdot - \cdot t_1) ~ \circ ~  t$ as endomorphisms of $\mathscr{M}^+$. Moreover, we have 
\[
U_p \, \circ \, \phi = \phi(p \cdot -) \, \circ \, U_p, \quad \varphi \, \circ \, \phi(p \cdot -) = \phi \, \circ \, \varphi, \quad S_p \, \circ \, \phi = \phi \, \circ \, S_p
\]
as endomorphisms of $\mathscr{M}^+$.
\end{enumerate}
\end{thm}
\begin{proof}
The action of the prime-to-$p$ Hecke operators $\mathcal{H}_{K^p}$ is clear. For the other operators, we note that the action of the group $J_{\opn{ord}}$ on $\mathfrak{IG}_{K^p}$ induces an action of 
$T(\Z_p)$, $[\overline{U_{P'}^{\opn{int}}} \cdot \mathrm{diag}(1, p^{-1}) \cdot \overline{U_{P'}^{\opn{int}}}]$,  $[\overline{U_{P'}^{\opn{int}}} \cdot \mathrm{diag}(p^{-1}, 1) \cdot \overline{U_{P'}^{\opn{int}}}]$, $[\opn{diag}(p^{-1}, p^{-1}) \cdot \overline{U_{P'}^{\opn{int}}}]$ on $\mathfrak{IG}_{U_{P'}^{\opn{int}}K^p}$, as well as an action of $\tilde{\mu_{p^{\infty}}}/T_p(\mu_{p^{\infty}}) = \mu_{p^\infty} = \widehat{\mbb{G}}_m$. By Fourier theory, this action amounts to an action of $C_{\opn{cont}}(\mbb{Z}_p, \mbb{Z}_p)$. See \cite[\S 7.1]{howe2020unipotent}. The relations are easily determined. 
\end{proof}

Let $T^+ \subset T(\mbb{Q}_p)$ denote the submonoid of elements $\opn{diag}(t_1, t_2)$ such that $v_p(t_1) \geq v_p(t_2)$. Then we see that $T(\Z_p), U_p$ and $S_p$ generate an action of $T^+$. We let $\Theta \colon \mathscr{M}^+ \to \mathscr{M}^+$ be the operator corresponding to the action of the identity function $\mathrm{Id}_{\Z_p} \in C_{\opn{cont}}(\mbb{Z}_p, \mbb{Z}_p)$. 

\begin{remark} 
Let $U_{K_{p,P}}$ be a compact open subgroup of $U_{P'}$. Since the unipotent part of $J_{\opn{ord}}$ is abelian, the action of $J_{\opn{ord}}$ on $\mathfrak{IG}_{K^p}$ induces an action of $\tilde{\mu_{p^{\infty}}}/\overline{U_{K_{p, P}}}$ on $\mathfrak{IG}_{U_{K_{p, P}}K^p}$. Note that $\tilde{\mu_{p^{\infty}}}/\overline{U_{K_{p, P}}}$ is a formal torus with cocharacter group $U_{K_{p, P}}(-1)$, and $p$-adic Fourier theory identifies measures on $U_{K_{p, P}}(-1)^{\vee}$ with functions on $U_{K_{p, P}}(-1) \otimes \widehat{\mbb{G}}_m$. The action of $U_{K_{p, P}}(-1) \otimes \widehat{\mathbb{G}}_m$ thus gives an action of $C_{\opn{cont}}(U_{K_{p, P}}(-1)^{\vee}, \Z_p)$ on sections of $\mathfrak{IG}_{U_{K_{p, P}} K^p}$. As in the proof of Theorem \ref{PropertiesOfpadicmfsThm}, we identify $U_{P'}^{\opn{int}}(-1)^{\vee}$ with (the constant group scheme) $\Z_p$.
\end{remark}

Let $\mathscr{M}^+_{U_{K_{p,P}}} =  \mathrm{H}^0\big(\mathfrak{IG}_{U_{K_{p,P}} K^p}, \mathcal{O}_{\mathfrak{IG}_{U_{K_{p,P}}K^p}}\big)$. Using the action of $\mathrm{diag}(p^n,1) \in J_{\opn{ord}}$ by conjugation, we can identify $\mathfrak{IG}_{U_{K_{p,P}} K^p}$ with $\mathfrak{IG}_{U_{P'}^{\opn{int}} K^p}$ and $\mathscr{M}^+$ with $\mathscr{M}^+_{U_{K_{p,P}}}$. This conjugation action will transport the $U_p$, $\varphi$, $S_p$-actions and will conjugate the action of $C_{\opn{cont}}(\Z_p, \Z_p) =  C_{\opn{cont}}(U_{P'}^{\opn{int}}(-1)^\vee, \Z_p)$ with $C_{\opn{cont}}(U_{K_{p,P}}(-1)^\vee, \Z_p)$. Set $\mathscr{M}_{{U_{K_{p,P}}}} = \mathscr{M}_{{U_{K_{p,P}}}}^+[1/p]$. We let 
\[
\Theta_{U_{K_{p,P}}} \colon \mathscr{M}_{U_{K_{p,P}}} \rightarrow \mathscr{M}_{U_{K_{p,P}}}
\]
denote the map induced by the action of the identity function of $\Q_p$ in $C_{\opn{cont}}(U_{K_{p, P}}(-1)^\vee, \Z_p) \otimes \Q_p$. Note that conjugation by  $\mathrm{diag}(p^n,1)$ sends $\Theta_{U_{K_{p,P}}}$ to $p^n \Theta$. 

\subsection{Relation to classical forms}

To compare $p$-adic modular forms with classical modular forms (in this level of generality), we work over $\Spa(\Q_p^{\opn{cycl}}, \Z_p^{\opn{cycl}})$. We fix an isomorphism $\Z_p \cong \Z_p(1)$. Therefore, we can (and do) identify $P'$ and $\underline{P(\Q_p)}$.  We let $\mathcal{X}_{K^p, \Q_p^{\opn{cycl}}}$ be the perfectoid modular curve and we let $\mathcal{IG}_{K^p, \Q_p^{\opn{cycl}}} = \mathfrak{IG}_{K^p} \times \Spa (\Q_p^{\opn{cycl}}, \Z_p^{\opn{cycl}})$.  We have a natural $P(\Q_p)$-equivariant map $\mathcal{IG}_{K^p, \Q_p^{\opn{cycl}}}  \rightarrow \mathcal{X}_{K^p, \Q_p^{\opn{cycl}}}$. 

Let $K_p \subseteq G(\Q_p)$ be a compact open subgroup and let $g \in G(\Q_p)$. The map
\[
\mathcal{IG}_{K^p, \Q_p^{\opn{cycl}}} \xrightarrow{\cdot g} \mathcal{X}_{K^p, \Q_p^{\opn{cycl}}} \rightarrow \mathcal{X}_{K_p K^p, \Q_p^{\opn{cycl}}}
\]
factors through an open immersion  $\mathcal{IG}_{K_{p,P}K^p, \Q_p^{\opn{cycl}}} \rightarrow \mathcal{X}_{K_p K^p, \Q_p^{\opn{cycl}}}$ where $K_{p,P} = g K_{p} g^{-1} \cap P(\Q_p)$. One can think of this open as a connected component of the ordinary locus in $\mathcal{X}_{K_p K^p, \Q_p^{\opn{cycl}}}$.

\begin{remark} 
In general this open immersion $\mathcal{IG}_{K_{p,P}K^p, \Q_p^{\opn{cycl}}} \rightarrow \mathcal{X}_{K_p K^p, \Q_p^{\opn{cycl}}}$ is not defined over $\Q_p$ because the way we construct the level structure is different on both sides. On the other hand, for hyperspecial or Iwahori level and $g = 1$, it is defined over $\mbb{Q}_p$.
\end{remark}

Let $P_{\opn{dR}, K} \to X_{K}$ denote the right $\overline{P}$-torsor parameterising frames of $\mathcal{H}_E$ respecting the Hodge filtration (see \S \ref{ApplicationToModularCurvesAppendix}). We now consider the $M_{K_{p,P}}$-torsor $\mathcal{IG}_{U_{K_{p,P}}K^p, \Q_p^{\opn{cycl}}} \rightarrow \mathcal{IG}_{K_{p,P}K^p, \Q_p^{\opn{cycl}}}$. For ease of notation, set $K = K^pK_p$. Then we have a commutative diagram

\begin{eqnarray*}
\xymatrix{ \mathcal{IG}_{U_{K_{p,P}}K^p, \Q_p^{\opn{cycl}}} \ar[r] \ar[d]& P^{\opn{an}}_{\opn{dR}, K, \Q_p^{\opn{cycl}}}\ar[d] \\
\mathcal{IG}_{K_{p,P}K^p, \Q_p^{\opn{cycl}}} \ar[r] & \mathcal{X}_{K, \Q_p^{\opn{cycl}}}}
\end{eqnarray*}
where $P^{\opn{an}}_{\opn{dR}, K, \Q_p^{\opn{cycl}}}$ denotes the analytification of $P_{\opn{dR}, K}$. The top map is induced by the Hodge-Tate map and the unit root splitting. It is $M_{K_{p,P}}$-equivariant via the natural map $M_{K_{p,P}} \rightarrow T^{\opn{an}} \subset \overline{P}^{\opn{an}}$. 

The following proposition describes the relation between nearly holomorphic modular forms and $p$-adic modular forms.

\begin{proposition} 
We have a natural map 
\begin{equation} \label{PdRtoMdRNoweight}
\mathrm{H}^0(P_{\opn{dR},K}, \mathcal{O}_{P_{\opn{dR},K}}) \otimes_{\qp} \Q_p^{\opn{cycl}} \rightarrow \mathscr{M}_{U_{K_{p,P}}} \hatot_{\mbb{Q}_p} \Q_p^{\opn{cycl}}
\end{equation}
which is equivariant for the action of the prime-to-$p$ Hecke algebra $\mathcal{H}_{K^p}$, as well as
for the action of $M_{K_{p,P}}$ via the map $M_{K_{p,P}} \rightarrow T \subset \overline{P}$. Hence, for any $\kappa \in X^*(T)$, we obtain a map:
\begin{equation} \label{HktopadicEqn}
\mathrm{H}^0(X_K, \mathcal{H}_\kappa) \otimes_{\qp} \Q_p^{\opn{cycl}} \rightarrow \mathrm{Hom}_{M_{K_{p,P}}}( -w_0\kappa, \mathscr{M}_{U_{K_{p,P}}}) \hatot_{\Q_p} \Q_p^{\opn{cycl}}\{w_0\kappa\}
\end{equation}
where $\mathcal{H}_{\kappa}=\opn{Hom}_{T}(-w_0\kappa, \pi_*\mathcal{O}_{P_{\opn{dR}, K}} )$ (with $\pi \colon P_{\opn{dR}, K} \to X_K$ the structural map), and $\Q_p^{\opn{cycl}}\{w_0\kappa\}$ means that $T^+$ acts on $\mbb{Q}_p^{\opn{cycl}}$ through the character $w_0\kappa$.

Assume that $gK_pg^{-1}$ has an Iwahori decomposition: 
\[
g K_p g^{-1} = U_{gK_p g^{-1}} \cdot T_{gK_pg^{-1}} \cdot \overline{U}_{gK_p g^{-1}}.
\]
Let $\mathcal{H}_{K_p}^+$ be the subalgebra of the Hecke algebra $\mathcal{H}_{K_p}$ generated by $[K_p \cdot g^{-1} t  g \cdot K_p]$ with $t \in T^+$. Then the map (\ref{HktopadicEqn}) is equivariant for the morphism sending $[K_p \cdot g^{-1} t g \cdot K_p]$ to $t \in T^+$.
\end{proposition}
\begin{proof}
The existence of the first map just follows from the fact that $\mathcal{IG}_{U_{K_{p, P}} K^p} \to P_{\opn{dR}, K}$ is equivariant for $M_{K_{p, P}}$ and functorial in $K^p$ (the map is induced by restriction). Therefore the maps \eqref{PdRtoMdRNoweight} and \eqref{HktopadicEqn} are well-defined.

Let $t \in T^+$. Set $K' = (g^{-1}tg) K (g^{-1}tg)^{-1} \cap K$ and $K'' = (g^{-1}tg)^{-1}K' (g^{-1}tg)$. set $U_{K_{p, P}} = g K_p g^{-1} \cap U_{P'}$ and $U_{K_{p, P}'} = g K_p' g^{-1} \cap U_{P'}$. Then we have a correspondence 
\[
X_{K} \xleftarrow{p_1} X_{K'} \xrightarrow{p_2} X_K
\]
where $p_1$ is the natural map, and $p_2$ is the composition $X_{K'} \xrightarrow{\cdot (g^{-1}tg)} X_{K''} \to X_{K}$. We also have a diagram of correspondences (over $\mbb{Q}_p^{\opn{cycl}}$)
\[
\begin{tikzcd}
                  & {p_1^{-1}P_{\opn{dR}, K}} \arrow[ld, "p_1"'] \arrow[rr, "\lambda"] &                                                                                              & {p_2^{-1}P_{\opn{dR}, K}} \arrow[rd, "p_2"] &                   \\
{P_{\opn{dR}, K}} &                                                                    & {\mathcal{IG}_{U_{K'_{p, P}}K^p}} \arrow[lu] \arrow[ru] \arrow[ld, "q_1"'] \arrow[rd, "q_2"] &                                             & {P_{\opn{dR}, K}} \\
                  & {\mathcal{IG}_{U_{K_{p, P}}K^p}} \arrow[lu]                        &                                                                                              & {\mathcal{IG}_{U_{K_{p, P}}K^p}} \arrow[ru] &                  
\end{tikzcd}
\]
where $q_1$ is the natural map induced from the inclusion $U_{K'_{p, P}} \subset U_{K_{p, P}}$, $q_2$ is induced from right multiplication by $t$ via the inclusion $t^{-1}U_{K'_{p, P}} t \subset U_{K_{p, P}}$, and the map $\lambda$ is the action of $t$ (via the torsor structure on $P_{\opn{dR}}$) composed with the morphism $p_1^{-1}P_{\opn{dR}, K} \to p_2^{-1}P_{\opn{dR}, K}$ arising from the $G(\mbb{Q}_p)$-equivariant structure on $P_{\opn{dR}}$. The claim now follows from the fact that the left-hand square is Cartesian. Indeed, one has 
\[
[K : K'] = [gK_pg^{-1}: gK_p'g^{-1}] = [U_{gK_pg^{-1}} : tU_{gK_pg^{-1}}t^{-1} \cap U_{gK_pg^{-1}}] = [U_{K_{p, P}} : U_{K'_{p, P}}]
\]
using the Iwahori decomposition for $gK_pg^{-1}$.
\end{proof}

By the results in \S \ref{ApplicationToModularCurvesAppendix}, one has that $\mathcal{H}_{\kappa} = \opn{VB}_K^{\opn{can}}(M_{\kappa}^{\vee})$ is the quasi-coherent sheaf associated with the dual Verma module of lowest weight $w_0\kappa$. In particular $\mathcal{H}_{\kappa}$ is naturally a $\mathcal{D}_{X_K}$-module on $X_K$. By composing the induced connection $\mathcal{H}_{\kappa} \to \mathcal{H}_{\kappa} \otimes \Omega^1_{X_K}(\opn{log}D)$ with the map $\mathcal{H}_{\kappa} \otimes \Omega^1_{X_K}(\opn{log}D) \to \mathcal{H}_{\kappa+2\rho}$ induced from the Kodaira--Spencer isomorphism, we obtain a $\mbb{Q}_p$-linear derivation $\nabla$ on $\opn{H}^0\left(P_{\opn{dR}, K}, \mathcal{O}_{P_{\opn{dR}, K}} \right) = \bigoplus_{\kappa \in X^*(T)} \opn{H}^0\left(X_K, \mathcal{H}_{\kappa} \right)$ (c.f. Proposition \ref{OPdrKPropAppendix}(2)). 

\begin{proposition}
One has a commutative diagram: 
\begin{eqnarray*}
\xymatrix{ \mathrm{H}^0(P_{\opn{dR},K}, \mathcal{O}_{P_{\opn{dR},K}}) \otimes \Q_p^{\opn{cycl}}    \ar[d] \ar[r]^{\nabla} & \mathrm{H}^0(P_{\opn{dR},K}, \mathcal{O}_{P_{\opn{dR},K}}) \otimes \Q_p^{\opn{cycl}} \ar[d] \\
\mathscr{M}_{U_{K_{p,P}}} \otimes \Q_p^{\opn{cycl}}  \ar[r]^{\Theta_{U_{K_{p,P}}}} & \mathscr{M}_{U_{K_{p,P}}} \otimes \Q_p^{\opn{cycl}}}
\end{eqnarray*}
\end{proposition}
\begin{proof}
This is essentially a reformulation of \cite[Theorem 5.3.1]{howe2020unipotent}. Alternatively, one can check this on $q$-expansions.
\end{proof}

\subsection{The main theorem}  

We can now introduce the space of nearly overconvergent modular forms and state the main theorem. Let $K_p \subset G(\mbb{Q}_p)$ be a compact open subgroup, $g \in G(\mbb{Q}_p)$ and $K_{p, P} = g K_p g^{-1} \cap P(\mbb{Q}_p)$ as above. Recall that we have a morphism $\mathcal{IG}_{U_{K_{p, P}}K^p, \mbb{Q}_p^{\opn{cycl}}} \to P_{\opn{dR}, K, \mbb{Q}_p^{\opn{cycl}}}^{\opn{an}}$.

\begin{definition}
We define the space of \emph{nearly overconvergent modular forms} to be
\[
\mathscr{N}^{\dagger}_{U_{K_{p, P}}} \defeq \opn{colim}_U \opn{H}^0\left(U, \mathcal{O}_U \right)
\]
where the colimit is over all open subsets of $P_{\opn{dR}, K, \mbb{Q}_p^{\opn{cycl}}}^{\opn{an}}$ which contain the closure of $\mathcal{IG}_{U_{K_{p, P}}K^p, \mbb{Q}_p^{\opn{cycl}}}$ (via the morphism above) with transition maps given by restriction. Similarly, we define the space of \emph{overconvergent modular forms} to be \[ \mathscr{M}^{\dagger}_{U_{K_{p, P}}} \defeq \opn{colim}_V \opn{H}^0\left(V, \mathcal{O}_V\right), \] where the colimit is over all open subsets of $M_{\opn{dR}, K, \mbb{Q}_p^{\opn{cycl}}}^{\opn{an}} = P_{\opn{dR}, K, \mbb{Q}_p^{\opn{cycl}}}^{\opn{an}} \times^{\overline{P}^{\opn{an}}} T^{\opn{an}}$ containing the closure of $\mathcal{IG}_{U_{K_{p, P}}K^p, \mbb{Q}_p^{\opn{cycl}}}$. These are $\opn{LB}$-spaces of compact type, and one has a natural map $\mathscr{N}^{\dagger}_{U_{K_{p, P}}} \to \mathscr{M}_{U_{K_{p, P}}}$ induced from restriction.
\end{definition}

\begin{thm} \label{SheafVersionMainThm} 
The space $\mathscr{N}^{\dagger}_{U_{K_{p, P}}}$ comes equipped with actions of $U_p$, $\varphi$, $S_p$, $M_{K_{p, P}}$ and of locally analytic functions $C^{\opn{la}}(U_{K_{p, P}}(-1)^\vee, \Q_p) \subset C_{\opn{cont}}(U_{K_{p, P}}(-1)^\vee, \Q_p)$ which are compatible with the actions on $p$-adic modular forms via the map $\mathscr{N}^{\dagger}_{U_{K_{p, P}}} \to \mathscr{M}_{U_{K_{p, P}}}$. Furthermore, it also carries an action of the lower triangular nilpotent Lie algebra $\overline{\mathfrak{n}} \subset \mathfrak{gl}_2$ obtained by differentiating the torsor structure on $P_{\opn{dR}, K, \mbb{Q}_p^{\opn{cycl}}}^{\opn{an}}$, providing an ascending filtration $\opn{Fil}_r \mathscr{N}^{\dagger}_{U_{K_{p, P}}} \defeq \mathscr{N}^{\dagger}_{U_{K_{p, P}}}[\overline{\mathfrak{n}}^{r+1}]$ (the elements killed by $\overline{\mathfrak{n}}^{r+1}$).

The space $\mathscr{N}^{\dagger}_{U_{K_{p, P}}}$ satisfies the following additional properties:
\begin{enumerate}
\item The filtration is stable under $U_p$, $S_p$, $\varphi$, $M_{K_{p, P}}$ and $\mathrm{Fil}_0 \mathscr{N}^{\dagger}_{U_{K_{p, P}}}$ is the space of overconvergent modular forms. Furthermore, we have $U_p \circ \varphi = p S_p$ and $S_p$ commutes with both $U_p$ and $\varphi$.
\item Let $R_0^+$ be an admissible $\mbb{Z}_p$-algebra, $R_0 = R_0^+[1/p]$ and set $\mathcal{W} = \opn{Spa}(R_0, R_0^+)$. Suppose that $\kappa \colon M_{K_{p, P}} \rightarrow R_0^\times$ is a locally analytic character. Then for any $x \in \mathcal{W}(\overline{\mbb{Q}}_p)$ and $h \in \mbb{Q}$, there exists a quasi-compact open affinoid neighbourhood $\Omega = \opn{Spa}(R, R^+) \subset \mathcal{W}$ containing $x$ such that $\mathscr{N}^{\dagger}_{U_{K_{p, P}}, \kappa} \defeq \mathrm{Hom}_{M_{K_{p, P}}}(-w_0\kappa, \mathscr{N}^{\dagger}_{U_{K_{p, P}}}\hatot R)$ admits a slope $\leq h$ decomposition with respect to the operator $U_p$. Furthermore, there exists an integer $r \geq 0$ such that
\[
\mathscr{N}^{\dagger, \leq h}_{U_{K_{p, P}}, \kappa} = \left( \opn{Fil}_r \mathscr{N}^{\dagger}_{U_{K_{p, P}}, \kappa} \right)^{\leq h} .
\]
\item The analogous relations as in Theorem \ref{PropertiesOfpadicmfsThm}(3) hold for the actions of $U_p$, $\varphi$, $S_p$, $M_{K_{p, P}}$ and $C^{\opn{la}}(U_{K_{p, P}}(-1)^\vee, \Q_p)$ on $\mathscr{N}^{\dagger}_{U_{K_{p, P}}}$. In addition to this, there is a factorisation
\[
\begin{tikzcd}
H^0(P_{\opn{dR}, K}, \mathcal{O}_{P_{\opn{dR}}, K}) \otimes \Q_p^{\opn{cycl}} \arrow[dr] \arrow[rr, "(\ref{PdRtoMdRNoweight})"] & & \mathscr{M}_{U_{K_{p,P}}} \hatot_{\mbb{Q}_p} \Q_p^{\opn{cycl}} \\
& \mathscr{N}^\dagger_{U_{K_p, P}} \arrow[ur] &             
\end{tikzcd}
\]
which is compatible with filtrations. The operator $\nabla$ intertwines with the action of the identity function in $C^{\opn{la}}(U_{K_{p, P}}(-1)^\vee, \Q_p)$. 
\end{enumerate}
\end{thm}

We will prove this theorem when $K_p = \opn{GL}_2(\mbb{Z}_p)$ and $g=1$ in \S \ref{GMinterpolationSection} and \S \ref{AdditionalStructuresSection} using results from \S \ref{SectionNOCMFs} and \S \ref{SectionPADICINTCONT}. As explained in the following section, this is sufficient for proving the theorem for general $K_p$ and $g$.

\subsection{Reduction to hyperspecial level} \label{RedToHypLevelSSec}

We now explain why it suffices to establish Theorem \ref{SheafVersionMainThm} when $K_p = \opn{GL}_2(\mbb{Z}_p)$ and $g=1$. Recall that we denoted $U_{P'} = \tbyt{1}{\mbb{Q}_p(1)}{}{1}$.

\begin{lemma}
Let $K_p, K_p' \subset G(\mbb{Q}_p)$ be compact open subgroups, and let $g, g' \in G(\mbb{Q}_p)$.
\begin{enumerate}
    \item Suppose that $U_{K_{p, P}} := g K_p g^{-1} \cap U_{P'} = g' K_p' (g')^{-1} \cap U_{P'} =: U_{K'_{p, P}}$ and $(g')^{-1}g K_p' g^{-1}g' \subset K_p$. Then the natural map $P_{\opn{dR}, K_p' K^p} \to P_{\opn{dR}, K_p K^p}$ induced from right multiplication by $g^{-1}g'$ (via the $G(\mbb{Q}_p)$-equivariant structure on $P_{\opn{dR}}$) induces an isomorphism
    \[
    \mathscr{N}^{\dagger}_{U_{K'_{p, P}}} \xrightarrow{\sim} \mathscr{N}^{\dagger}_{U_{K_{p, P}}}
    \]
    respecting filtrations and the $M_{K'_{p, P}}$ and $M_{K_{p, P}}$-actions. This implies that $\mathscr{N}^{\dagger}_{U_{K_{p, P}}}$ is intrinsic to the compact open subgroup $U_{K_{p, P}}$.
    \item Let $n \in \mbb{Z}$ and $t = \opn{diag}(p^n, 1) \in T(\mbb{Q}_p)$. Suppose that $g = g' = 1$, and $t^{-1}K_p't = K_p$ (which implies $t^{-1}U_{K'_{p, P}} t = U_{K_{p, P}}$). Then the isomorphism $P_{\opn{dR}, K_p'K^p} \to P_{\opn{dR}, K_p K^p}$ obtained as the composition of right-translation by $t$ (via the $G(\mbb{Q}_p)$-equivariant structure) and right-translation by $t$ (via the torsor structure) induces an isomorphism
    \[
    \mathscr{N}^{\dagger}_{U_{K'_{p, P}}} \xrightarrow{\cdot t} \mathscr{N}^{\dagger}_{U_{K_{p, P}}}
    \]
    respecting filtrations and the $M_{K'_{p, P}}$ and $M_{K_{p, P}}$-actions.
\end{enumerate}
Therefore, since any compact open subgroup of $U_{P'}$ is conjugate to $U_{P'}^{\opn{int}}$ via some $\opn{diag}(p^n, 1)$, by transporting structure it is enough to prove Theorem \ref{SheafVersionMainThm} when $U_{K_{p, P}} = U_{P'}^{\opn{int}}$ (e.g. when $K_p = \opn{GL}_2(\mbb{Z}_p)$ and $g=1$).
\end{lemma}
\begin{proof}
For the first part, let $U = U_{K_{p, P}} = U_{K'_{p, P}}$. We have a commutative diagram 
\[
\begin{tikzcd}
                              & {P_{\opn{dR}, K_p'K^p}} \arrow[d, "f"] \\
\mathcal{IG}_{U K^p} \arrow[ru] \arrow[r] & {P_{\opn{dR}, K_p K^p}}               
\end{tikzcd}
\]
where $f$ denotes the map induced by right multiplication by $g^{-1}g'$, the diagonal map is induced from right-translation by $g'$ and the horizontal map is induced from right-translation by $g$. The map $f$ is finite \'{e}tale away from the cusps not lying in $\mathcal{IG}_{UK^p}$, hence it induces a map $\mathscr{N}^{\dagger}_{U_{K'_{p, P}}} \to \mathscr{N}^{\dagger}_{U_{K_{p, P}}}$. 

Furthermore, we can find strict neighbourhoods $V_1$ (resp. $V_2$) of $\mathcal{IG}_{U K^p}$ in $P^{\opn{an}}_{\opn{dR}, K_p K^p}$ (resp. $P^{\opn{an}}_{\opn{dR}, K_p'K^p}$) such that $f \colon V_2 \to V_1$ is an isomorphism (cf. \cite[Proposition 2.2.1]{KisinLai} -- the key point is that we can find a strict neighbourhood $V$ of $f^{-1}(\mathcal{IG}_{UK^p}) \subset P^{\opn{an}}_{\opn{dR}, K_p'K^p}$ such that the connected components of $f^{-1}(\mathcal{IG}_{UK^p})$ and $V$ are in bijection with one another). This proves part (1). 

The second part follows immediately from the fact that the isomorphism $\mathcal{IG}_{U_{K'_{p, P}} K^p} \to \mathcal{IG}_{U_{K_{p, P}} K^p}$, induced from right-translation by $t$ intertwines with the isomorphism $P_{\opn{dR}, K_p'K^p} \to P_{\opn{dR}, K_p K^p}$ in the statement of the lemma.
\end{proof}

When working at hyperspecial level, we will omit any subscript which includes the level subgroup (e.g. $\mathscr{N}^{\dagger}$, $\mathscr{M}^{\dagger}$, $\mathscr{M}$ denote the spaces of (nearly) overconvergent and $p$-adic modular forms etc.). Furthermore, in this setting the morphism from the Igusa tower to $P_{\opn{dR}}$ is defined over $\mbb{Q}_p$, so it is not necessary to base-change to $\mbb{Q}_p^{\opn{cycl}}$. In fact, to consider families of nearly overconvergent modular forms, we will work over $\opn{Spa}(R, R^+)$ for some admissible $\mbb{Z}_p$-algebra $R^+$.

\section{A system of neighbourhoods of the Igusa tower} \label{SectionNOCMFs}

The purpose of this section is to define and study an explicit system of strict neighbourhoods of the Igusa tower inside $P_{\opn{dR}}^{\opn{an}}$, which will be crucial for our constructions. The basic idea behind this construction is that, after multiplying by suitable powers of the Hasse invariant, one can, modulo large powers of $p$, overconverge elements defined on the ordinary locus.

\subsection{The setting} \label{TheSettingSubSec}

Let $R^+$ be an admissible $\mbb{Z}_p$-algebra and set $R = R^+[1/p]$. Fix a neat compact open subgroup $K^p \subset G(\mbb{A}_f^p)$ and let $\mathcal{X}/\opn{Spa}(R, R^+)$ be the compact modular curve of level $\opn{GL}_2(\mbb{Z}_p)K^p$ over $\opn{Spa}(R, R^+)$. We have (see \cite{Katz}) a system of neighbourhoods of the ordinary locus
\[
\mathcal{X}_{\opn{ord}} \to \ldots \to \mathcal{X}_r \to \ldots \to \mathcal{X}_1 \to \mathcal{X}
\]
where $\mathcal{X}_r$ denotes the (unique) quasi-compact open whose rank one points $x$ satisfy $|h|_x \geq p^{-1/p^{r+1}}$ for any local lift $h$ of the Hasse invariant. 

For any integer $n \geq 1$, let $\mathcal{IG}_n \defeq \mathcal{IG}_{P'_n K^p}$ be the Igusa tower of level
\[
P'_n = \tbyt{1 + p^n \mbb{Z}_p}{\mbb{Z}_p(1)}{}{1+p^n \mbb{Z}_p}
\]
which is a finite $T(\mbb{Z}/p^n \mbb{Z})$-torsor over $\mathcal{X}_{\opn{ord}}$. We also use the notation $\mathcal{IG}_{\infty} = \mathcal{IG}_{U_{P'}(\mbb{Z}_p)K^p}$. We define $\mathcal{H}_{\mathcal{E}}$ to be the canonical extension of the first relative de Rham homology of the universal elliptic curve over the good reduction locus in $\mathcal{X}$. The formal model $\mathfrak{X}$ defines a $\mathcal{O}_{\mathcal{X}}^+$-lattice $\mathcal{H}_{\mathcal{E}}^+ \subset \mathcal{H}_{\mathcal{E}}$. We similarly define $\omega_{\mathcal{E}}^+ \subset \omega_{\mathcal{E}}$ and $\omega_{\mathcal{E}^D}^+ \subset \omega_{\mathcal{E}^D}$ for the Hodge bundle of the universal and dual universal generalised elliptic curve and the corresponding $\mathcal{O}_{\mathcal{X}}^+$-lattices.

If $\opn{Spa}(A, A^+) \subset \mathcal{X}$ is a quasi-compact open affinoid subspace, then we let $\opn{Spa}(A_r, A_r^+)$ (resp. $\opn{Spa}(A_{\opn{ord}}, A_{\opn{ord}}^+)$, resp. $\opn{Spa}(A_{\opn{ord}, n}, A_{\opn{ord}, n}^+)$) denote the pullback to $\mathcal{X}_r$ (resp. $\mathcal{X}_{\opn{ord}}$, resp. $\mathcal{IG}_n$). We warn the reader that we will often implicitly assume that $\opn{Spa}(A_{\opn{ord}}, A_{\opn{ord}}^+) \neq \varnothing$, otherwise most of what we say will be vacuous. Note that if $\omega_{\mathcal{E}}^+$ is trivial over $\opn{Spa}(A, A^+)$ and $h \in A^+$ is a fixed local lift of the Hasse invariant, then $A_{\opn{ord}}^+ = A^+\langle 1/h \rangle$. If $\{e, f\}$ is a fixed basis of $\mathcal{H}_{\mathcal{E}}^+$ over $\opn{Spa}(A, A^+)$ respecting the Hodge filtration (i.e. $e \in \omega_{\mathcal{E}}^+$), then we have 
\[
P_{\opn{dR}}^{\opn{an}} \times_{\mathcal{X}} \opn{Spa}(A, A^+) \cong \overline{P}^{\opn{an}}_{A} = \tbyt{(\mbb{G}_m^{\opn{an}})_{A}}{}{(\mbb{G}_a^{\opn{an}})_{A}}{(\mbb{G}_m^{\opn{an}})_{A}}
\]
and the basis $\{e, f\}$ gives rise to coordinates $T_1, T_2, U \in \mathcal{O}^+(P_{\opn{dR}}^{\opn{an}} \times_{\mathcal{X}} \opn{Spa}(A, A^+))$ such that the universal trivialisation of $\mathcal{H}_{\mathcal{E}}^+$ over $P_{\opn{dR}}^{\opn{an}} \times_{\mathcal{X}} \opn{Spa}(A, A^+)$ takes $e$ (resp. $f$) to $(0, T_1)$ (resp. $(T_2, U)$). Equivalently, if $\{ e_{\opn{univ}}, f_{\opn{univ}} \}$ denotes the universal basis of $\mathcal{H}^+_{\mathcal{E}}$ over $P_{\opn{dR}}^{\opn{an}} \times_{\mathcal{X}} \opn{Spa}(A, A^+)$ (with $e_{\opn{univ}} \in \omega_{\mathcal{E}}^+$) then we have $e = T_1 e_{\opn{univ}}$ and $f = U e_{\opn{univ}} + T_2 f_{\opn{univ}}$.

\begin{notation}
For any $1 \leq n \leq \infty$, we will let $\mathscr{G}_n$ denote the Galois group $\Gal(\mathcal{IG}_{n}/\mathcal{X}_{\opn{ord}}) \cong T\left(\mbb{Z}/p^n \mbb{Z} \right)$ ($\cong T(\mbb{Z}_p)$ if $n=\infty$). 
\end{notation}

For the rest of this section, we fix an open affinoid $\opn{Spa}(A, A^+) \subset \mathcal{X}$ over which $\omega_{\mathcal{E}}^+$, $\mathcal{H}_{\mathcal{E}}^+$ and $\omega_{\mathcal{E}^D}^+$ are trivial, and fix a basis $\{ e, f \}$ of $\mathcal{H}_{\mathcal{E}}^+$ respecting the Hodge filtration
\[
0 \to \omega_{\mathcal{E}^D}^+ \to \mathcal{H}_{\mathcal{E}}^+ \to (\omega_{\mathcal{E}}^+)^{\vee} \to 0 .
\]
Recall that over $\mathcal{IG}_n$, one has universal trivialisations 
\[
\psi_1 \colon \left(\mathcal{E}^D[p^{n}]^{\circ}\right)^D \xrightarrow{\sim} \mbb{Z}/p^n\mbb{Z}, \quad \psi_2 \colon \mathcal{E}[p^{n}]^{\circ} \xrightarrow{\sim} \mu_{p^{n}}
\]
which induce trivialisations
\[
\phi_1 \colon \mathcal{O}^+_{\mathcal{IG}_n}/p^n \xrightarrow{\sim} (\omega_{\mathcal{E}^D}^+/p^n)|_{\mathcal{IG}_{n}}, \quad  \phi_2 \colon \mathcal{O}^+_{\mathcal{IG}_n}/p^n \xrightarrow{\sim} (\omega_{\mathcal{E}}^+/p^n)^{\vee}|_{\mathcal{IG}_n}.
\]
More precisely, $\phi_1(1)$ is the vector $\opn{dlog} \circ \; \psi_1^{-1}(1)$ and $\phi_2(1)$ is defined to be the vector dual to $\opn{dlog} \circ \; \psi_2^D(1)$. We set $e_n \defeq \phi_1(1)$ and define $f_n$ to be the image of $\phi_2(1)$ under the unit root splitting $(\omega_{\mathcal{E}}^+/p^n)^{\vee} \hookrightarrow \mathcal{H}_{\mathcal{E}}^+/p^n$. Then $\{e_n, f_n \}$ defines a basis of $\mathcal{H}_{\mathcal{E}}^+/p^n$ over $\mathcal{IG}_n$ respecting the Hodge filtration. 

\begin{definition} \label{DefOfAlphanBetanGamman}
Let $\alpha_n, \gamma_n \in (A_{\opn{ord}, n}^+)^{\times}$ and $\beta_n \in A_{\opn{ord}, n}^+$ be any elements such that
\[
e \equiv \alpha_n e_n, \quad \quad f \equiv \beta_n e_n + \gamma_n f_n, \quad \text{ mod } p^n ,
\]
and $\alpha_n^{-1} \beta_n \in A_{\opn{ord}}^+$.
\end{definition}

\begin{remark}
    We can choose $\alpha_n, \beta_n, \gamma_n$ in Definition \ref{DefOfAlphanBetanGamman} such that the last condition is satisfied as follows. Let $f^{\opn{ur}} \in \mathcal{H}_{\mathcal{E}}^+|_{\opn{Spa}(A_{\opn{ord}}, A_{\opn{ord}}^+)}$ denote the image of $f$ under the composition
\[
\mathcal{H}_{\mathcal{E}}^+|_{\opn{Spa}(A_{\opn{ord}}, A_{\opn{ord}}^+)} \twoheadrightarrow \left(\omega_{\mathcal{E}}^+\right)^{\vee}|_{\opn{Spa}(A_{\opn{ord}}, A_{\opn{ord}}^+)} \hookrightarrow \mathcal{H}_{\mathcal{E}}^+|_{\opn{Spa}(A_{\opn{ord}}, A_{\opn{ord}}^+)}
\]
where the first map arises from the Hodge filtration and the second map is the unit root splitting. There exists an element $\upsilon \in A_{\opn{ord}}^+$ such that $f^{\opn{ur}} = \upsilon e + f$. We can then take $\alpha_n, \gamma_n \in (A_{\opn{ord}, n}^+)^{\times}$ such that $e \equiv \alpha_n e_n$ and $f^{\opn{ur}} \equiv \gamma_n f_n$ modulo $p^n$, and set $\beta_n = -\upsilon \alpha_n$.
\end{remark}

\begin{remark}
Let $\lambda = (\lambda_1, \lambda_2) \in T(\mbb{Z}/p^n\mbb{Z}) = \left(\mbb{Z}/p^n\mbb{Z}\right)^{\times} \times \left(\mbb{Z}/p^n\mbb{Z}\right)^{\times}$, which corresponds to an element $\sigma_{\lambda} \in \Gal(\mathcal{IG}_n/\mathcal{X}_{\opn{ord}})$ such that $\sigma_{\lambda}(\psi_i) = \lambda_i \circ \psi_i$ for $i=1, 2$. Then, $\sigma_{\lambda}(e_n) \equiv \lambda_1^{-1} e_n$ and $\sigma_\lambda(f_n) \equiv \lambda_2^{-1} f_n$ modulo $p^n$. This implies that $\sigma_{\lambda}(\alpha_n) \equiv \lambda_1 \alpha_n$, $\sigma_{\lambda}(\beta_n) \equiv \lambda_1 \beta_n$, and $\sigma_{\lambda}(\gamma_n) \equiv \lambda_2 \gamma_n$ modulo $p^n$. 
\end{remark}

\begin{remark} \label{RemarkCongruencesalphan}
The map $\mathcal{IG}_\infty \times_{\mathcal{X}} \opn{Spa}(A, A^+) = \opn{Spa}(A_{\opn{ord}, \infty}, A_{\opn{ord}, \infty}^+) \to P_{\opn{dR}}^{\opn{an}} \times_{\mathcal{X}} \opn{Spa}(A, A^+)$ induced from the universal trivialisations and the unit root splitting is described by the point
\[
\tbyt{\gamma_{\infty}}{}{\beta_{\infty}}{\alpha_{\infty}} \in \overline{P}^{\opn{an}}(A_{\opn{ord}, \infty}), 
\]
i.e. the morphism sending $T_1 \mapsto \alpha_\infty$, $U \mapsto \beta_{\infty}$, and $T_2 \mapsto \gamma_{\infty}$.
\end{remark}

\subsection{A preliminary lemma}

We will first prove an elementary lemma that will be useful later for constructing strict neighbourhoods of the Igusa tower over the ordinary locus. Let $C^+$ be an admissible $\mbb{Z}_p$-algebra, $C = C^+[1/p]$, and let $\mathcal{Y} = \mbb{G}_m^{\opn{an}}$ over $\opn{Spa}(C, C^+)$ with coordinate $T$. Let $n \geq 1$ and suppose that we have a collection of elements
$\{ c_{\lambda} \in (C^+)^{\times} : \lambda \in \left(\mbb{Z}/p^n \mbb{Z} \right)^{\times} \}$ with the property that $c_{\lambda} \equiv \lambda c_{1}$ modulo $p^nC^+$ for all $\lambda \in \left(\mbb{Z}/p^n \mbb{Z}\right)^{\times}$. Define 
\[ P(T) := \prod_{\lambda \in (\Z / p^n \Z)^\times} (T - c_{\lambda}) \in C^+[T] \subset \mathcal{O}_{\mathcal{Y}}^+(\mathcal{Y}) \]
and, for every $0 \leq m \leq n$, let
\begin{align*} 
\rho_{n, m} & \defeq p^{-m p^{n - m}} \prod_{ \substack{x \in (\Z / p^n \Z)^\times \\ x \not \equiv 1 \text{ (mod $p^m$)}}} |x - 1|_p \\
\nu_{n, m} & \defeq - v_p(\rho_{n, m})
\end{align*}
where $v_p$ is the $p$-adic valuation satisfying $v_p(p) = 1$. Note that $\nu_{n, m} \geq 0$ and $\rho_{n, m} = p^{-\nu_{n, m}} = |p^{\nu_{n, m}}|_p$.

\begin{lemma} \label{rhoquotientlemma}
For any $0 \leq m < n$, we have $\rho_{n,m} / \rho_{n,m+1} = p^{p^{n - (m + 1)}}$ and $\nu_{n, m+1} - \nu_{n, m} = p^{n-(m+1)}$.
\end{lemma}

\begin{proof}
We have
\begin{align*} \nu_{n, m} &= m p^{n - m} + \sum_{k = 1}^{m - 1} k \cdot \, \# \{ 1+p^k a \, : \, 0 \leq a < p^{n-k}, \;\; (a, p) = 1\} \\ 
&= m p^{n - m} + \sum_{k = 1}^{m - 1} k (p-1) p^{n - (k+1)}.
\end{align*}
A simple computation shows that $\nu_{n, m+1} - \nu_{n, m} = p^{n - (m+1)}$ as required.
\end{proof}

The following lemma relates estimates of the norm of $P(T)$ to the norm of its linear factors.

\begin{lemma} \label{LemmaInequality}
For any point $x \in \mathcal{Y}$ with associated valuation $| \cdot |_x$, one has
\[ 
|P(T)|_x \leq |p^{\nu_{n, m}}|_x \Longleftrightarrow \exists \lambda \in (\Z / p^n \Z)^\times \text{ such that } |T - c_{\lambda} |_x  \leq |p^m|_x.
\]
\end{lemma}

\begin{proof}
Note that $p$ is topologically nilpotent, so $|p|_x < 1$ and $|p^i|_x$ converges to zero as $i \to +\infty$ in the value group $\Gamma \cup \{0\}$ of the valuation $|\cdot|_x$ (i.e. for any $\gamma \in \Gamma$, there exists an integer $i \geq 0$ such that $|p^i|_x < \gamma$). 

Suppose that there exists an element $\mu \in \left(\mbb{Z}/p^n \mbb{Z} \right)^{\times}$ such that $|p^{i+1}|_x < |T - c_{\mu}|_x < |p^{i}|_x$ for some integer $i \geq 0$. Then $|T - c_{\mu}|_x$ is never equal to $|c_{\mu} - c_{\lambda}|_x = |\mu - \lambda|_x$ for any $\lambda \neq \mu$. Let $J_{\mu, i+1} \subset \left(\mbb{Z}/p^n \mbb{Z} \right)^{\times}$ be the subset of all elements $\lambda$ such that $\lambda \equiv \mu$ modulo $p^{i+1}$, and let $J_{\mu, i+1}^c$ denote its complement. Then we have
\[
|P(T)|_x = \prod_{\lambda \in J_{\mu, i+1}} |T-c_{\mu}|_x \cdot \prod_{\lambda \in J_{\mu, i+1}^c} | \lambda - \mu |_x = \left\{ \begin{array}{cc} |T- c_{\mu}|_x^{p^{n-(i+1)}} |p^{-(i+1)p^{n-(i+1)}}|_x |p^{\nu_{n, i+1}}|_x & \text{ if } 0 \leq i < n \\ |T-c_{\mu}|_x |p^{-n}|_x |p^{\nu_{n, n}}|_x & \text{ if } i \geq n \end{array} \right.
\]
Now we return to the proof of the lemma. 

Suppose that there exists $\lambda \in \left(\mbb{Z}/p^n \mbb{Z} \right)^{\times}$ such that $|T - c_{\lambda}|_x \leq |p^{m}|_x$. Then we must have that
\[
|P(T)|_x \leq \prod_{\lambda' \in J_{\lambda, m}} |p^{m}|_x \cdot \prod_{\lambda' \in J_{\lambda, m}^c} |c_{\lambda'} - c_{\lambda}|_x = |p^{\nu_{n, m}}|_x .
\]
Conversely, suppose that $|P(T)|_x \leq |p^{\nu_{n, m}}|_x$. Suppose there exists $\mu \in \left(\mbb{Z}/p^n \mbb{Z} \right)^{\times}$ such that $|p^{i+1}|_x < |T - c_{\mu}|_x < |p^{i}|_x$ for some integer $i \geq 0$. We may assume $i \leq m-1$ otherwise there is nothing to prove. Then by the above calculation, we find that
\[
|P(T)|_x = |T - c_{\mu}|_x^{p^{n-(i+1)}} |p^{-(i+1)p^{n-(i+1)}}|_x |p^{\nu_{n, i+1}}|_x \leq |p^{\nu_{n, m}}|_x. 
\]
But since $|T - c_{\mu}|_x > |p^{i+1}|_x$, this implies that $\nu_{n, i+1} > \nu_{n, m}$ for some $i \leq m-1$. This contradicts Lemma \ref{rhoquotientlemma}.

Finally, we suppose that for every $\mu$, there exists an integer $i \geq 0$ such that $|T - c_{\mu}|_x = |p^{i}|_x$, and assume that $|P(T)|_x \leq |p^{\nu_{n, m}}|_x$. Let $j = \opn{max}\{ i : |T - c_{\lambda}|_x = |p^{i}|_x \text{ for some } \lambda \}$, and we assume $0 \leq j \leq m-1$, otherwise there is nothing to prove. Let $\mu$ be such that $|T - c_{\mu}|_x = |p^{j}|_x$. Then we see that
\[
|P(T)|_x = \prod_{\lambda \in J_{\mu, j}} |T - c_{\lambda}|_x \cdot \prod_{\lambda \in J_{\mu, j}^c} |c_{\lambda} - c_{\mu}|_x \geq |p^{jp^{n-j}}|_x \prod_{\substack{y \in \left(\mbb{Z}/p^n \mbb{Z}\right)^{\times} \\ y \not\equiv 1 \pmod{{p^j}}}} | y - 1 |_x = |p^{\nu_{n, j}}|_x .
\]
This implies that $\nu_{n, j} \geq \nu_{n, m}$, which contradicts Lemma \ref{rhoquotientlemma} because $j \leq m-1$.
\end{proof}

We note the following corollary:

\begin{corollary} \label{CorollaryInequalityPolynomials}
Let $1 \leq m \leq n$ and suppose that $\{ d_{\mu} \in (C^+)^{\times} : \mu \in \left(\mbb{Z}/p^m \mbb{Z} \right)^{\times} \}$ is a collection of elements such that $d_{\mu} \equiv c_{\lambda}$ modulo $p^m C^+$ whenever $\mu \equiv \lambda$ modulo $p^m$. Set $Q(T) = \prod_{\mu \in \left(\mbb{Z}/p^m \mbb{Z}\right)^{\times}} (T - d_{\mu})$. Then for any point $x \in \mathcal{Y}$, one has
\[
|P(T)|_x \leq |p^{\nu_{n, m}}|_x \; \Longleftrightarrow \; |Q(T)|_x \leq |p^{\nu_{m, m}}|_x.
\]
\end{corollary}
\begin{proof}
By Lemma \ref{LemmaInequality}, it is enough to show that
\[
\exists \lambda \in \left(\mbb{Z}/p^n \mbb{Z} \right)^{\times} \text{ such that } |T - c_{\lambda} |_x  \leq |p^{m}|_x \Longleftrightarrow \exists \mu \in \left(\mbb{Z}/p^m \mbb{Z} \right)^{\times} \text{ such that } |T - d_{\mu} |_x  \leq |p^{m}|_x .
\]
But this holds by construction.
\end{proof}

\subsection{Ordinary neighbourhoods} \label{OrdinaryNHoodsSubSec}

Fix a quasi-compact open affinoid $\opn{Spa}(A, A^+) \subset \mathcal{X}$ as above. We now define the strict (local) neighbourhoods of the Igusa tower over the ordinary locus. For any integer $n \geq 1$, we define 
\begin{itemize}
    \item $P_{\alpha_n}(T_1) = \prod_{\lambda = (\lambda_1, 1) \in T(\mbb{Z}/p^n \mbb{Z})} (T_1 - \sigma_{\lambda}(\alpha_n)) \; \in \; A_{\opn{ord}}^+[T_1]$
    \item $P_{\gamma_n}(T_2) = \prod_{\lambda = (1, \lambda_2) \in T(\mbb{Z}/p^n \mbb{Z})} (T_2 - \sigma_{\lambda}(\gamma_n)) \; \in \; A_{\opn{ord}}^+[T_2]$
    \item $Q_n(T_1, U) = U - \alpha_n^{-1}\beta_n T_1 \; \in \; A_{\opn{ord}}^+[T_1, U]$.
\end{itemize}
We view all of these polynomials as global sections of $\mathcal{O}^+_{P_{\opn{dR}, A_{\opn{ord}}}^{\opn{an}}}$ via the coordinates $T_1, T_2, U$ arising from the fixed basis $\{ e, f \}$ as above.

\begin{definition}
For $1 \leq m \leq n$, let $\mathcal{U}_{\opn{HT}, n, m} \subset P_{\opn{dR}, A_{\opn{ord}}}^{\opn{an}}$ be the quasi-compact open affinoid whose points $| \cdot |_x$ satisfy the inequalities
\[ 
|P_{\alpha_n}(T_1)|_x \leq |p^{\nu_{n, m}}|_x, \;\;\; |P_{\gamma_n}(T_2)|_x \leq |p^{\nu_{n, m}}|_x, \; \; \; |Q_n(T_1, U) |_x \leq |p^{m}|_x. 
\]
\end{definition}

For any element $x \in \overline{P}^{\opn{an}}(A_{\opn{ord}, n})$ and integer $m \leq n$, we let $\mathcal{B}_n(x, p^{-m}) \subset \overline{P}^{\opn{an}}_{A_{\opn{ord}, n}} \cong P_{\opn{dR}, A_{\opn{ord}, n}}^{\opn{an}}$ denote the rigid ball of radius $p^{-m}$, i.e. if we have $x = \tbyt{x_2}{}{y}{x_1}$, then it is the quasi-compact open affinoid subspace defined by the conditions $|T_1 - x_1| \leq |p^{m}|$, $|U - y| \leq |p^{m}|$, and $|T_2 - x_2| \leq |p^{m}|$.

\begin{lemma} \label{DecompositionIntoBallsLemma}
Let $g_n = \tbyt{\gamma_n}{}{\beta_n}{\alpha_n}$. Then
\[ \mathcal{U}_{\opn{HT}, n, m} \times_{\mathcal{X}_{\opn{ord}}} \mathcal{IG}_n = \bigsqcup_{\lambda \in T(\Z / p^m \Z)} \mathcal{B}_n(\sigma_\lambda(g_n), p^{-m}) = \bigsqcup_{\lambda \in T(\Z / p^m \Z)} \sigma_\lambda(g_n) \cdot \mathcal{B}_n(1, p^{-m}) \]
where, by a slight abuse of notation, $\sigma_\lambda$ denotes any lift of $\sigma_\lambda \in \mathscr{G}_m$ to an element of $\mathscr{G}_n$. In particular, we have $\mathcal{U}_{\opn{HT}, n} \defeq \mathcal{U}_{\opn{HT}, n, n} = \mathcal{U}_{\opn{HT}, n, m}$ for any $1 \leq m \leq n$.
\end{lemma}

\begin{proof}
This follows immediately from Lemma \ref{LemmaInequality} and Corollary \ref{CorollaryInequalityPolynomials} applied to $\{ \sigma_\lambda(\alpha_n) : \lambda = (\lambda_1, 1) \in T(\Z / p^n \Z) \}$ and $\{ \sigma_\lambda(\gamma_n) : \lambda = (1, \lambda_2) \in T(\Z / p^n \Z) \}$ (see Remark \ref{RemarkCongruencesalphan} for the congruence properties of these elements) and each of the $(\mbb{G}_m^{\opn{an}})_{A_{\opn{ord}}}$ components of $P_{\opn{dR}, A_{\opn{ord}}}^{\opn{an}}$.
\end{proof}

\begin{corollary} \label{OCextensionsarecofinalCorollary}
The family $\{ \mathcal{U}_{\opn{HT}, n} : n \geq 1 \} $ forms a cofinal system of strict quasi-compact open neighbourhoods of $\mathcal{IG}_{\infty} \times_{\mathcal{X}} \opn{Spa}(A, A^+)$ inside $P_{\opn{dR}, A_{\opn{ord}}}^{\opn{an}}$.
\end{corollary}

\begin{proof}
It suffices to prove this after base-change along the map $\opn{Spa}(A_{\opn{ord}, \infty}, A_{\opn{ord}, \infty}^+) \to \opn{Spa}(A_{\opn{ord}}, A_{\opn{ord}}^+)$, because this morphism is profinite pro\'{e}tale and hence closed. But
\[
\mathcal{U}_{\opn{HT}, n} \times_{\opn{Spa}(A_{\opn{ord}}, A_{\opn{ord}}^+)} \opn{Spa}(A_{\opn{ord}, \infty}, A_{\opn{ord}, \infty}^+) = \bigsqcup_{\lambda \in T(\mbb{Z}/p^n\mbb{Z})} \mathcal{B}_{\infty}(\sigma_{\lambda}(g_{\infty}), p^{-n})
\]
where $\sigma_{\lambda}$ denotes a lift to $\mathscr{G}_{\infty}$ (here we are using the fact that $g_n \equiv g_{\infty}$ modulo $p^n$). These are clearly cofinal.
\end{proof}

\subsection{Overconvergent neighbourhoods} \label{OverconvergentNBHDSsection}

Fix a quasi-compact open affinoid $\opn{Spa}(A, A^+) \subset \mathcal{X}$ over which $\mathcal{H}_{\mathcal{E}}^+$, $\omega_{\mathcal{E}}^+$, and $\omega_{\mathcal{E}^D}^+$ trivialise. We now define a cofinal system of strict neighbourhoods of $\mathcal{IG}_{\infty, A}$ inside $P_{\opn{dR}, A}^{\opn{an}}$.

\begin{definition} \label{DefOfOverconvergentExtension}
Let $n \geq 1$ be an integer. We say that a quasi-compact open affinoid subspace $\mathcal{U} \subset P_{\opn{dR}, A}^{\opn{an}}$ is an overconvergent extension of $\mathcal{U}_{\opn{HT}, n}$ if
\begin{enumerate}
    \item One has $\mathcal{U}_{\opn{HT}, n} = \mathcal{U} \cap P_{\opn{dR}, A_{\opn{ord}}}^{\opn{an}}$.
    \item $\mathcal{U}$ contains the closure of $\mathcal{IG}_{\infty, A}$ inside $P_{\opn{dR}, A}^{\opn{an}}$.
\end{enumerate}
Given an overconvergent extension $\mathcal{U}$, we set $\mathcal{U}_r \defeq \mathcal{U} \cap P_{\opn{dR}, A_r}^{\opn{an}}$ for any integer $r \geq 1$. Note that any $\mathcal{U}_r$ is also an overconvergent extension.
\end{definition}

\begin{proposition} \label{PropositionOCneighbourhoods}
For any $n \geq 1$, there exists an overconvergent extension of $\mathcal{U}_{\opn{HT}, n}$. Moreover, the collection of overconvergent extensions of $\mathcal{U}_{\opn{HT}, n}$ for varying $n$ forms a cofinal system of quasi-compact open strict neighbourhoods of $\mathcal{IG}_{\infty, A}$ inside $P_{\opn{dR}, A}^{\opn{an}}$.
\end{proposition}
\begin{proof}
Let $n \geq 1$. Then $\mathcal{U}_{\opn{HT}, n+1} = \mathcal{U}_{\opn{HT}, n+1, n+1}$ (resp. $\mathcal{U}_{\opn{HT}, n} = \mathcal{U}_{\opn{HT}, n+1, n}$) is described by the inequalities
\begin{align*}
|P_{\alpha_{n+1}}(T_1)| \leq |p^{\nu_{n+1, n+1}}|, \quad |P_{\gamma_{n+1}}(T_2)| \leq |p^{\nu_{n+1, n+1}}|, \quad |Q_{n+1}(T_1, U)| \leq |p^{n+1}| \\
\text{(resp. }  |P_{\alpha_{n+1}}(T_1)| \leq |p^{\nu_{n+1, n}}|, \quad |P_{\gamma_{n+1}}(T_2)| \leq |p^{\nu_{n+1, n}}|, \quad |Q_{n+1}(T_1, U)| \leq |p^{n}| \text{ ).}
\end{align*}
Since $A_{\opn{ord}}^+ = A^+\langle 1/h \rangle$, there exists a sufficiently large integer $N \geq 0$, and elements $P_{\alpha_{n+1}}'(T_1) \in A^+[T_1]$, $P_{\gamma_{n+1}}'(T_2) \in A^+[T_2]$, $Q_{n+1}'(T_1, U) \in A^+[T_1, U]$ such that
\[
P'_{\alpha_{n+1}}(T_1) \equiv h^N P_{\alpha_{n+1}}(T_1), \quad P'_{\gamma_{n+1}}(T_2) \equiv h^N P_{\gamma_{n+1}}(T_2), \quad Q_{n+1}'(T_1, U) \equiv h^N Q_{n+1}(T_1, U)
\]
modulo $p^{t}A_{\opn{ord}}^+$, for some fixed integer $t$ such that $p^{-t} \leq \rho_{n+1, n+1}$.

Define $\mathcal{V}$ (resp. $\mathcal{U}$) to be the rational subset of $P_{\opn{dR}, A}^{\opn{an}}$ defined by the inequalities 
\begin{align*}
    |P'_{\alpha_{n+1}}(T_1)| \leq |p^{\nu_{n+1, n+1}}|, \quad |P'_{\gamma_{n+1}}(T_2)| \leq |p^{\nu_{n+1, n+1}}|, \quad |Q'_{n+1}(T_1, U)| \leq |p^{n+1}| \\
    \text{(resp. } |P'_{\alpha_{n+1}}(T_1)| \leq |p^{\nu_{n+1, n}}|, \quad |P'_{\gamma_{n+1}}(T_2)| \leq |p^{\nu_{n+1, n}}|, \quad |Q'_{n+1}(T_1, U)| \leq |p^{n}| \text{ ).}
\end{align*}
Then $\mathcal{IG}_{\infty, A} \subset \mathcal{V} \subset \mathcal{U}$. Since $|p^{\nu_{n+1, n+1}}| < |p^{\nu_{n+1, n}}|$ (Lemma \ref{rhoquotientlemma}) and $|p^{n+1}| < |p^{n}|$, we see that $\overline{\mathcal{V}} \subset \mathcal{U}$, where the closure is in $P_{\opn{dR}, A}^{\opn{an}}$. This implies that $\mathcal{U}$ is an overconvergent extension of $\mathcal{U}_{\opn{HT}, n}$. The rest of the proposition follows from a standard compactness argument, and the fact that $\{\mathcal{U}_{\opn{HT}, n}: n \geq 1\}$ are cofinal over the ordinary locus.
\end{proof}

Let $\mathcal{U}$ be an overconvergent extension of $\mathcal{U}_{\opn{HT}, n}$ for some integer $n \geq 1$. Then we have a chain of quasi-compact open affinoid neighbourhoods 
\[
\mathcal{U}_1 \supset \mathcal{U}_2 \supset \ldots \supset \mathcal{U}_r \supset \ldots \supset \mathcal{U}_{\infty} \defeq \mathcal{U}_{\opn{HT}, n}
\]
such that $\mathcal{U}_{\infty}$ is the locus in $\mathcal{U}_r$ ($\defeq \mathcal{U} \cap P_{\opn{dR}, A_r}^{\opn{an}}$) where the Hasse invariant $h \in \mathcal{O}^+(\mathcal{U}_r)$ is invertible, for any $r \geq 1$. On sections, this chain of inclusions is induced by maps
\[
B_1^+ \to B_2^+ \to \ldots \to B_{\infty}^+
\]
where $B_r^+ = \mathcal{O}^+(\mathcal{U}_r)$. Each $B_r \defeq \mathcal{O}(\mathcal{U}_r)$ is a Banach space -- denote the norm for which $B_r^+$ is the unit ball by $|\!| \cdot |\!|_r$. The key result of this section is to show that these morphisms are very close to being isometries, with the failure to being an isometry measured by powers of the Hasse invariant (whose valuation can be made arbitrarily small). This will allow us to overconverge the Gauss--Manin connection to $p$-adic powers.

\begin{proposition} \label{KernelKilledByHasseProp}
Let $r \geq 1$. With notation as above, there exists an integer $M \geq 0$ (depending on $r$) such that for any $k \geq 1$, the kernel of the map
\[
B^+_{r}/p^k \to B^+_{\infty}/p^k
\]
is killed by $h^{kM}$.
\end{proposition}
\begin{proof}
We first establish the case $k=1$. Note that $B_{\infty}^+ = B_r^+\langle 1/h \rangle$ because $\mathcal{U}_r$ is an overconvergent extension of $\mathcal{U}_{\infty}$. This implies that $(B_\infty^+/p) = (B_r^+/p)[1/h]$, and hence any element $x \in B_r^+/p$ in the kernel of the map $B_r^+/p \to B_{\infty}^+/p$ is killed by some power of the Hasse invariant. Since $B_r^+/p$ is Noetherian, the kernel of this map is a finitely-generated ideal, and hence we can find a sufficiently large integer $M \geq 0$ such that $h^M$ kills the kernel. 

This proves the proposition in the case $k=1$. The general case now follows from a simple induction argument using the fact that $B_{\infty}^+$ is $p$-torsion free.
\end{proof}

The following corollary is a generalisation of \cite[Proposition 4.10]{AI_LLL} and it will be key for the proof of our main result.

\begin{corollary} \label{OCcorollary}
With notation as above, for any real number $0 < \delta < 1$ there exists an integer $s = s(\delta) \geq r$ such that the following holds: for all $k \geq 1$, $x \in B_{r}$ and $c \in \mbb{Q}$, one has
\[
|\!| x |\!|_{\infty} \leq p^{c-k} \text{ and } |\!| x |\!|_{r} \leq p^c \quad \Rightarrow \quad |\!| x |\!|_{s} \leq p^{c-\delta k}  .
\]
\end{corollary}

\begin{proof}
By raising $x$ to an integral power, it is enough to prove the statement for $c \in \mbb{Z}$, and by rescaling it is enough to prove the statement for $c=0$. Therefore, we have an element $x \in B_{r}^+$ whose image is in $p^k B_{\infty}^+$. By Proposition \ref{KernelKilledByHasseProp}, there exists an $M \geq 0$ such that $|\!| (p^{-1} h^M)^k x | \!|_s \leq |\!| (p^{-1} h^M)^k x | \!|_r \leq 1$ for any $s \geq r$ and any $k \geq 1$. Since $|\!|h^{-1}|\!|_s \to 1^+$ as $s \to +\infty$, taking $s$ large enough such that $|\!|h^{-1}|\!|_s  \leq p^{\frac{(1 - \delta)}{M}}$ we obtain
\[
|\!|x|\!|_s \leq |\!| (p h^{-M})^k |\!|_s \leq p^{-\delta k}
\]
as required.
\end{proof}

\section{\texorpdfstring{$p$}{p}-adic interpolation of continuous operators} \label{SectionPADICINTCONT}

In this section we establish the general abstract results on $p$-adic interpolation of operators that are used in \S \ref{GMinterpolationSection} in order to $p$-adically interpolate the Gauss-Manin connection. It is likely that the results in this section can be generalised to the setting of PEL Shimura varieties -- in particular some of the notation in this section will differ from the rest of the article.

\begin{remark}
At this point it might be useful to recall the strategy we described in the introduction for showing the $p$-adic interpolation of the Gauss--Manin connection. We make some comments concerning how the abstract results will be used in the following section.\footnote{The reader can also choose to read \S \ref{GMinterpolationSection} before reading this section.}
\begin{enumerate}
    \item In \S \ref{SecContAndAnalytic} we introduce general definitions of continuous, analytic and locally analytic actions on topological modules over $R[1/p]$ for some admissible $\zp$-algebra $R$. Proposition \ref{Lemmaepsan1} will be helpful as we will show that the Gauss-Manin connection has locally a very simple description modulo large powers of $p$ (cf. Proposition \ref{PropExplicitForm}).
    \item In \S \ref{SectNilpOp} we will specialise to a more precise situation and the rings $L_m = S^+ \langle \frac{X - 1}{p^m}, \frac{Y}{p^m}, \frac{Z -1}{p^m} \rangle$ will correspond by Proposition \ref{PropExplicitForm} to (the ring of functions of) one of the balls whose disjoint union forms the strict quasi-compact open neighbourhoods $\mathcal{U}_{\opn{HT}, m}$ of the Igusa tower over the ordinary locus of Corollary \ref{OCextensionsarecofinalCorollary}.
    \item Finally, the sequence $V_0 \subseteq \ldots \subseteq V_r \subseteq \ldots \subseteq V_\infty$ considered in \S \ref{SubSecOC} will correspond (locally) to a collection of overconvergent extensions of the neighbourhoods $\mathcal{U}_{\mathrm{HT}, m}$.
\end{enumerate}
\end{remark}

\subsection{Function spaces}

Let $R^+$ be an admissible $\zp$-algebra and set $R = R^+[1/p]$. Let $C_{\opn{cont}}(\zp, R)$, $C^{h\opn{-an}}(\zp, R)$ and $C^{\opn{la}}(\zp, R)$ denote the $R$-algebras of continuous, analytic of radius $p^{-h}$ ($h \in \N$), and locally analytic functions from $\zp$ to $R$ respectively. We also set $C_{\opn{cont}}(\zp, R^+)$ to be the $R^+$-algebra of continuous functions from $\mbb{Z}_p$ to $R^+$. We recall the following classical result.

\begin{prop}[Amice] \label{PropAmice} For $k \geq 0$ set $\binom{x}{k} = \frac{x(x-1)\cdots(x-k+1)}{k!}$. Then
\begin{itemize}
    \item The family $\binom{x}{k}$ is an orthonormal basis of $C_{\opn{cont}}(\zp, R)$ (over $R$).
    \item The family $\lfloor \frac{k}{p^h} \rfloor ! \binom{x}{k}$ is an orthonormal basis of $C^{h\opn{-an}}(\zp, R)$.
    \item A function $f = \sum_{k \in \N} a_k \binom{x}{k} \in C_{\opn{cont}}(\mbb{Z}_p, R)$ is locally analytic if and only if for some $\varepsilon > 0$, the term $p^{k \varepsilon} |a_k| \to 0$ as $k \to + \infty$, where $|\cdot |$ denotes the Banach norm on $R$ with unit ball $R^+$.
\end{itemize}
\end{prop}

\begin{proof}
See \cite[Corolaire I.2.4, Théorème I.4.7, Corolaire I.4.8]{ColmezFonctions}.
\end{proof}

For any $\varepsilon > 0$ we will denote by $C_\varepsilon(\mbb{Z}_p, R)$ the subspace of functions $f = \sum_{k \in \N} a_k \binom{x}{k}$ with $p^{k \varepsilon} |a_k| \to 0$ as $k \to +\infty$. Observe that $C^{\opn{la}}(\zp, R) = \varinjlim_{h \to +\infty} C^{h\opn{-an}}(\zp, R) = \varinjlim_{\varepsilon > 0} C_\varepsilon(\zp, R)$.

\subsection{Continuous and analytic actions} \label{SecContAndAnalytic}

We begin with some elementary calculations. Let $R^+$ be an admissible $\Z_p$-algebra, $R = R^+[1/p]$ and let $V$ be a $\qp$-Banach vector space equipped with a topological $R$-module structure.

\begin{definition} \label{DefAmice}
Let $C = C_{\opn{cont}}(\zp, R)$ or $C_\epsilon(\zp, R)$ or $C^{\opn{la}}(\zp, R)$ and let $W$ be a topological $R$-module. If $T \in \mathrm{End}_R(W)$ is a continuous operator, then we say that $T$ extends to a continuous (resp. $\epsilon$-analytic, resp. locally analytic) action if there exists a continuous $R$-algebra action of $C_{\opn{cont}}(\zp, R)$ (resp. $C_\epsilon(\mbb{Z}_p, R)$, resp. $C^{\opn{la}}(\zp, R)$), i.e. a continuous $R$-bilinear map
\[ C \times W \to W, \]
(where the left hand side is equipped with the product topology) that respects the algebra structure on $C$ given by multiplication of functions, such that the structural map $\zp \to R$ acts as $T$.
\end{definition}

\begin{lemma} \label{LemmaExtension}
Let $T \in \mathrm{End}_R(V)$ be a continuous operator and let $|\!| \cdot |\!|$ denote the operator norm on $\mathrm{End}_R(V)$. Then
\begin{enumerate}
    \item $T$ extends to a continuous action if and only if the norms
    \[ \left| \! \left| \frac{T (T - 1) \ldots (T - k + 1)}{k!} \right| \! \right| \]
    are uniformly bounded for all $k \geq 0$.
    \item $T$ extends to an $\epsilon$-analytic action if and only if the norms
    \[ p^{- k \epsilon} \left| \! \left| \frac{T (T-1) \ldots (T - k + 1)}{k!} \right| \! \right| \]
    are uniformly bounded for all $k \geq 0$.
    \item $T$ extends to a locally analytic action if and only if for all $\epsilon > 0$ the norms
    \[ p^{- k \epsilon} \left| \! \left| \frac{T (T-1) \ldots (T - k + 1)}{k!} \right| \! \right| \]
    are uniformly bounded for all $k \geq 0$.
\end{enumerate}
Moreover, any such extension of $T$ is unique.
\end{lemma}

\begin{proof}
This is an immediate consequence of the characterisation of continuous, $\epsilon$-analytic and locally analytic functions in Proposition \ref{PropAmice}.
\end{proof}

We will also need the following uniqueness property for $\varepsilon$-analytic/locally analytic actions on $\opn{LB}$-spaces with injective transition maps.

\begin{lemma} \label{UniquenessOfActionsLemma}
Let $\{ V_i \}_{i \in I}$ be a (filtered) countable sequence of $R$-modules which are also $\mbb{Q}_p$-Banach spaces, with each transition map injective. Let $V \defeq \varinjlim_{i \in I} V_i$ equipped with the direct limit topology. Let $\varepsilon > 0$ and let $T_i \colon V_i \to V_i$ be continuous $R$-linear operators, all compatible with each other via the transition maps. Set $T = \varinjlim_{i \in I}T_i$.
\begin{enumerate}
    \item Let $j \in I$ and $C(\mbb{Z}_p, R)^{\opn{pol}} \subset C_{\varepsilon}(\mbb{Z}_p, R)$ be the subalgebra of polynomial functions. There is a natural $R$-algebra action of $C(\mbb{Z}_p, R)^{\opn{pol}}$ on $V_j$ extending $T_j$. Then any continuous $R$-linear action 
    \[
    C_{\varepsilon}(\mbb{Z}_p, R) \times V_j \to V
    \]
    which extends this $R$-algebra action of $C(\mbb{Z}_p, R)^{\opn{pol}}$ is unique.
    \item Any continuous $R$-algebra action
    \[
    C_{\varepsilon}(\mbb{Z}_p, R) \times V \to V
    \]
    extending $T$ is unique.
    \item Any continuous $R$-algebra action
    \[
    C^{\opn{la}}(\mbb{Z}_p, R) \times V \to V
    \]
    extending $T$ is unique.
\end{enumerate}
\end{lemma}
\begin{proof}
In this proof we will continually use the fact that if $X$ is $\mbb{Q}_p$-Banach space and $Y = \varinjlim_{i \in I}Y_i$ is a countable filtered colimit of $\mbb{Q}_p$-Banach spaces (with the direct limit topology) where all the transition maps are injective, then any continuous $\mbb{Q}_p$-linear map
\[
U \colon X \to \varinjlim_{i \in I} Y_i = Y
\]
factors through $Y_i$ for some $i \in I$ (see \cite[Corollary 8.9]{Schneider}). Note that $C_{\varepsilon}(\mbb{Z}_p, R)$ is a $\mbb{Q}_p$-Banach space, and $C^{\opn{la}}(\mbb{Z}_p, R) = \varinjlim_{\varepsilon > 0} C_{\varepsilon}(\mbb{Z}_p, R)$ with the direct limit topology. 

For part (1), let $f \in C_{\varepsilon}(\mbb{Z}_p, R)$ and $v \in V_j$. Then $f$ can be written as the limit of $\{f_k\}_{k \geq 1}$, where $f_k \in C(\mbb{Z}_p, R)^{\opn{pol}}$. The map
\[
C_{\varepsilon}(\mbb{Z}_p, R) \to V, \quad \quad g \mapsto g \cdot v
\]
is continuous (by assumption), so by the above, factors though some $V_i$. Since $V_i$ is Hausdorff, we must have that $f \cdot v$ is the (unique) limit of $f_k \cdot v$, which are already determined.

Part (2) is similar, using the fact that the action of polynomial functions are already determined by $T$ (because it is an algebra action). Part (3) just follows from part (2).
\end{proof}

The following useful result explains how one can perturb a locally analytic action to produce an $\varepsilon$-analytic action for some appropriately chosen $\varepsilon$.

\begin{proposition} \label{Lemmaepsan1}
Let $T_1, T_2 \in \mathrm{End}_R(V)$ and suppose that $T_1$ extends to a locally analytic action and that $|\!|T_2|\!| \leq 1$. Then for any $\epsilon > 0$, there exists an integer $N_\epsilon \geq 1$ (depending on $T_1$ but not on $T_2$) such that $T_1 + p^{N_\epsilon} T_2$ extends to an $\epsilon$-analytic action.  
\end{proposition}

\begin{proof}
Let $N \geq 0$ and denote $T = T_1 + p^N T_2$. Let $f_k(X)$ denote the polynomial
\[ f_k(X) = \frac{X (X-1) \ldots (X - k + 1)}{k!}. \]
Let $\Lambda$ denote the set of all ordered tuples $I = (-1 = t_0 \leq t_1 < t_2 < \ldots < t_{2r-1} \leq t_{2r} = k-1)$ with $r \geq 1$. For such an $I \in \Lambda$ and any $1 \leq i \leq r$, let $k_i = t_{2i-1} - t_{2i-2}$ and $\ell_i = t_{2i} - t_{2i - 1}$, and set
 \[ z_I = f_{k_1}(T_1 - (t_0 + 1)) \cdot T_2^{\ell_1} \cdot f_{k_2}(T_1 - (t_2 + 1)) \cdot T_2^{\ell_2} \cdot \ldots \cdot f_{k_r}(T_1 - (t_{2r} + 1)) \cdot T_2^{\ell_r} .\]
Then we can write
\begin{equation} \label{eqtmp2alt}
f_k(T) = f_k(T_1) + \sum_{I \in \Lambda, b \neq 0} \frac{p^{N b}}{b!} \binom{k}{a}^{-1} \binom{a}{k_1, \ldots, k_r}^{-1} z_I,
\end{equation}
where we have denoted $a = \sum_{i=1}^r k_i$ and $b = \sum_{i=1}^r \ell_i$ so that $a + b = k$.

We now give a bound for the operators $z_I$. Note that for any integer $t \in \Z$, $f_k(X - t)$ is also a polynomial in $X$ and the Mahler expansion of $f_k(X - t)$ is of the form
\[ f_k(X - t) = \sum_{i = 0}^k a_i f_i(X) \] 
for $a_i \in \Z_p$ because $f_k(x - t) \in \zp$ for any $x \in \zp$ and the coefficients are just computed
using the discrete difference operator. This implies that, if $\epsilon > 0$ and $C \in \mbb{R}_{> 0}$ is such that
\[ p^{-k \epsilon / 2} |\!| f_k(T_1) |\!| \leq C \]
for all $k \geq 0$ then,
for any $t \in \Z$ and $k \geq 0$, we have
\[
p^{-k \epsilon / 2} |\!|f_k(T_1 - t)|\!| \leq C.
\]
Hence, by the fact that $|\!|T_2|\!| \leq 1$, we deduce
\begin{equation} \label{eqtmp1alt} p^{- a \epsilon / 2} |\!|z_I|\!| \leq C^{r}.
\end{equation}
Using that for any multinomial coefficient we have $v_p( \binom{n}{a_1, \ldots, a_s}^{-1} ) \geq - \log_p(n)$, we get from \eqref{eqtmp2alt} and \eqref{eqtmp1alt} that
\[ p^{-k \epsilon} |\!|f_k(T) |\!| \leq \mathrm{max}(C, p^{-k \epsilon} p^{-Nb + b/(p-1)} p^{2 \log_p(k)} p^{(k - b) \epsilon / 2} C^r ). \]
Since $b \geq r - 1$, taking $N$ such that $p^{-N + 1/(p-1)} \leq C^{-1}$ we deduce that $p^{-Nb + b/(p-1)} C^r \leq C$ and the right hand side is bounded above by
\[ p^{-k \epsilon / 2 + 2 \log_p(k)} C \]
which is uniformly bounded in $k$ as $\epsilon > 0$. This concludes the proof.
\end{proof}

\subsection{Nilpotent operators} \label{SectNilpOp}

Let $R^+ \to S^+$ be a morphism of $p$-adically complete and separated, $p$-torsion free $\zp$-algebras with $R^+$ admissible, and denote by $R = R^+[1/p], S = S^+[1/p]$ their associated $\qp$-Banach algebras,

\begin{definition}
Let $L_0 = S^+\langle X, Y , Z \rangle$ and, for any integer $m \geq 1$, let $L_m = S^+ \langle \frac{X - 1}{p^m}, \frac{Y}{p^m} , \frac{Z-1}{p^m}\rangle$. Set $V_0 = L_0[1/p]$, $V_m = L_m[1/p]$.
\end{definition}

We equip $V_0$ with the norm induced by the lattice $L_0$, or equivalently the norm given by $|\!| \sum_{a,b, c} s_{a, b, c} X^a Y^b Z^c |\!| =  \max_{a,b, c} |\!|s_{a,b, c}|\!|$, where the norm on the right hand side is the $\qp$-Banach norm of $S$. This allows us to view $V_0$ as a Banach algebra with unit ball $L_0$. Analogously, for $m \geq 1$, we can view $V_m$ as a Banach algebra with the norm induced by the lattice $L_m$, or equivalently defined by $|\!| \sum_{a,b, c} s_{a, b, c} (X-1)^a Y^b (Z-1)^c |\!|_m =  \max_{a,b,c} p^{-m(a + b + c)} |\!|s_{a,b, c}|\!|$.

Let $\theta: S^+ \to S^+$ be an $R^+$-linear derivation, which can naturally be viewed as an $S^+$-linear functional $\Omega^1_{S^+ / R^+} \to S^+$. We define $\nabla_\theta: V_m \to V_m$ to be the unique $R$-linear derivation such that it acts as $\theta$ on $S$ and satisfies
\[ \nabla_\theta(X) = Y, \;\;\; \nabla_\theta(Y) = \nabla_\theta(Z) = 0. \] In particular, we have a commutative diagram
\[
\begin{tikzcd}
V_m \arrow[r, "{\nabla_\theta}"] \arrow[d]             & V_m \arrow[d]            \\
S \arrow[r, "{\theta}"] & S,
\end{tikzcd}
\]
where the vertical maps are induced by sending $X \mapsto 1$, $Y \mapsto 0$, $Z \mapsto 1$.

We will show next that, whenever $\theta : S^+ \to S^+$ extends to a continuous action, the action of $C^{\opn{la}}(\zp, R) \subseteq C_{\opn{cont}}(\zp, R)$ on $S$ extends to $V_m$ for any $m \geq 0$. We begin with some elementary lemmas for which we introduce some notation. For any $k \in \N_{>0}$ and $0 \leq r \leq k$, let $\Sigma_{k, r}$ be the set of subsets of $\{0, \ldots, k-1\}$ of size $r$. If $I \in \Sigma_{k, r}$, let $k_1, \ldots, k_\ell$ (with $\ell$ and the $k_i$'s depending on $I$) be the lengths of the largest blocks of consecutive integers in $I$, so that $\sum_{i=1}^\ell k_i = r$. More precisely, $I$ can be written as
\[ 
I = \bigcup_{1 \leq j \leq \ell} [i_j, i_j + k_j - 1] = \bigcup_{1 \leq j \leq \ell} I_j
\]
with no adjacent intervals, i.e. $i_j > i_{j-1} + k_{j-1}$ for all $2 \leq j \leq \ell$. For any $I \in \Sigma_{k, r}$, we denote
\[ 
g_I(X) = \prod_{i \in I} (X - i), \;\;\; f_I(X) = \prod_{1 \leq j \leq \ell} \frac{1}{k_j!} \prod_{i \in I_j} (X - i),
\]
with the convention that $g_{\emptyset} = f_{\emptyset} = 1$. The following easy lemma will be very useful.

\begin{lemma} \label{Lemmaformulatheta}
We have
\[ f_k(\nabla_\theta)(s X^a Y^b Z^c) = \sum_{r = 0}^{\min(k, a)} \sum_{I \in \Sigma_{k, k-r}} \binom{k - r}{k_1, \ldots, k_\ell}^{-1} \binom{k}{r}^{-1} \binom{a}{r}  f_I(\theta)(s) X^{a-r} Y^{b+r} Z^c. \]
\end{lemma}

\begin{proof}
First observe that, since everything is linear with respect to the variable $Z$, we may assume that $c = 0$. We have 
\[ \binom{k - r}{k_1, \ldots, k_\ell}^{-1} \binom{k}{r}^{-1} \binom{a}{r}  f_I(\theta)(s) = \frac{a (a - 1) \cdots (a - r + 1)}{k!} g_I(\theta)(s), \]
so that we need to prove the formula
\[f_k(\nabla_\theta)(s X^a Y^b) =  \sum_{r = 0}^{\min(k, a)} \sum_{I \in \Sigma_{k, k-r}} \frac{a (a - 1) \cdots (a - r + 1)}{k!} g_I(\theta)(s) X^{a-r} Y^{b+r}. \]
We check the formula by induction. For $k=1$ we need to prove that
\[ f_k(\nabla_\theta)(s X^a Y^b) = \theta(s) X^a Y^b + a s X^{a-1} Y^{b + 1}, \]
which follows from Leibniz rule. Assume the result holds for $k$. We calculate
\begin{eqnarray*}
f_{k+1}(\nabla_\theta)(s X^a Y^b) &=& \frac{(\nabla_\theta - k)}{k+1} f_k(\nabla_\theta)(s X^a Y^b) \\
&=& (\nabla_\theta - k) \sum_{r = 0}^{\min(k, a)} \sum_{I \in \Sigma_{k, k-r}} \frac{a (a - 1) \cdots (a - r + 1)}{(k+1)!} g_I(\theta)(s) X^{a-r} Y^{b+r}.
\end{eqnarray*}
We have
\[ (\nabla_\theta - k) g_I(\theta)(s) X^{a - r} Y^{b + r} = (\theta - k) g_I(\theta)(s) X^{a - r} Y^{b + r} + g_I(\theta)(s) (a - r) X^{a - (r + 1)} Y^{b + (r + 1)}, \]
and observe that $(\theta - k) g_I(X) = g_{I \cup \{ k \}}(X)$. The result follows by decomposing $\Sigma_{k+1, k-r}$ as those subsets containing $k$ and those not containing $k$.
\end{proof}

The following corollary implies that, if the extension of $\nabla_\theta$ to $V_m$ extends to an $\epsilon$-analytic action, then the same holds for any $m' \geq m$.

\begin{corollary} \label{Corollarycomparisonnorms}
Let $m' \geq m$, $s \in S$, $a,b \in \N$ and $k \geq 0$. Then
\[ \bigg|\! \bigg| f_k(\nabla_\theta) \left(s \left(\frac{X - 1}{p^m}\right)^a \left( \frac{Y}{p^m} \right)^b \left( \frac{Z - 1}{p^m}\right)^c \right) \bigg| \! \bigg|_{m} =  \bigg| \! \bigg| f_k(\nabla_\theta) \left(s \left(\frac{X - 1}{p^{m'}}\right)^a \left( \frac{Y}{p^{m'}} \right)^b \left( \frac{Z - 1}{p^{m'}}\right)^c  \right) \bigg| \! \bigg|_{m'}. \]
In particular,
\[ \sup_{x \in L_{m}} |\! | f_k(\nabla_\theta) (x)) |\!|_m = \sup_{x \in L_{m'}} |\!| f_k(\nabla_\theta) (x))|\!|_{m'}\]
for all $m' \geq m$.
\end{corollary}

\begin{proof}
The first assertion follows immediately from Lemma \ref{Lemmaformulatheta}. Indeed, calling
\[ \alpha_r(s) = \sum_{I \in \Sigma_{k, k-r}} \binom{k - r}{k_1, \ldots, k_\ell}^{-1} \binom{k}{r}^{-1} \binom{a}{r}  f_I(\theta)(s), \]
which is independent of $m$ or $m'$
we have, for all $m' \geq m$,
\begin{eqnarray*}
\bigg|\! \bigg| f_k(\nabla_\theta) \left(s \left(\frac{X - 1}{p^{m'}}\right)^a \left( \frac{Y}{p^{m'}} \right)^b \left( \frac{Z - 1}{p^{m'}}\right)^c \right) \bigg| \! \bigg|_{m'} &=& \bigg|\!\bigg| \sum_{r = 0}^{\min(k, a)} \alpha_{r}(s) \left(\frac{X - 1}{p^{m'}}\right)^{a-r} \left( \frac{Y}{p^{m'}} \right)^{b+r} \left( \frac{Z - 1}{p^{m'}}\right)^c  \bigg|\! \bigg|_{m'}  \\
&=& \sup_{r} |\!|\alpha_r(s)|\!|,
\end{eqnarray*}
which is independent of $m'$. The last assertion follows from the first one. This finishes the proof.
\end{proof}

\begin{proposition} \label{Proplocanextension1}
Assume that $\theta$ extends to a continuous action on $S^+$. Then the operator $\nabla_\theta$ extends uniquely to a locally analytic action on $V_m$ for any $m \geq 0$.
\end{proposition}

\begin{proof}
By Corollary \ref{Corollarycomparisonnorms} and a change of coordinates, we can assume $m = 0$. By Lemma \ref{LemmaExtension}, it suffices to show that, for any $\epsilon > 0$, there exists a constant $C_\epsilon \in \R_{>0}$ such that for all $k \geq 0$ we have
\[ p^{-k \epsilon} |\!|f_k(\nabla_\theta) |\!| \leq C_\epsilon, \]
i.e. for all $F \in V_0$, we have
\[ p^{-k \epsilon} |\!|f_k(\nabla_\theta) (F) |\!| \leq C_\epsilon |\!|F|\!|. \]
It suffices to prove this for $F = s X^a Y^b Z^c$ with $s \in S^+$. By Lemma \ref{Lemmaformulatheta}, it suffices to show that, for each $0 \leq r \leq \min(k, a)$ and any $I \in \Sigma_{k, k-r}$, the value
\[ p^{-k \epsilon} \left|\! \left| \binom{k - r}{k_1, \ldots, k_\ell}^{-1} \binom{k}{r}^{-1} \binom{a}{r}  f_I(\theta)(s) \right| \! \right| \]
is uniformly bounded. But, since $|\!|f_I(\theta)|\!| \leq 1$ (because $\theta$ extends to a continuous action on $S^+$) this value is bounded above by $p^{-k \epsilon + 2\log_p(k)}$, which is uniformly bounded in $k$ as $-k \epsilon + 2 \log_p(k) \to - \infty$ as $k \to + \infty$. This finishes the proof.
\end{proof}

\subsection{Overconvergence} \label{SubSecOC}

We finish this section with the abstract results that will allow us to extend the locally analytic action from the ordinary locus to overconvergent neighbourhoods. The setting is as follows. Let
\[ V_0 \subseteq \ldots \subseteq V_r \subseteq \ldots \subseteq V_\infty \]
be a sequence of topological $R$-modules which are also $\qp$-Banach spaces and let $\nabla : V_\infty \to V_\infty$ be a continuous $R$-linear operator stabilising each $V_i$. Let $| \! | \cdot | \! |_i$ denote the Banach norm on $V_i$ and suppose that $| \! | x | \! |_s \leq | \! | x | \! |_r$ for all $x \in V_r$ and $r \leq s \leq \infty$.

\begin{proposition} \label{OverconvergenceOfActionProp}
Assume that the following property holds: for any $0 < \delta < 1$ and $r \in \N$, there exists $s = s(\delta) \geq r$ such that, for all $c \in \Q$, $h \in \N$ and $x \in V_r$ we have
\begin{equation} \label{Equationproperty}
 |\!|x|\!|_r \leq p^c \text{ and } |\!|x|\!|_{\infty} \leq p^{c - h} \implies |\!|x|\!|_s \leq p^{c -\delta h}.
\end{equation}
Assume that, for some $\epsilon > 0$, the operator $\nabla$ extends to an $\epsilon$-analytic action on $V_\infty$. Then, for any $r \in \mbb{N}$ and $\gamma > \epsilon$, there exists $s \in \N$ (depending only on $\epsilon$, $\gamma$ and $|\!|\nabla|\!|_r$) such that, for any $x\in V_r$,
\[ p^{-k \gamma} |\!|f_k(\nabla)(x)|\!|_s \to 0 \text{ as } k \to +\infty. \]
In particular, the operator $\nabla$ extends to a continuous $R$-linear action $C_{\gamma}(\mbb{Z}_p, R) \times V_r \to V_s$.
\end{proposition}

\begin{proof}
Let $x \in V_r$ be such that $|\!|x|\!|_r \leq 1$ (and hence $|\!|x|\!|_\infty \leq 1$), and let $C_r \in \mbb{Q}_{\geq 0}$ be a constant such that $|\!|\nabla|\!|_r \leq p^{C_r}$. From the inequality $|\!|\nabla|\!|_r \leq p^{C_r}$ we deduce that
\[ |\!|k! f_k(\nabla)(x)|\!|_r \leq  p^{k C_r}. \]
On the other hand, since $\nabla$ extends to an $\epsilon$-analytic action on $V_\infty$, there exists $C_\epsilon > 0$ such that
\[ |\!| k! f_k(\nabla)(x)|\!|_\infty \leq p^{-v_p(k!) + C_\epsilon + k \epsilon} = p^{kC_r - (k C_r - C_\epsilon - k \epsilon + v_p(k!))} \]
for all $k \in \N$. Let $k$ be large enough such that $k C_r - C_\epsilon - k \epsilon + v_p(k!) > 0$, which is always possible as soon as $C_r \geq \epsilon$. Applying \eqref{Equationproperty} to the element $k! f_k(\nabla)(x)$ with $c = kC_r$ and $h = k C_r - C_\epsilon - k \epsilon + v_p(k!) > 0$ (which we may assume to be an integer) we deduce that, for any $\delta < 1$, there exists $s \in \N$ such that
\[ |\!| k! f_k(\nabla)(x)|\!|_s \leq p^{k C_r - \delta(k C_r - C_\epsilon - k \epsilon + v_p(k!))}. \]
Now let $\gamma > \epsilon$. We obtain
\begin{eqnarray} \label{eaea}
p^{-k \gamma} |\!|f_k(\nabla)(x)|\!|_s &\leq& p^{k C_r - \delta k C_r + \delta C_\epsilon + k \delta \epsilon - \delta v_p(k!) - k \gamma + v_p(k!)} \\
& \leq & p^{-k \left( (\delta - 1) C_r + (\gamma - \delta \epsilon) + (\delta - 1) \right) + \delta C_{\varepsilon}} \nonumber
\end{eqnarray}
where the last inequality follows from $v_p(k!) \leq k$. One can easily show that one can choose $\delta$ (that will depend on $\varepsilon$, $\gamma$ and $C_r$) in such a way that $ \left( (\delta - 1) C_r + (\gamma - \delta \epsilon) + (\delta - 1) \right) > 0$ which implies that \eqref{eaea} goes to $0$ as $k \to + \infty$. This finishes the proof.
\end{proof}

\begin{remark} \label{OverconvergenceOfActionRemark}
It is not necessary that the maps $V_0 \to V_1 \to \cdots \to V_{\infty}$ are injective, and the above proof still holds without this assumption.
\end{remark}

\section{\texorpdfstring{$p$}{p}-adic interpolation of the Gauss-Manin connection} \label{GMinterpolationSection}

This section is devoted to the proof of the assertion of Theorem \ref{SheafVersionMainThm} concerning the existence of the action of locally analytic functions on the space of nearly overconvergent modular forms. We will establish this by first proving a local version over some affinoid $\opn{Spa}(A, A^+) \subset \mathcal{X}$, and then explain how the construction globalises. We assume throughout that $\opn{Spa}(A, A^+)$ is the adic generic fibre of an open formal subscheme $\opn{Spf}A^+ \subset \mathfrak{X}$.

\subsection{The Gauss--Manin connection} \label{TheGMConnectionSubSec}

Recall from Appendix \ref{AppendixClassicalNHMFs} that $\pi_* \mathcal{O}_{P_{\opn{dR}}} = \opn{VB}^{\opn{can}}(\mathcal{O}_{\overline{P}})$ is naturally a $\mathcal{D}_X$-module, where $\pi \colon P_{\opn{dR}} \to X$ denotes the structural map and $\mathcal{O}_{\overline{P}}$ denotes sections of the lower-triangular Borel $\overline{P} \subset \opn{GL}_2$. The $(\mathfrak{g}, \overline{P})$-module $\mathcal{O}_{\overline{P}}$ carries some additional structure, namely $\overline{P}$ acts through \emph{algebra} automorphisms of $\mathcal{O}_{\overline{P}}$ and $\mathfrak{g}$ acts through \emph{derivations} $\mathcal{O}_{\overline{P}} \to \mathcal{O}_{\overline{P}}$. Therefore we obtain a derivation $\nabla \colon \pi_*\mathcal{O}_{P_{\opn{dR}}} \to \pi_*\mathcal{O}_{P_{\opn{dR}}}$ as the composition 
\begin{equation} \label{DerivationCompositionEqn}
\pi_*\mathcal{O}_{P_{\opn{dR}}} \to \pi_*\mathcal{O}_{P_{\opn{dR}}} \otimes_{\mathcal{O}_X} \Omega^1_{X/R}(\opn{log}D) \to \pi_*\mathcal{O}_{P_{\opn{dR}}}
\end{equation}
where the first map is induced from the $\mathcal{D}_X$-module structure, and the second is induced from the Kodaira--Spencer isomorphism $\Omega^1_{X/R}(\opn{log}D) \cong \omega_{\mathcal{E}} \otimes \omega_{\mathcal{E}^D}$, the adjoint map to the universal trivialisation $\pi^*(\omega_{\mathcal{E}} \otimes \omega_{\mathcal{E}^D}) \xrightarrow{\sim} \mathcal{O}_{P_{\opn{dR}}}$, and the multiplication structure on $\pi_*\mathcal{O}_{P_{\opn{dR}}}$.

Since $\pi$ is affine, the derivation $\nabla$ is equivalent to a derivation $\nabla \colon \mathcal{O}_{P_{\opn{dR}}} \to \mathcal{O}_{P_{\opn{dR}}}$. We also let $\nabla \colon \mathcal{O}_{P_{\opn{dR}}^{\opn{an}}} \to \mathcal{O}_{P_{\opn{dR}}^{\opn{an}}}$ denote the corresponding derivation on the associated adic space. This immediately induces a derivation
\[
\nabla \colon \mathcal{N}^{\dagger} \to \mathcal{N}^{\dagger}
\]
where $\mathcal{N}^{\dagger} = \opn{colim}_U \pi_* \mathcal{O}_U$ is the colimit over all quasi-compact strict open neighbourhoods of $\mathcal{IG}_{\infty}$ in $P_{\opn{dR}}^{\opn{an}}$ (with structural map $\pi \colon U \to \mathcal{X}$). This is compatible with the Atkin--Serre operator via the map to $p$-adic modular forms.

\subsubsection{Local description} \label{LocalDescriptionOfGM}

We note the following local description of $\nabla \colon \mathcal{O}_{P_{\opn{dR}}^{\opn{an}}} \to \mathcal{O}_{P_{\opn{dR}}^{\opn{an}}}$, which follows from the way the $\mathcal{D}_X$-module structure is constructed in the proof of Lemma \ref{FEnrichedLemmaAppendix}. Let $\opn{Spa}(A, A^+) \subset \mathcal{X}$ be a quasi-compact open affinoid subspace as above, over which $\mathcal{H}_{\mathcal{E}}$, $\omega_{\mathcal{E}}$, $\omega_{\mathcal{E}^D}$ trivialise. Let $\{ e, f \}$ denote a basis of $\mathcal{H}_{\mathcal{E}}$ respecting the Hodge filtration, and let $\bar{f} \in \omega_{\mathcal{E}}^{\vee}$ denote the image of $f$ under the projection $\mathcal{H}_{\mathcal{E}} \twoheadrightarrow \omega_{\mathcal{E}}^{\vee}$ (we also denote this projection by $\overline{(-)}$). Via the Kodaira--Spencer isomorphism, there exists a unique differential $\kappa \in \Omega^1_{A/R}(\opn{log}D)$ such that
\[
\overline{\nabla(e)} = \bar{f} \otimes \kappa \in \omega_{\mathcal{E}}^{\vee} \otimes \Omega^1_{A/R}(\opn{log}D) .
\]
Let $D \colon A \to A$ denote the derivation dual to $\kappa$, then we can write 
\[
\nabla_D( (f \; e) ) = (f \; e) \cdot \delta 
\]
for some $\delta \in \mathfrak{g}_A$. Let $F \in A[U, T_1, T_2, T_1^{-1}, T_2^{-1}] \subset \mathcal{O}_{P_{\opn{dR}}^{\opn{an}}}(P_{\opn{dR}, A}^{\opn{an}})$ be a polynomial, which we view as an element of $\mathcal{O}_{\overline{P}} \otimes_{\mbb{Q}_p} A = A[U, T_1, T_2, T_1^{-1}, T_2^{-1}]$ (here we write a general element in $\overline{P}$ as $\tbyt{T_2}{}{U}{T_1}$). Then the action of $\nabla$ is described as 
\[
\nabla(F)(U, T_1, T_2) = \left( D \cdot F(U, T_1, T_2) + \delta \star_l F(U, T_1, T_2) \right) \cdot T_1 T_2^{-1} 
\]
where $D \cdot F(U, T_1, T_2)$ is the application of the derivation $D$ on the coefficients, and $\star_l$ denotes the $\mathfrak{g}$-action on $\mathcal{O}_{\overline{P}}$. The extra factor $T_1T_2^{-1}$ arises from the second map in (\ref{DerivationCompositionEqn}).

\begin{example} \label{ExampleInNilpotentBasis}
Suppose that $\delta = \tbyt{0}{1}{0}{0}$. Then the action $\delta \star_l -$ is given by the differential $U T_1 T_2^{-1} \partial_{T_1} - U \partial_{T_2}$. 
\end{example}

\subsection{The local construction} \label{LocalConstructionOfGMsubsec}

Fix a quasi-compact open affinoid $\opn{Spa}(A, A^+) \subset \mathcal{X}$ as in \S \ref{LocalDescriptionOfGM}. Let $n \geq 1$ be an integer. In this subsection, we will prove the local version of Theorem \ref{SheafVersionMainThm} using results from \S \ref{SectionNOCMFs} and \S \ref{SectionPADICINTCONT}. The first step is to understand the derivation 
\[
\nabla \colon \mathcal{O}_{\mathcal{U}_{\opn{HT}, n}} \to \mathcal{O}_{\mathcal{U}_{\opn{HT}, n}} 
\]
where $\mathcal{U}_{\opn{HT}, n}$ is as in Definition \ref{DefOfOverconvergentExtension}. Note that, since $\mathcal{IG}_{\infty} \to \mathcal{X}_{\opn{ord}}$ is a pro-\'{e}tale $T(\mbb{Z}_p)$-torsor, this derivation extends uniquely to a derivation $\nabla \colon \mathcal{O}_{\mathcal{U}_{\opn{HT}, n, A_{\opn{ord}, \infty}}} \to \mathcal{O}_{\mathcal{U}_{\opn{HT}, n, A_{\opn{ord}, \infty}}}$, and via the decomposition in Lemma \ref{DecompositionIntoBallsLemma} and the local description in \S \ref{LocalDescriptionOfGM}, this decomposes as
\[
\nabla = \oplus_{\sigma_{\lambda} \in \mathscr{G}_n} \nabla^{\lambda} \colon \bigoplus_{\lambda \in T(\mbb{Z}/p^n \mbb{Z})} \mathcal{O}_{\mathcal{B}_{\infty}(\sigma_{\lambda}(g_\infty), p^{-n})} \to \bigoplus_{\lambda \in T(\mbb{Z}/p^n \mbb{Z})} \mathcal{O}_{\mathcal{B}_{\infty}(\sigma_{\lambda}(g_\infty), p^{-n})} .
\]

\begin{proposition} \label{PropExplicitForm}
Let $n \geq 1$ be an integer and $\sigma_\lambda \in \mathscr{G}_\infty$ corresponding to $\lambda = (\lambda_1, \lambda_2) \in T(\mbb{Z}_p)$. For ease of notation, set $\mathcal{B}^{\lambda} = \mathcal{B}_{\infty}(\sigma_{\lambda}(g_\infty), p^{-n})$. 
\begin{enumerate}
    \item The operator $\nabla^{\lambda} \colon \mathcal{O}_{\mathcal{B}^{\lambda}}(\mathcal{B}^{\lambda}) \to \mathcal{O}_{\mathcal{B}^{\lambda}}(\mathcal{B}^{\lambda})$ is integral, i.e. it preserves $\mathcal{O}^+_{\mathcal{B}^{\lambda}}(\mathcal{B}^{\lambda})$. Moreover, the operator $\nabla^{\lambda}$ has the following explicit description:
    via the $A_{\opn{ord}, \infty}^+$-algebra isomorphism
    \begin{equation} \label{OBlambdaiso}
    \mathcal{O}^+_{\mathcal{B}_{\lambda}}(\mathcal{B}^{\lambda}) \xrightarrow{\sim} A_{\opn{ord}, \infty}^+ \langle \frac{X-1}{p^n}, \frac{Y}{p^n}, \frac{Z-1}{p^n} \rangle 
    \end{equation}
    given by sending $T_1 \mapsto \sigma_{\lambda}(\alpha_{\infty}) X$, $U \mapsto \sigma_{\lambda}(\beta_{\infty}) X + \sigma_{\lambda}(\gamma_{\infty}) Y$, $T_2 \mapsto \sigma_{\lambda}(\gamma_{\infty}) Z X^{-1}$, the operator $\nabla^{\lambda}$ is an $R^+$-linear derivation such that
    \begin{itemize}
        \item $\nabla^{\lambda}(g) \equiv s_{\lambda} \cdot \theta(g)$ (mod $ p^n$) for some $s_{\lambda} \in \mbb{Z}_p^{\times}$ and all $g \in A_{\opn{ord}, \infty}^+$
        \item $\nabla^{\lambda}(X) \equiv Y$ (mod $p^n$), and $\nabla^{\lambda}(Y) = \nabla^{\lambda}(Z) = 0$
    \end{itemize}
    where $\theta \colon A_{\opn{ord}, \infty}^+ \to A_{\opn{ord}, \infty}^+$ is the Atkin--Serre differential operator.
    \item The operator $\nabla \colon \mathcal{O}_{\mathcal{U}_{\opn{HT}, n}}(\mathcal{U}_{\opn{HT}, n}) \to \mathcal{O}_{\mathcal{U}_{\opn{HT}, n}}(\mathcal{U}_{\opn{HT}, n})$ is integral, i.e. it preserves $\mathcal{O}^+_{\mathcal{U}_{\opn{HT}, n}}(\mathcal{U}_{\opn{HT}, n})$.
\end{enumerate}
\end{proposition}
\begin{proof}
First recall that for $\lambda = (\lambda_1, \lambda_2)$ we have $\sigma_{\lambda}(\alpha_{\infty}) = \lambda_1 \alpha_{\infty}$, $\sigma_{\lambda}(\beta_{\infty}) = \lambda_1 \beta_{\infty}$ and $\sigma_{\lambda}(\gamma_{\infty}) = \lambda_2 \gamma_{\infty}$. Via the local description in \S \ref{LocalDescriptionOfGM}, we can calculate the Gauss--Manin connection using the basis $\{\lambda_1^{-1} e_{\infty}, \lambda_2^{-1} f_{\infty} \}$ of $\mathcal{H}_{\mathcal{E}}$ over $A_{\opn{ord}, \infty}$.  Let $\kappa \in \Omega^1_{A_{\opn{ord}, \infty}/R}(\opn{log}D)$ denote the unique differential satisfying 
\[
 \overline{\nabla(\lambda_1^{-1}e_{\infty})} = \lambda_2^{-1}\bar{f}_{\infty} \otimes \kappa
\]
and let $D \colon A_{\opn{ord}, \infty} \to A_{\opn{ord}, \infty}$ denote the derivation dual to $\kappa$. It is well-known (see \cite{howe2020unipotent} for example) that $D = s_{\lambda} \theta$ for some $s_{\lambda} \in \mbb{Z}_p^{\times}$ (this factor arises from comparing the universal trivialisations over $\mathfrak{IG}_{\infty}$ via the polarisation on $\mathcal{E}$). Let $\delta \in \mathfrak{g}_{A_{\opn{ord}, \infty}}$ be the element such that
\[
 \nabla_{D}(\lambda_2^{-1}f_{\infty} \; \lambda_1^{-1}e_{\infty} ) = (\lambda_2^{-1}f_{\infty} \; \lambda_1^{-1}e_{\infty}) \cdot \delta 
\]
We have that $\delta = \tbyt{0}{1}{0}{0}$. The coordinates corresponding to this basis are given by $X, Y, Z'$ satisfying the identity
\[
\tbyt{T_2}{}{U}{T_1} = \tbyt{Z'}{}{Y}{X} \sigma_{\lambda}(g_{\infty}), \quad \quad g_{\infty} = \tbyt{\gamma_{\infty}}{}{\beta_{\infty}}{\alpha_{\infty}}
\]
and we find that the action of $\delta \star_l -$ on $\mathcal{O}_{\overline{P}} \otimes_{\mbb{Q}_p} A_{\opn{ord}, \infty}$ is given by $Y X (Z')^{-1} \partial_{X} - Y \partial_{Z'}$ (see Example \ref{ExampleInNilpotentBasis}). Set $Z = Z' X$. Then we calculate that
\begin{itemize}
\item $\nabla^{\lambda}(g) = (Dg) \cdot X (Z')^{-1} = s_{\lambda} \cdot \theta(g) X^2 Z^{-1}$ for any $g \in A_{\opn{ord}, \infty}$
    \item $\nabla^{\lambda}(X) = YX (Z')^{-1} \cdot X (Z')^{-1} = Y X^4 Z^{-2}$
    \item $\nabla^{\lambda}(Y) = 0$
    \item $\nabla^{\lambda}(Z) = \nabla^{\lambda}(Z' X) = (-Y X (Z')^{-1})X + Z' Y X^4 Z^{-2} = - Y X^3 Z^{-1} + Y X^3 Z^{-1} = 0$.
\end{itemize}
This completes the proof of the first part of the proposition since $X \equiv Z \equiv 1$ (mod $p^n$). Part (2) follows from the fact that $\mathfrak{IG}_{\infty} \to \mathfrak{X}_{\opn{ord}}$ is a proetale $T(\mbb{Z}_p)$-torsor, so it is enough to check integrality of $\nabla$ after base-change along this torsor.
\end{proof}

We are now in a position to apply the general results in \S \ref{SectionPADICINTCONT} as we have an integral derivation $\nabla$ which is congruent mod $p^n$ to a nilpotent derivation that extends the Atkin--Serre operator. Recall that the Atkin--Serre operator $\theta \colon A^+_{\infty} \to A^+_{\infty}$ extends to a $R^+$-algebra action of $C_{\opn{cont}}(\mbb{Z}_p, R^+)$. Set $\mathscr{N}_{\mathcal{U}_{\opn{HT}, n}} = \mathcal{O}_{\mathcal{U}_{\opn{HT}, n}}(\mathcal{U}_{\opn{HT}, n})$. We have the following proposition:

\begin{proposition} \label{GMinterpolationOverOrdProp}
    Let $\epsilon > 0$. Then there exists $n(\epsilon) > 0$ such that for any $n \geq n(\epsilon)$ there exists a unique continuous $R$-algebra action
    \[
    C_{\varepsilon}(\mbb{Z}_p, R) \times \mathscr{N}_{\mathcal{U}_{\opn{HT}, n}} \to \mathscr{N}_{\mathcal{U}_{\opn{HT}, n}}
    \]
    extending $\nabla \colon \mathscr{N}_{\mathcal{U}_{\opn{HT}, n}} \to \mathscr{N}_{\mathcal{U}_{\opn{HT}, n}}$.
\end{proposition}
\begin{proof}
    It is enough to check this for the operators $\nabla^{\lambda}$ introduced at the start of this section. Indeed, the operator $\nabla = \oplus_{\lambda} \nabla^{\lambda}$ is Galois invariant and the map $\mathscr{N}_{\mathcal{U}_{\opn{HT}, n}} \hookrightarrow \mathscr{N}_{\mathcal{U}_{\opn{HT}, n}} \hatot_{A_{\opn{ord}}} A_{\opn{ord}, \infty}$ is an isometry (as $\mathfrak{IG}_{\infty} \to \mathfrak{X}_{\opn{ord}}$ is a pro\'{e}tale $T(\mbb{Z}_p)$-torsor). In this case, we have that
    \begin{itemize}
        \item By Proposition \ref{PropExplicitForm}, the derivation $\nabla^{\lambda}$ is congruent modulo $p^n$ to the operator ``$\nabla_{s_{\lambda} \theta}$'' defined in \S \ref{SectNilpOp}, extending the derivation $s_{\lambda} \theta$;
        
        \item The derivation $s_{\lambda}\theta$ extends to an action of $C_{\opn{cont}}(\mbb{Z}_p, R^+)$.
    \end{itemize}
    Therefore the result follows from Propositions \ref{Lemmaepsan1} and \ref{Proplocanextension1}. The action is unique because $\mathscr{N}_{\mathcal{U}_{\opn{HT}, n}}$ is a $\mbb{Q}_p$-Banach space and we can compute the action using Mahler expansions. 
\end{proof}

We now overconverge the above proposition. For any $n \geq 1$, let $\mathcal{U}_n$ denote an overconvergent extension of $\mathcal{U}_{\opn{HT}, n}$ as in Definition \ref{DefOfOverconvergentExtension}, and set $\mathcal{U}_{n, r} = \mathcal{U}_n \cap P_{\opn{dR}, A_r}^{\opn{an}}$. Without loss of generality, we may assume that $\mathcal{U}_n = \mathcal{U}_{n, 1}$. Set $\mathscr{N}^{\dagger}_{\mathcal{U}_{n,r}} = \mathcal{O}_{\mathcal{U}_{n,r}}(\mathcal{U}_{n, r})$, and recall from \S \ref{OverconvergentNBHDSsection} that we have a chain of restriction maps 
\begin{equation} \label{ChainOfNdaggerOC}
\mathscr{N}^{\dagger}_{\mathcal{U}_n} = \mathscr{N}^{\dagger}_{\mathcal{U}_{n,1}} \to \mathscr{N}^{\dagger}_{\mathcal{U}_{n,2}} \to \cdots \to \mathscr{N}_{\mathcal{U}_{\opn{HT}, n}}
\end{equation}
induced from the inclusions $\mathcal{U}_{\opn{HT, n}} \subset \cdots \subset \mathcal{U}_{n, 2} \subset \mathcal{U}_{n, 1} = \mathcal{U}_n$. Recall that Corollary \ref{OCcorollary} holds for this chain of maps.

\begin{proposition} \label{GeneralUVepsilonaction}
Let $\varepsilon > 0$. Then for any quasi-compact strict open neighbourhood $U$ of $\mathcal{IG}_{\infty, A}$ in $P_{\opn{dR}, A}^{\opn{an}}$, there exists a quasi-compact strict open neighbourhood $V \subseteq U$ of $\mathcal{IG}_{\infty, A}$ in $P_{\opn{dR}, A}^{\opn{an}}$ and a unique continuous $R$-linear action 
    \[
    C_{\varepsilon}(\mbb{Z}_p, R) \times \mathscr{N}^{\dagger}_{U} \to \mathscr{N}^{\dagger}_{V} 
    \]
    extending the action of polynomial functions in $C_{\varepsilon}(\mbb{Z}_p, R)$ induced from the operator $\nabla \colon \mathscr{N}^{\dagger}_{U} \to \mathscr{N}^{\dagger}_{U} \to \mathscr{N}^{\dagger}_{V}$. Moreover, these actions are all compatible if one changes $\varepsilon$, $U$ or $V$.
\end{proposition}
\begin{proof}
Let $n \geq 1$ be such that $n \geq n(\epsilon / 2)$ as in Proposition \ref{GMinterpolationOverOrdProp} and such that $\mathcal{U} \subset U$ for some overconvergent extension $\mathcal{U}$ of $\mathcal{U}_{\opn{HT}, n}$ (i.e., we can find a sufficiently large integer $n$ such that $U$ contains an overconvergent extension of $\mathcal{U}_{\opn{HT}, n}$; this is always possible since overconvergent extensions are cofinal by Corollary \ref{OCextensionsarecofinalCorollary}). Then, by Proposition \ref{OverconvergenceOfActionProp}, there exists some $r>0$ such that $\nabla$ extends to an $\epsilon$-analytic function on $\mathscr{N}^\dagger_{\mathcal{U}_{n, r}}$. This implies the result with $V = \mathcal{U}_{n, r}$. The last part follows from the unicity of the action.
\end{proof}

\subsection{The global construction}

We now explain how the construction in the previous subsection globalises. Let $U$ be any quasi-compact open strict neighbourhood of $\mathcal{IG}_{\infty}$ in $P_{\opn{dR}}^{\opn{an}}$ with structural map $\pi \colon U \to \mathcal{X}$. Set $\mathcal{N}^{\dagger}_U = \pi_* \mathcal{O}_U$. Let $\opn{Spa}(A, A^+) \subset \mathcal{X}$ be a quasi-compact open affinoid as in the previous section. Let $U_A = U \times_{\mathcal{X}} \opn{Spa}(A, A^+) \subset P_{\opn{dR}, A}^{\opn{an}}$ and note that $U_A$ is a quasi-compact strict open neighbourhood of $\mathcal{IG}_{\infty, A}$ in $P_{\opn{dR}, A}^{\opn{an}}$. We have $\mathcal{N}^{\dagger}_{U}(\opn{Spa}(A, A^+)) = \mathcal{O}_{U_A}(U_A)$.

By Proposition \ref{GeneralUVepsilonaction} (and passing to the limit as $\varepsilon \to 0$), there exists a continuous $R$-linear action
\[
C^{\opn{la}}(\mbb{Z}_p, R) \times \mathcal{N}^{\dagger}_{U}(\opn{Spa}(A, A^+)) \to \mathcal{N}^{\dagger}(\opn{Spa}(A, A^+))
\]
which is compatible with changing $U$. Furthermore, by the unicity property in Proposition \ref{GeneralUVepsilonaction}, this action is functorial in $\opn{Spa}(A, A^+)$. Since the opens $\opn{Spa}(A, A^+) \subset \mathcal{X}$ satisfying the conditions in the previous subsection are stable under finite intersections and cover $\mathcal{X}$, we obtain an induced action
\[
C^{\opn{la}}(\mbb{Z}_p, R) \times \opn{H}^0(\mathcal{X}, \mathcal{N}^{\dagger}) \to \opn{H}^0(\mathcal{X}, \mathcal{N}^{\dagger})
\]
extending the action of $\nabla$. This action is unique on global sections because $\mathscr{N}^{\dagger} = \opn{H}^0(\mathcal{X}, \mathcal{N}^{\dagger}) = \opn{colim}_U \opn{H}^0(U, \mathcal{O}_{P_{\opn{dR}}^{\opn{an}}})$ as $U$ runs over all quasi-compact open strict neighbourhoods of $\mathcal{IG}_{\infty}$ (the topological space $\mathcal{X}$ is quasi-compact and satisfies condition (4) in \cite[Tag 009F]{stacks-project}), and one can choose a cofinal system of quasi-compact open neighbourhoods such that the transition maps in this colimit are injective (see Proposition \ref{ReductionStructurePdRProposition}). One then applies Lemma \ref{UniquenessOfActionsLemma}. This completes the proof of the locally analytic action in Theorem \ref{SheafVersionMainThm}.

\section{Additional structures on nearly overconvergent modular forms} \label{AdditionalStructuresSection}

In this section we study various natural structures on the space of nearly overconvergent modular forms. We start by showing that $\mathcal{N}^\dagger$ is equipped with a natural action of $T(\mbb{Z}_p)$ and define (nearly) overconvergent modular forms of weight $\kappa$ as the $(-w_0\kappa)$-isotypic component for this action, where $w_0$ denotes the longest Weyl element of $\opn{GL}_2$. We then define a filtration on $\mathcal{N}^\dagger$ and the $U_p$-operator, and study how all of these operators interact with each other. We will finally conclude with the construction of the overconvergent projector, the relation with $p$-depletion, and a comparison with the spaces of nearly overconvergent modular forms introduced in \cite{AI_LLL}.

\subsection{The group action}

The definition of the $T(\mbb{Z}_p)$-action on (nearly) overconvergent modular forms is almost immediate from the construction. Indeed, by viewing $T(\mbb{Z}_p)$ as a subgroup of $T^{\opn{an}} \subset P^{\opn{an}}$, we obtain an action of $T(\mbb{Z}_p)$ on $P_{\opn{dR}}^{\opn{an}}$. Since $\mathcal{IG}_{\infty}$ is stable under this action, we see that the $T(\mbb{Z}_p)$-action on $P_{\opn{dR}}^{\opn{an}}$ maps any (quasi-compact) strict open neighbourhood of $\mathcal{IG}_{\infty}$ into another (quasi-compact) strict open neighbourhood; hence we obtain an action of $T(\mbb{Z}_p)$ on $\mathcal{N}^{\dagger}$. By exactly the same argument for $M_{\opn{dR}}^{\opn{an}}$, we also obtain a $T(\mbb{Z}_p)$-action on $\mathcal{M}^{\dagger}$, and these actions are compatible via the natural maps $\mathcal{M}^{\dagger} \to \mathcal{N}^{\dagger} \to \mathcal{M}$. Here $\mathcal{M}^{\dagger} \defeq \opn{colim}_U \pi_* \mathcal{O}_U$ where the colimit runs over all strict open neighbourhoods of $\mathcal{IG}_{\infty}$ in $M_{\opn{dR}}^{\opn{an}}$ (with structural map $\pi \colon U \to \mathcal{X}$).  

\begin{definition} \label{DefNearlyOCModFormsWeight}
Let $\kappa \colon T(\mbb{Z}_p) \to R^\times$ be a locally analytic character. We define the sheaf of nearly overconvergent modular forms of weight $\kappa$ as
\[ \mathcal{N}^\dagger_\kappa := \mathcal{N}^\dagger[-w_0\kappa] = \opn{Hom}_{T(\mbb{Z}_p)}(-w_0\kappa, \mathcal{N}^{\dagger}), \]
i.e. the $(-w_0\kappa)$-isotypic part for the action of $T(\mbb{Z}_p)$ on $\mathcal{N}^\dagger$. We define $\mathcal{M}^\dagger_\kappa$ analogously and we set $\mathscr{N}^{\dagger}_{\kappa}$ (resp. $\mathscr{M}^{\dagger}_{\kappa}$) to be the global sections of $\mathcal{N}^{\dagger}_{\kappa}$ (resp. $\mathcal{M}^{\dagger}_{\kappa}$).
\end{definition}

We note that our space of overconvergent modular forms of a specified weight agrees with the construction in \cite{HaloSpectral}. We begin with the following lemma:

\begin{lemma} \label{ReductionsOfStructureOfMdRLemma}
    For every integer $r \geq 1$, $M_{\opn{dR}}^{\opn{an}} \times_{\mathcal{X}} \mathcal{X}_r$ has a reduction of structure to an \'{e}tale torsor $\mathcal{F}_r \to \mathcal{X}_r$ for the group
    \[
    \mathcal{T}_r \defeq \mbb{Z}_p^{\times} \left( 1 + p^{r+1 - \frac{1}{p-1}} \mbb{G}_a^+ \right) \times \mbb{Z}_p^{\times} \left( 1 + p^{r+1 - \frac{1}{p-1}} \mbb{G}_a^+ \right) \subset \mbb{G}_m^{\opn{an}} \times \mbb{G}_m^{\opn{an}} = T
    \]
    where $\mbb{G}_a^+ = \opn{Spa}(\mbb{Q}_p\langle t \rangle, \mbb{Z}_p\langle t \rangle )$ denotes the unit ball. Furthermore, the torsors $\{ \mathcal{F}_r \}_{r \geq 1}$ form a cofinal system of strict quasi-compact open neighbourhoods of $\mathcal{IG}_{\infty}$ in $M_{\opn{dR}}^{\opn{an}}$.
\end{lemma}
\begin{proof}
    This follows directly from the results in \cite{Pilloni} and \cite{AIS} (see also \cite[Proposition 5.15]{BoxerPilloni}). 
\end{proof}

This leads to the following proposition.

\begin{proposition} \label{PropLocallyFree}
Let $\kappa: T(\mbb{Z}_p) \to R^\times$ be a locally analytic character. Then $\mathcal{M}^\dagger_\kappa$ is a locally free $\mathcal{O}_{\mathcal{X}}^{\dagger}$-module of rank one, where $\mathcal{O}_{\mathcal{X}}^{\dagger} \defeq \opn{colim}_r \mathcal{O}_{\mathcal{X}_r}$. Moreover, it is equal to the sheaf of weight $\kappa$ overconvergent modular forms as in \cite{HaloSpectral}.\footnote{Note that, in \cite{HaloSpectral}, the authors define overconvergent modular forms with weight given by a locally analytic character on $\mbb{Z}_p^{\times}$, which implicitly uses the polarisation on $\mathcal{E}$. However, the extension to characters of $T(\mbb{Z}_p)$ easily follows from their methods.} 
\end{proposition}
\begin{proof}
As above, $\mathcal{F}_r \to \mathcal{X}_r$ denotes the \'{e}tale $\mathcal{T}_r$-torsor providing a reduction of structure of $M_{\opn{dR}}^{\opn{an}} \times_{\mathcal{X}} \mathcal{X}_r$. The proposition follows immediately from the fact that the torsors $\{\mathcal{F}_r\}_{r \geq 1}$ form a cofinal system of quasi-compact strict open neighbourhoods of $\mathcal{IG}_{\infty}$ in $M_{\opn{dR}}^{\opn{an}}$.
\end{proof}

We finish this subsection by describing how the locally analytic action in \S \ref{GMinterpolationSection} interacts with the group action.

\begin{lemma} \label{IntertwiningOFGMandZpLemma}
    Let $\phi \in C^{\opn{la}}(\mbb{Z}_p, R)$. Then for any $t = \opn{diag}(t_1, t_2) \in T(\Z_p)$, we have 
    \[
    t ~ \circ ~ \phi = \phi( t_2^{-1} \cdot - \cdot t_1) ~ \circ ~  t
    \]
    as endomorphisms of $\mathscr{N}^{\dagger}$.
\end{lemma}
\begin{proof}
    Since $\mathscr{N}^{\dagger}$ is a (topological) $C^{\opn{la}}(\mbb{Z}_p, R)$-module and polynomial functions are dense in $C^{\opn{la}}(\mbb{Z}_p, R)$, it suffices to check this with $\phi$ equal to the structural map $\mbb{Z}_p \hookrightarrow R$, i.e. one has $t \circ \nabla = t_2^{-1}t_1 \nabla \circ t$ as endomorphisms of $\mathscr{N}^{\dagger}$. This amounts to checking the relation on $\pi_*\mathcal{O}_{P_{\opn{dR}}^{\opn{an}}} = \opn{VB}^{\opn{can}}_{K}(\mathscr{O}_{\overline{P}})$. But this just follows from the description of the $\mathcal{D}_{\opn{FL}}$-module structure in Lemma \ref{FEnrichedLemmaAppendix}. Indeed, if we identify the $\mathcal{D}_{\opn{FL}}$-module associated with $\mathscr{O}_{\overline{P}}$ with $\pi_* \mathcal{O}_G$ (where $\pi \colon G \to \opn{FL} = G/\overline{P}$ is the structural map), then the action of $\nabla$ corresponds to
    \[
    f \mapsto \left( g \mapsto (\opn{Ad}(g)X \star_l f)(g) + (X \star_r f)(g) \right), \quad \quad f \in \pi_* \mathcal{O}_G
    \]
    where $\star_l$ (resp $\star_r$) denotes the $\mathfrak{g}$-action obtained from left-translation (resp. right-translation) of the argument and $X = \left(\begin{smallmatrix} 0 & 1 \\ 0 & 0 \end{smallmatrix}\right)$. The action of $t$ is given by right-translation of the argument (i.e. the torsor structure), so the relation follows from $\opn{Ad}(t)X = t_2^{-1}t_1 X$.
\end{proof}

\subsection{Reductions of structure of \texorpdfstring{$P_{\opn{dR}}$}{PdR}}

It turns out that we can extend the torsors in Lemma \ref{ReductionsOfStructureOfMdRLemma} to reductions of structure of $P_{\opn{dR}}^{\opn{an}}$. More precisely, for any integer $r \geq 1$, let 
\[
\overline{\mathcal{P}}_r \defeq T(\mbb{Z}_p) \cdot \left\{ x \in \overline{\mathcal{P}} : x \equiv 1 \text{ modulo } p^{r+1 - \frac{1}{p-1}} \right\}
\]
where $\overline{\mathcal{P}}$ denotes the adic generic fibre of the formal $p$-adic completion of $\overline{P}_{\mbb{Z}_p}$.

\begin{proposition} \label{ReductionStructurePdRProposition}
    For any integer $r \geq 1$, there exists an integer $s \geq r$ and an \'{e}tale $\overline{\mathcal{P}}_r$-torsor $\tilde{\mathcal{F}}_{r, s} \to \mathcal{X}_s$ such that
    \begin{enumerate}
        \item $\tilde{\mathcal{F}}_{r, s}$ provides a reduction of structure of $P_{\opn{dR}}^{\opn{an}} \times_{\mathcal{X}} \mathcal{X}_s$ (to an \'{e}tale $\overline{\mathcal{P}}_r$-torsor) and $\mathcal{IG}_{\infty}$ provides a reduction of structure (to a pro\'{e}tale $T(\mbb{Z}_p)$-torsor) of $\tilde{\mathcal{F}}_{r, s} \times_{\mathcal{X}_s} \mathcal{X}_{\infty}$. In particular $\tilde{\mathcal{F}}_{r, s}$ is a quasi-compact open subset of $P_{\opn{dR}}^{\opn{an}}$ containing the closure of $\mathcal{IG}_{\infty}$.
        \item The pushout $\tilde{\mathcal{F}}_{r, s} \times^{\overline{\mathcal{P}}_r} \mathcal{T}_r$ along the natural projection $\overline{\mathcal{P}}_r \twoheadrightarrow \mathcal{T}_r$ coincides with $\mathcal{F}_r \times_{\mathcal{X}_r} \mathcal{X}_s$.
    \end{enumerate}
    Moreover, for any two \'{e}tale $\overline{\mathcal{P}}_{r}$-torsors $\mathscr{F}$ and $\mathscr{G}$ over $\mathcal{X}_s$ satisfying (1), there exists an integer $s' \geq s$ such that we have an identification of torsors
    \begin{equation} \label{FscrequalsGscrTorsors}
    \mathscr{F} \times_{\mathcal{X}_s} \mathcal{X}_{s'} = \mathscr{G} \times_{\mathcal{X}_s} \mathcal{X}_{s'} 
    \end{equation}
    (viewed as open subspaces of $P_{\opn{dR}}^{\opn{an}}$). The collection of torsors $\{ \tilde{\mathcal{F}}_{r, s} \}_{r, s}$ as $r, s$ vary form a cofinal system of quasi-compact open subsets of $P_{\opn{dR}}^{\opn{an}}$ containing the closure of $\mathcal{IG}_{\infty}$.
\end{proposition}
\begin{proof}
    Let $\mathcal{IG}_{r, s}$ denote the pushout of $\mathcal{F}_{r} \times_{\mathcal{X}_r} \mathcal{X}_s$ along the natural map $\mathcal{T}_r \twoheadrightarrow T(\mbb{Z}/p^{r+\delta}\mbb{Z})$, where $\delta = 0$ (resp. $\delta = 1$) if $p=2$ (resp. $p$ is odd). We first establish existence of the torsors $\tilde{\mathcal{F}}_{r, s}$. Note that we have a Cartesian diagram:
    
    \[
\begin{tikzcd}
{\mathcal{IG}_{r, \infty}} \arrow[d] \arrow[r] & {\mathcal{IG}_{r, s}} \arrow[d] \\
\mathcal{X}_{\infty} \arrow[r]                 & \mathcal{X}_s                  
\end{tikzcd}
    \]
    where $\mathcal{IG}_{r, \infty} \to \mathcal{X}_{\infty}$ denotes the Igusa cover with Galois group $T(\mbb{Z}/p^{r+\delta}\mbb{Z})$. Let $\{ \opn{Spf}A^{+, (i)} \}_{i \in I}$ be a finite open cover of $\mathfrak{X}$ over which $\mathcal{H}_{\mathcal{E}}$ and $\omega_{\mathcal{E}}$ trivialise, and let $U^{(i)} = \opn{Spa}(A^{(i)}, A^{+, (i)})$ denote the adic generic fibre of $\opn{Spf}A^{+, (i)}$. Let $U_s^{(i)} = \opn{Spa}(A^{(i)}_{s}, A^{+,(i)}_{s})$ and $U_{r, s}^{(i)} = \opn{Spa}(A^{(i)}_{r, s}, A^{+,(i)}_{r, s})$ denote the pullbacks of $U^{(i)}$ to $\mathcal{X}_s$ and $\mathcal{IG}_{r, s}$ respectively. For $(i, j) \in I \times I$, set 
    \[
    U^{(i, j)}_{r, s} = U^{(i)}_{r, s} \times_{\mathcal{X}_s} U^{(j)}_{r, s} = \opn{Spa}(B_{r, s}^{(i, j)}, B_{r, s}^{+, (i, j)} ) .
    \]
    Note that $A^{+,(i)}_{r, \infty} = A^{+,(i)}_{r, s}\langle 1/h \rangle$ and $B_{r, \infty}^{+, (i, j)} = B_{r, s}^{+, (i, j)}\langle 1/h \rangle$, where $h$ is a local lift of the Hasse invariant, and the natural restriction maps $A^{+,(i)}_{r, s} \to A^{+,(i)}_{r, \infty}$ and $B_{r, s}^{+, (i, j)} \to B_{r, \infty}^{+, (i, j)}$ are injective (since all of these open subspaces arise from a cover of the formal scheme $\mathfrak{X}$). From the construction of $\mathcal{F}_r$ (see \cite{HaloSpectral}), we observe that $\mathcal{F}_r$ has sections over $U^{(i)}_{r, s}$ (for any $s \geq r$).

    Let $t_i \colon U^{(i)} \to P_{\opn{dR}, A^{(i)}}^{\opn{an}}$ be a section of the torsor $P_{\opn{dR}}^{\opn{an}} \to \mathcal{X}$. Then a section of $\mathcal{F}_{r+1}$ over $U_{r+1, r+1}^{(i)}$ can be described as 
    \[
    u_i \defeq t_i \cdot g_i \colon U_{r+1, r+1}^{(i)} \to \mathcal{F}_{r+1} \times_{\mathcal{X}_{r+1}} U_{r+1, r+1}^{(i)}
    \]
    for some $g_i \in T^{\opn{an}}(A_{r+1, r+1}^{(i)})$. We can choose the sections $t_i$ such that we have 
    \[
    t_i|_{U^{(i)} \cap U^{(j)}} = t_j|_{U^{(i)} \cap U^{(j)}} \cdot x_{i, j}
    \]
    with $x_{i, j} \in \overline{\mathcal{P}}(A^{(i, j)})$. Then the image of $g_j^{-1} x_{i, j} g_i$ under the projection $\overline{P}^{\opn{an}} \to T^{\opn{an}}$ is contained in $\mathcal{T}_{r+1}(B^{(i, j)}_{r+1, r+1})$. Let 
    \[
    \tilde{u}_i \colon U_{r+1, \infty}^{(i)} \xrightarrow{u_i} \mathcal{F}_{r+1} \times_{\mathcal{X}_{r+1}} U_{r+1, \infty}^{(i)} \to P_{\opn{dR}}^{\opn{an}} \times_{\mathcal{X}} U_{r+1, \infty}^{(i)}
    \]
    where the second map is given by the unit root splitting. Then we can write $\tilde{u}_i = t_i \cdot \tilde{g}_i$ for some $\tilde{g}_i \in \overline{P}^{\opn{an}}(A^{(i)}_{r+1, \infty})$ whose image under the projection $\overline{P}^{\opn{an}}(A^{(i)}_{r+1, \infty}) \to T^{\opn{an}}(A^{(i)}_{r+1, \infty})$ coincides with $g_i$. Furthermore, we have $\tilde{g}_j^{-1} x_{i, j} \tilde{g}_i \in \overline{\mathcal{P}}_{r+1}(B_{r+1, \infty}^{(i, j)})$.

    By using the fact that $A^{+,(i)}_{r+1, \infty} = A^{+,(i)}_{r+1, r+1}\langle 1/h \rangle$, we can find elements $h_i \in \overline{P}^{\opn{an}}(A_{r+1, r+1}^{(i)})$ such that $\tilde{g}_i^{-1} h_i \in \overline{\mathcal{P}}_{r+1}(A_{r+1, \infty}^{(i)})$. Set $v_{i, j} = h_j^{-1} x_{i, j} h_i \in \overline{P}^{\opn{an}}(B_{r+1, r+1}^{(i, j)})$. Then we see that (because $I$ is finite) we can find a sufficiently large integer $s \gg r+1$ such that $v_{i, j} \in \overline{\mathcal{P}}_{r}(B_{r+1, s}^{(i, j)})$. We may also choose $s$ such that the image of $h_i$ under the projection map $\overline{P}^{\opn{an}} \to T^{\opn{an}}$ (denoted $h_i'$) satisfies $g_i^{-1}h'_i \in \mathcal{T}_r(A_{r+1, s}^{(i)})$. The \'{e}tale $\overline{\mathcal{P}}_r$-torsor $\tilde{\mathcal{F}}_{r, s}$ is then defined by the collection of sections
    \[
    \{ t_i \cdot h_i \colon U_{r+1, s}^{(i)} \to P_{\opn{dR}}^{\opn{an}} : i \in I \},
    \]
    or alternatively, by the transition matrices $\{ v_{i, j} : (i, j) \in I \times I \}$. The cocycle condition is automatic, and $\tilde{\mathcal{F}}_{r, s}$ provides a reduction of structure of $P_{\opn{dR}}^{\opn{an}}$ by construction. Clearly the pushout of $\tilde{\mathcal{F}}_{r, s}$ along the map $\overline{\mathcal{P}}_{r} \to \mathcal{T}_r$ is given by the sections $\{ t_i \cdot h_i' \}_{i \in I}$, which describes the torsor $\mathcal{F}_r \times_{\mathcal{X}_r} \mathcal{X}_s$ (by the condition $g_i^{-1}h'_i \in \mathcal{T}_r(A_{r+1, s}^{(i)})$ and the fact that $\mathcal{F}_r \times_{\mathcal{X}_r} \mathcal{X}_s$ is the pushout of $\mathcal{F}_{r+1} \times_{\mathcal{X}_{r+1}} \mathcal{X}_s$ along the map $\mathcal{T}_{r+1} \to \mathcal{T}_r$). Finally, the pullback $\tilde{\mathcal{F}}_{r, s} \times_{\mathcal{X}_s} \mathcal{X}_{\infty}$ is described by the orbit
    \[
    \left( \mathcal{F}_{r+1} \times_{\mathcal{X}_{r+1}} \mathcal{X}_{\infty} \right) \cdot \overline{\mathcal{P}}_r \subset P_{\opn{dR}}^{\opn{an}} \times_{\mathcal{X}} \mathcal{X}_{\infty}
    \]
    where we view $\left( \mathcal{F}_{r+1} \times_{\mathcal{X}_{r+1}} \mathcal{X}_{\infty} \right) \subset P_{\opn{dR}}^{\opn{an}} \times_{\mathcal{X}} \mathcal{X}_{\infty}$ via the unit root splitting. Therefore, we immediately see that $\mathcal{IG}_{\infty}$ defines a reduction of structure of $\tilde{\mathcal{F}}_{r, s} \times_{\mathcal{X}_s} \mathcal{X}_{\infty}$. This completes the proof of existence. 

    Suppose that $\mathscr{G}$ is an \'{e}tale $\overline{\mathcal{P}}_{r}$-torsor over $\mathcal{X}_s$ satisfying (1). Then it suffices to establish an identification of the form (\ref{FscrequalsGscrTorsors}) for $\mathscr{F} = \tilde{\mathcal{F}}_{r, s'}$ (with $s' \geq s$). We will continue to use the same notation as above for the construction of $\tilde{\mathcal{F}}_{r, s'}$. We have sections 
    \[
    \{ t_i \cdot h_i \colon U_{r+1, s'}^{(i)} \to P_{\opn{dR}, A_{r+1, s'}^{(i)}}^{\opn{an}} : i \in I \}
    \]
    whose pullback to $U_{r+1, \infty}^{(i)}$ give sections of $\mathcal{IG}_{\infty} \times^{T(\mbb{Z}_p)} \overline{\mathcal{P}}_{r+1}$. Since $\mathscr{G}$ satisfies $(1)$, we see that $\mathscr{G}$ is a strict neighbourhood of $\mathcal{IG}_{\infty} \times^{T(\mbb{Z}_p)} \overline{\mathcal{P}}_{r+1}$ inside $P_{\opn{dR}}^{\opn{an}}$. This implies that there exists a sufficiently large integer $s'' \gg s'$ such that $t_i \cdot h_i$ factors as
    \[
    U_{r+1, s''}^{(i)} \to \mathscr{G} \times_{\mathcal{X}_s} U_{r+1, s''}^{(i)} \to P_{\opn{dR}}^{\opn{an}} \times_{\mathcal{X}} U_{r+1, s''}^{(i)}
    \]
    i.e. we obtain a section of $\mathscr{G}$. This implies that $\mathscr{G} \times_{\mathcal{X}_s} \mathcal{X}_{s''} = \tilde{\mathcal{F}}_{r, s'} \times_{\mathcal{X}_{s'}} \mathcal{X}_{s''}$ as required. Furthermore, one easily verifies that the system $\{\tilde{\mathcal{F}}_{r, s} \}_{r, s}$ is a cofinal collection of strict neighbourhoods of $\mathcal{IG}_{\infty}$.
\end{proof}

We consider the following representations and sheaves. Let $T^+ \subset T(\mbb{Q}_p)$ denote the submonoid of diagonal matrices $\opn{diag}(t_1, t_2)$ satisfying $v_p(t_1) \geq v_p(t_2)$. One has an isomorphism $\mbb{Z}^2 \times T(\mbb{Z}_p) \cong T(\mbb{Q}_p)$ given by sending $(n_1, n_2) \times t$ to the matrix $\left( \begin{smallmatrix} p^{n_1} & \\ & p^{n_2} \end{smallmatrix}\right) t$. We let $\langle \cdot \rangle \colon T(\mbb{Q}_p) \to T(\mbb{Z}_p)$ denote the projection to the second component.

\begin{definition} \label{DefinitionOfVrwithmonoidaction}
    Let $r \geq 1$ and let $\Sigma_r = \mathcal{\overline{P}}_r \cdot T^+ \cdot \mathcal{\overline{P}}_r$.
    \begin{enumerate}
        \item Let $V_r = \mathcal{O}(\overline{\mathcal{P}}_r)$ denote the $\mbb{Q}_p$-Banach algebra of global sections $\overline{\mathcal{P}}_r \to \mbb{A}^{1, \opn{an}}$. We view this as a representation of $\Sigma_r$ via the following action:
        \[
        (p \star_l f)(-) = f(p^{-1} \cdot -), \quad ( t \star_l f)(-) = f(t^{-1} \cdot - \cdot t\langle t \rangle^{-1}), \quad \quad p \in \overline{\mathcal{P}}_r, t \in T^+, f \in V_r .
        \]
        \item If $\kappa \colon T(\mbb{Z}_p) \to R^{\times}$ is an $r$-analytic character, then we set 
        \[
        V_{r, \kappa} \defeq \{ f \in V_r\hatot R : f(x t) = (w_0\kappa)(t^{-1}) f(x) \text{ for all } t \in \mathcal{T}_r \} .
        \]
        This sub-Banach space is stable under the action of $\Sigma_r$. 
        \item For any torsor $\pi \colon \tilde{\mathcal{F}}_{r, s} \to \mathcal{X}_s$ as in Proposition \ref{ReductionStructurePdRProposition}, we let 
        \[
        \mathcal{V}_{r, s} = \left( \pi_*\mathcal{O}_{\tilde{\mathcal{F}}_{r, s}} \hatot V_r \right)^{\overline{\mathcal{P}}_r, \star_l} \cong \pi_*\mathcal{O}_{\tilde{\mathcal{F}}_{r, s}}
        \]
        where the invariants are with respect to the $\star_l$-action above. This is a locally projective Banach sheaf of $\mathcal{O}_{\mathcal{X}_s}$-modules in the sense of \cite[Definition 2.5.2]{BPHCT}, and comes equipped with an action $\star_r$ of $\overline{\mathcal{P}}_r$ via the torsor structure (or equivalently via the action $(p \star_r f)(-) = f(- \cdot p)$ on $V_r$). Similarly, we set 
        \[
        \mathcal{V}_{r, s, \kappa} = \left( \pi_*\mathcal{O}_{\tilde{\mathcal{F}}_{r, s}} \hatot V_{r, \kappa} \right)^{\overline{\mathcal{P}}_r, \star_l}
        \]
        which is also a locally projective Banach sheaf of $\mathcal{O}_{\mathcal{X}_s}$-modules. We can view $\mathcal{V}_{r, s, \kappa} \subset \mathcal{V}_{r, s}$ as the subsheaf of sections $f \in \mathcal{V}_{r, s}$ satisfying $t \star_r f = (w_0\kappa)(t^{-1}) f$ for all $t \in \mathcal{T}_r$ (recall that we are viewing $\mathcal{X}$ as an adic space over $\opn{Spa}(R, R^+)$). 
    \end{enumerate}
\end{definition}

\begin{remark}
    Note that $\mathcal{V}_{r, s}$ and $\mathcal{V}_{r, s, \kappa}$ depend on the choice of the torsor $\tilde{\mathcal{F}}_{r, s}$, but for any two choices, one can always find an integer $s' \geq s$ such that the pullback of the sheaves to $\mathcal{X}_{s'}$ coincide (see Proposition \ref{ReductionStructurePdRProposition}). We will therefore omit the choice of torsor from the notation.
\end{remark}

\begin{remark} \label{CompactOperatorOnVrkappaRem}
    Let $T^{++} \subset T^+$ be the semi-group of diagonal matrices $\opn{diag}(t_1, t_2)$ which satisfy $v_p(t_1) > v_p(t_2)$. Then, for any $t \in T^{++}$, we have $t \star_l V_{r+1, \kappa} \subset V_{r, \kappa}$. Since the inclusion $V_{r, \kappa} \to V_{r+1, \kappa}$ is compact, we see that the operator $t \star_l - \colon V_{r+1, \kappa} \to V_{r+1, \kappa}$ is also compact. 
\end{remark}

Since the system of torsors $\{ \tilde{\mathcal{F}}_{r, s} \}_{r, s}$ is cofinal, we see that
\[
\mathscr{N}^{\dagger} = \varinjlim_{r, s} \opn{H}^0\left( \mathcal{X}_s, \mathcal{V}_{r, s} \right) \quad \text{ and } \quad \mathscr{N}_{\kappa}^{\dagger} = \varinjlim_{r, s} \opn{H}^0\left( \mathcal{X}_s, \mathcal{V}_{r, s, \kappa} \right)
\]
where the transition maps are with respect to restriction and the natural inclusions $V_{r} \subset V_{r+1}$ and $V_{r, \kappa} \subset V_{r+1, \kappa}$. We also have versions of this on the level of sheaves, namely
\[
\mathcal{N}^{\dagger} = \varinjlim_{r, s} \mathcal{V}_{r, s} \quad \text{ and } \quad \mathcal{N}_{\kappa}^{\dagger} = \varinjlim_{r, s} \mathcal{V}_{r, s, \kappa} .
\]
Here we are implicitly viewing $\mathcal{V}_{r, s}$ and $\mathcal{V}_{r, s, \kappa}$ as sheaves of $\mathcal{O}_{\mathcal{X}}$-modules by pushing forward along the inclusion $\mathcal{X}_s \subset \mathcal{X}$. These descriptions are beneficial for the discussion of filtrations and Hecke operators below.

\subsection{Filtrations} \label{FiltrationsSubSec}

We now define an ascending filtration on nearly overconvergent forms.

\begin{definition}
    Let $r \geq 1$ and $\kappa \colon T(\mbb{Z}_p) \to R^{\times}$ an $r$-analytic character.
    \begin{enumerate}
        \item For an integer $h \geq 0$, let $\opn{Fil}_h V_r \subset V_r$ denote the sub-$\Sigma_r$-representation consisting of global sections $f \in V_r$
        \[
        \tbyt{x}{}{y}{z} \mapsto f(x, y, z), \quad \quad \tbyt{x}{}{y}{z} \in \overline{\mathcal{P}}_r
        \]
        which are polynomial in the variable $y$ of degree $\leq h$. This coincides with the subspace of elements in $V_r$ killed by the action of $\overline{\mathfrak{n}}^{h+1}$ under $\star_r$. Similarly, we let $\opn{Fil}_h V_{r, \kappa} \subset V_{r, \kappa}$ denote the subrepresentation of elements which are polynomial in the variable $y$ of degree $\leq h$ (or equivalently the elements killed by $\overline{\mathfrak{n}}^{h+1}$).
        \item For an integer $h \geq 0$, let $\opn{Fil}_h\mathcal{V}_{r, s} \subset \mathcal{V}_{r, s}$ and $\opn{Fil}_h\mathcal{V}_{r, s, \kappa} \subset \mathcal{V}_{r, s, \kappa}$ denote the subsheaves given by
        \[
        \opn{Fil}_h\mathcal{V}_{r, s} = \left( \pi_* \mathcal{O}_{\tilde{\mathcal{F}}_{r, s}} \hatot \opn{Fil}_h V_{r} \right)^{\overline{\mathcal{P}}_r, \star_l}, \quad \opn{Fil}_h\mathcal{V}_{r, s, \kappa} = \left( \pi_* \mathcal{O}_{\tilde{\mathcal{F}}_{r, s}} \hatot \opn{Fil}_h V_{r, \kappa} \right)^{\overline{\mathcal{P}}_r, \star_l} .
        \]
        As in (1), these are the subsheaves killed by the action of $\overline{\mathfrak{n}}^{h+1}$. The sheaf $\opn{Fil}_h\mathcal{V}_{r, s, \kappa}$ is a locally free $\mathcal{O}_{\mathcal{X}_s}$-module of finite rank which is independent of $r$. We set
        \[
        \opn{Fil}_h \mathcal{N}^{\dagger} =\varinjlim_{r, s} \opn{Fil}_h\mathcal{V}_{r, s} \quad \text{ and } \quad \opn{Fil}_h \mathcal{N}_{\kappa}^{\dagger} =\varinjlim_{r, s} \opn{Fil}_h\mathcal{V}_{r, s, \kappa}
        \]
        where, as above, we are implicitly considering $\opn{Fil}_h\mathcal{V}_{r, s}$ and $\opn{Fil}_h\mathcal{V}_{r, s, \kappa}$ as sheaves on $\mathcal{X}$ by pushing forward along the inclusion $\mathcal{X}_s \subset \mathcal{X}$.
        \item We let $\opn{Fil}_h\mathscr{N}^{\dagger} = \opn{H}^0\left(\mathcal{X}, \opn{Fil}_h\mathcal{N}^{\dagger} \right)$ and $\opn{Fil}_h\mathscr{N}^{\dagger}_{\kappa} = \opn{H}^0\left(\mathcal{X}, \opn{Fil}_h\mathcal{N}_{\kappa}^{\dagger} \right)$. Equivalently, we have
        \[
        \opn{Fil}_h\mathscr{N}^{\dagger} = \varinjlim_{r, s}\opn{H}^0\left(\mathcal{X}_s, \opn{Fil}_h\mathcal{V}_{r, s} \right) \quad \text{ and } \quad \opn{Fil}_h\mathscr{N}_{\kappa}^{\dagger} = \varinjlim_{r, s}\opn{H}^0\left(\mathcal{X}_s, \opn{Fil}_h\mathcal{V}_{r, s, \kappa} \right) . 
        \]
    \end{enumerate}
\end{definition}

\begin{remark}
    One can easily verify that $\opn{Fil}_h V_{r, \kappa} = \opn{Sym}^h \opn{St} \otimes R(w_0\kappa + h(1; -1))$, where $\opn{St}$ is the standard representation of $\opn{GL}_2$ and $R(w_0\kappa + h(1;-1))$ denotes the one-dimensional line given by the character $w_0\kappa + h(1;-1)$ (inflated to $\overline{\mathcal{P}}_r$). If we view $\kappa$ as a pair $(\kappa'; w)$ satisfying $\kappa(t, st^{-1}) = \kappa'(t)w(s)$, then we have
    \[
    \opn{Fil}_h \mathcal{V}_{r, s, \kappa} = \opn{Sym}^h \mathcal{H}_{\mathcal{E}} \otimes \left(\omega_{\mathcal{E}^D}^{\kappa' -h + w} \otimes \omega_{\mathcal{E}}^{- w}  \right)
    \]
    which agrees with the definition given in \cite{Urban} (after identifying $\omega_{\mathcal{E}}$ and $\omega_{\mathcal{E}^D}$ via the polarisation). 

    Similarly, we have
    \[
    \opn{Fil}_h\mathcal{V}_{r, s} = \opn{Sym}^h\mathcal{H}_{\mathcal{E}} \; \hatot \; \omega_{\mathcal{E}^D}^{-h} \; \hatot \; \pi'_* \mathcal{O}_{\mathcal{F}_r \times_{\mathcal{X}_r} \mathcal{X}_s}
    \]
    where $\pi' \colon \mathcal{F}_r \times_{\mathcal{X}_r} \mathcal{X}_s \to \mathcal{X}_s$ denotes the structural map. One easily sees that $\opn{Fil}_0\mathcal{N}^{\dagger} = \mathcal{M}^{\dagger}$ and $\opn{Fil}_0\mathcal{N}_{\kappa}^{\dagger} = \mathcal{M}^{\dagger}_{\kappa}$.
\end{remark}

\subsubsection{Compatibility with the Gauss--Manin connection}

Let $U=\opn{Spa}(B, B^+) \to \mathcal{X}_s$ be an \'{e}tale morphism over which the torsor $\pi \colon \tilde{\mathcal{F}}_{r, s} \to \mathcal{X}_s$ trivialises. Then any section $f \in \mathcal{V}_{r, s}(U)$ is described as a morphism $f \colon \pi^{-1}(U) \to V_r \hatot B$ such that $f(xp) = p^{-1} \star_l f(x)$ for all $x \in \pi^{-1}(U)$ and $p \in \overline{\mathcal{P}}_r$. If we identify $\pi^{-1}(U) = U \times \overline{\mathcal{P}}_r$ then the connection $\nabla$ on $\mathcal{V}_{r, s}(U)$ has the following description
\[
\nabla(f)(p) = (D \cdot f(p) + \delta \star_l f(p)) \cdot y
\]
for $p \in U \times \overline{\mathcal{P}}_r$ and $f \in \mathcal{V}_{r, s}(U)$, where $D \colon B \to B$ is some derivation and $\delta \in \mathfrak{g}_B$. Here $y \in \opn{Fil}_0\mathcal{V}_{r, s, 2\rho}(U)$ is the element arising from the Kodaira--Spencer isomorphism. If $f \in \opn{Fil}_h\mathcal{V}_{r, s}(U)$ then the action $\delta \star_l -$ can increase the degree of $f(p)$ in the unipotent variable by at most $1$, so we immediately see that:

\begin{lemma}
For all $h \geq 0$, one has $\nabla\left( \opn{Fil}_h \mathcal{V}_{r, s} \right) \subset \opn{Fil}_{h+1} \mathcal{V}_{r, s}$.
\end{lemma}

Let $\opn{Gr}_h \mathcal{V}_{r, s, \kappa} = \opn{Fil}_h\mathcal{V}_{r, s, \kappa} / \opn{Fil}_{h-1}\mathcal{V}_{r, s, \kappa}$ (with the convention $\opn{Fil}_{-1}\mathcal{V}_{r, s, \kappa} = 0$). We have the following lemma describing the connection on graded pieces.

\begin{lemma} \label{NablaBarIsoUpToDenominators}
Let $\kappa \colon T(\mbb{Z}_p) \to R^{\times}$ be an $r$-analytic character which we view as a pair $(\kappa'; w)$ of $r$-analytic characters $\mbb{Z}_p^{\times} \to R^{\times}$ such that $\kappa(t, st^{-1}) = \kappa'(t)w(s)$. Let $u_{\kappa} \in R$ be the element such that $\kappa'(t) = \opn{exp}(u_{\kappa} \opn{log}(t))$ for $t \in 1 + p^{r} \mbb{Z}_p$. Then 
\[
\nabla \colon \opn{Gr}_h \mathcal{V}_{r, s, \kappa} \to \opn{Gr}_{h+1}\mathcal{V}_{r, s, \kappa + 2\rho}
\]
is an isomorphism multiplied by $u_{\kappa} - h$.
\end{lemma}
\begin{proof}
By using the polarisation $\mathcal{E} \xrightarrow{\sim} \mathcal{E}^D$, we can identify this map with that in \cite[Theorem 3.18]{AI_LLL}, from which the result follows.
\end{proof}

\subsection{The \texorpdfstring{$U_p$}{Up}-operator} \label{TheUpOperatorSSec}

Let $r \geq 1$. We have a correspondence
\[
\mathcal{X}_r \xleftarrow{p_1} \mathcal{C}_{r} \xrightarrow{p_2} \mathcal{X}_{r+1}
\]
where $\mathcal{C}_r$ denotes the moduli of order $p$ subgroups $C \subset \mathcal{E}[p]$ which are disjoint from the canonical subgroup $H_1^{\opn{can}}$. The map $p_1$ is the natural one, and $p_2(E) = E/C$ is the quotient by the order $p$ subgroup. Note that $p_1$ is finite flat of degree $p$ and $p_2$ is an isomorphism. Let $(\mathcal{E}, C)$ denote the universal generalised elliptic curve and order $p$ subgroup over $\mathcal{C}_{r}$, and set $\mathcal{E}' = \mathcal{E}/C$. Let $\lambda \colon \mathcal{E} \to \mathcal{E}'$ denote the corresponding isogeny, and $\lambda^D \colon (\mathcal{E}')^D \to \mathcal{E}^D$ its dual. We have a natural isomorphism $\lambda_* \colon \mathcal{H}_{\mathcal{E}} \xrightarrow{\sim} \mathcal{H}_{\mathcal{E}'}$.

\begin{definition} \label{DefinitionOfPhiTopCorr}
    Let $P_{\opn{dR}, r}^{\opn{an}} = P_{\opn{dR}}^{\opn{an}} \times_{\mathcal{X}} \mathcal{X}_r$. Let $\phi \colon p_1^{-1}P_{\opn{dR}, r}^{\opn{an}} \to p_2^{-1}P_{\opn{dR}, r+1}^{\opn{an}}$ denote the isomorphism given by sending a trivialisation $\psi \colon \mathcal{O}^{\oplus 2} \xrightarrow{\sim} \mathcal{H}_E$ to the trivialisation
    \[
    \lambda_* \circ \psi \circ \left( \begin{smallmatrix} 1 & \\ & p^{-1} \end{smallmatrix} \right) \colon \mathcal{O}^{\oplus 2} \xrightarrow{\sim} \mathcal{H}_{E/C}
    \]
    where $C \subset E[p]$ denotes the choice of order $p$ subgroup which is disjoint from the canonical subgroup. 
\end{definition}

We consider the following commutative diagram
\[
\begin{tikzcd}
{P_{\opn{dR}, r}^{\opn{an}}} \arrow[d] & {p_1^{-1}P_{\opn{dR}, r}^{\opn{an}}} \arrow[l, "p_1"'] \arrow[rd] \arrow[rr, "\phi"] &                                                    & {p_2^{-1}P_{\opn{dR}, r+1}^{\opn{an}}} \arrow[r, "p_2"] \arrow[ld] & {P_{\opn{dR}, r+1}^{\opn{an}}} \arrow[d] \\
\mathcal{X}_r                          &                                                                                      & \mathcal{C}_r \arrow[ll, "p_1"'] \arrow[rr, "p_2"] &                                                                    & \mathcal{X}_{r+1}                       
\end{tikzcd}
\]
and we set $q_2 = p_2 \circ \phi$.

We have an analogue of this for Igusa towers over the ordinary locus, namely 
\[
\mathcal{IG}_{\infty} \xleftarrow{p_1} p_1^{-1}\mathcal{IG}_{\infty} \xrightarrow{\phi} p_2^{-1}\mathcal{IG}_{\infty} \xrightarrow{p_2} \mathcal{IG}_{\infty} 
\]
where the map $\phi$ sends a pair of trivialisations $(\psi_1, \psi_2)$ of $E[p^{\infty}]^{\circ}$ and $E[p^{\infty}]^{\opn{et}}$ respectively to $(\lambda \circ \psi_1, p^{-1}\lambda \circ \psi_2)$. This is well-defined because the canonical subgroup of $E$ is $E[p]^{\circ}$, and $C$ is \'{e}tale. As above, we set $q_2 = p_2 \circ \phi$. Then we have a commutative diagram:
\[
\begin{tikzcd}
{\mathcal{IG}_{\infty}} \arrow[d] & {p_1^{-1}\mathcal{IG}_{\infty}} \arrow[l, "p_1"'] \arrow[r, "q_2"] \arrow[d] & {\mathcal{IG}_{\infty}} \arrow[d] \\
P_{\opn{dR}, r}^{\opn{an}}        & p_1^{-1}P_{\opn{dR}, r}^{\opn{an}} \arrow[l, "p_1"'] \arrow[r, "q_2"]     & P_{\opn{dR}, r+1}^{\opn{an}}       
\end{tikzcd}
\]
with the left-hand square Cartesian.

If we let $\mathcal{Z}$ (resp. $\mathcal{Z}'$) denote the closure of $\mathcal{IG}_{\infty}$ (resp. $p_1^{-1}\mathcal{IG}_{\infty}$) inside $P_{\opn{dR}, r}^{\opn{an}}$ (resp. $p_1^{-1}P_{\opn{dR}, r}^{\opn{an}}$) then we have a commutative diagram:
\[
\begin{tikzcd}
\mathcal{Z} \arrow[d]              & \mathcal{Z}' \arrow[l, "p_1"'] \arrow[r, "q_2"] \arrow[d]                 & \mathcal{Z} \arrow[d]              \\
P_{\opn{dR}, r}^{\opn{an}} & p_1^{-1}P_{\opn{dR}, r}^{\opn{an}} \arrow[l, "p_1"'] \arrow[r, "q_2"] & P_{\opn{dR}, r+1}^{\opn{an}}
\end{tikzcd}
\]

\begin{lemma}
The left-hand side square above is Cartesian (i.e. $p_1^{-1}(\mathcal{Z}) = \mathcal{Z}'$).
\end{lemma}
\begin{proof}
Clearly $\mathcal{Z}' \subset p_1^{-1}(\mathcal{Z})$. Now $p_1$ is finite flat hence it is an open morphism. Suppose that $x \in p_1^{-1}(\mathcal{Z}) - \mathcal{Z}'$. Set $U = p_1^{-1}P_{\opn{dR}, r}^{\opn{an}} - \mathcal{Z}'$ which is open. Then $p_1(U)$ is open and $p_1(x) \in p_1(U) \cap \mathcal{Z}$. Since $p_1(U)$ is open, this means we must have $p_1(U) \cap \mathcal{IG}_{\infty} \neq \varnothing$. But this implies that
\[
U \cap p_1^{-1}(\mathcal{IG}_{\infty}) = U \cap p_1^{-1}\mathcal{IG}_{\infty} \neq \varnothing 
\]
which is a contradiction.
\end{proof}

We now have the following construction of the $U_p$-operator. Thanks to the above lemma, for any quasi-compact open $\mathcal{Z} \subset U \subset P_{\opn{dR}, r+1}^{\opn{an}}$, there exists a quasi-compact open $\mathcal{Z} \subset U' \subset P_{\opn{dR}, r}^{\opn{an}}$ such that $q_2^{-1}(U) \supset p_1^{-1}(U')$. Therefore, we get maps
\[
\mathcal{O}_{P_{\opn{dR}, r+1}^{\opn{an}}}(U) \xrightarrow{q_2^*} \mathcal{O}_{p_1^{-1}P_{\opn{dR}, r}^{\opn{an}}}(q_2^{-1}(U)) \xrightarrow{\opn{res}} \mathcal{O}_{p_1^{-1}P_{\opn{dR}, r}^{\opn{an}}}(p_1^{-1}(U')) \xrightarrow{\opn{Tr}_{p_1}} \mathcal{O}_{P_{\opn{dR}, r}^{\opn{an}}}(U') 
\]
where $\opn{Tr}_{p_1}$ denotes the trace map associated with $p_1$. All of these maps are compatible, so we obtain an induced map on the limit, namely the $U_p$-operator $U_p \colon \mathscr{N}^{\dagger} \to \mathscr{N}^{\dagger}$.

\begin{remark}
Exactly the same construction works for $\mathscr{M}^{\dagger}$ and the $U_p$-operators are compatible under the natural map $\mathscr{M}^{\dagger} \to \mathscr{N}^{\dagger}$.
\end{remark}

\begin{remark}
    One also has similar constructions for the Frobenius operator $\varphi$ and diamond operator $S_p$, which are defined with respect to the action of $\left( \begin{smallmatrix} p^{-1} & \\ & 1 \end{smallmatrix} \right)$ and $\left( \begin{smallmatrix} p^{-1} & \\ & p^{-1} \end{smallmatrix} \right)$ respectively. One can easily show that $U_p \circ \varphi = p S_p$ and $S_p$ commutes with $U_p$ and $\varphi$. We leave the details to the reader.
\end{remark}

\subsubsection{Compatibility with the group action}

Recall that we have a cofinal system $\{U_i \}$ of neighbourhoods of the closure of $\mathcal{IG}_{\infty} \to P_{\opn{dR}, r}^{\opn{an}}$ satisfying $T(\mbb{Z}_p) \cdot U_i \subset U_i$. Here we are considering the action $T^{\opn{an}} \times P_{\opn{dR}, r}^{\opn{an}} \to P_{\opn{dR}, r}^{\opn{an}}$ through the inclusion $T^{\opn{an}} \subset \overline{P}^{\opn{an}}$.

\begin{lemma}
The $T(\mbb{Z}_p)$-action commutes with $U_p$, $\varphi$ and $S_p$ on $\mathscr{M}^{\dagger}$ and $\mathscr{N}^{\dagger}$.
\end{lemma}
\begin{proof}
For $U_p$, we just need to show that the group action commutes with $p_1$ and $q_2$. But this is clear (since $T(\mbb{Z}_p)$ commutes with $\opn{diag}(1, p^{-1})$). The proofs of the claims for $\varphi$ and $S_p$ are identical.
\end{proof}

\subsubsection{Compatibility with the connection}

Let $r \geq 1$ be an integer. By the construction above, one has a map $R$-modules
\[
U_p \colon \opn{H}^0\left( P_{\opn{dR}, r}^{\opn{an}}, \mathcal{O}_{P_{\opn{dR}, r}^{\opn{an}}} \right) \xrightarrow{\opn{res}} \opn{H}^0\left( P_{\opn{dR}, r+1}^{\opn{an}}, \mathcal{O}_{P_{\opn{dR}, r+1}^{\opn{an}}} \right) \to \opn{H}^0\left( P_{\opn{dR}, r}^{\opn{an}}, \mathcal{O}_{P_{\opn{dR}, r}^{\opn{an}}} \right) .
\]
We have the following lemma.

\begin{lemma} \label{ClaHeckeIntertwiningLemma}
For any $\phi \in C^{\opn{la}}(\mbb{Z}_p, R)$, one has 
\begin{itemize}
    \item $\phi(p \cdot -) \circ U_p = U_p \circ \phi$
    \item $\phi \circ \varphi = \varphi \circ \phi(p \cdot -)$
    \item $\phi \circ S_p = S_p \circ \phi$
\end{itemize}
as endomorphisms of $\mathscr{N}^{\dagger}$.
\end{lemma}
\begin{proof}
By the density of polynomial functions in $C^{\opn{la}}(\mbb{Z}_p, R)$, it suffices to check this for $\phi$ equal to the structural map $\mbb{Z}_p \hookrightarrow R$ (i.e. for the operator $\nabla$). The functoriality results in \cite[\S 6.2]{AI_LLL} imply that $p \nabla \circ U_p = U_p \circ \nabla$  on $\opn{H}^0\left( P_{\opn{dR}, r}^{\opn{an}}, \mathcal{O}_{P_{\opn{dR}, r}^{\opn{an}}} \right)$. Furthermore $U_p \circ \nabla$ and $p\nabla \circ U_p$ are two continuous $R$-linear morphisms $\mathscr{N}^{\dagger} \to \mathscr{N}^{\dagger}$ extending $U_p \circ \nabla = p\nabla \circ U_p$ on $\opn{H}^0\left( P_{\opn{dR}, r}^{\opn{an}}, \mathcal{O}_{P_{\opn{dR}, r}^{\opn{an}}} \right)$ for any $r \geq 1$. But any such extension must be unique (because $\opn{H}^0\left( P_{\opn{dR}, r}^{\opn{an}}, \mathcal{O}_{P_{\opn{dR}, r}^{\opn{an}}} \right)$ is dense in $\mathscr{N}^{\dagger}_{U}$ for any quasi-compact strict open neighbourhood $U$ of the Igusa tower). This proves the first bullet point. The second and third bullet points are similar (again following from the functoriality results in \cite[\S 6.2]{AI_LLL} and a density argument).
\end{proof}

\subsubsection{A reinterpretation}

We can also reinterpret the construction of $U_p$ using the torsors constructed in Proposition \ref{ReductionStructurePdRProposition}. We begin with the following lemma:

\begin{lemma} \label{PhiRestrictsToIsoLemma}
    Let $r \geq 1$ be an integer. Then there exists an integer $s \geq r$ such that
    \[
    p_2^{-1}\tilde{\mathcal{F}}_{r+1, s+1} \subset \phi\left(  p_1^{-1}\tilde{\mathcal{F}}_{r, s} \right) \subset p_2^{-1}\tilde{\mathcal{F}}_{r, s+1}
    \]
    where $\phi \colon p_1^{-1}P_{\opn{dR}, s}^{\opn{an}} \xrightarrow{\sim} p_2^{-1}P_{\opn{dR}, s+1}^{\opn{an}}$ is the morphism in Definition \ref{DefinitionOfPhiTopCorr}.
\end{lemma}
\begin{proof}
    Let $s \geq r$ be an integer such that $\tilde{\mathcal{F}}_{r, s}$ and $\tilde{\mathcal{F}}_{r+1, s+1}$ exist, and recall that the map $p_2$ is an isomorphism. Let $\mathscr{G} = \phi(p_1^{-1}\tilde{\mathcal{F}}_{r, s})$, which is an \'{e}tale torsor under the group
    \[
    t^{-1} \overline{\mathcal{P}}_r t \subset \overline{\mathcal{P}}_{r} , \quad \quad t = \tbyt{1}{}{}{p^{-1}} .
    \]
    Then we see that the pushout $\mathscr{G} \times^{t^{-1} \overline{\mathcal{P}}_r t} \overline{\mathcal{P}}_{r}$ is an \'{e}tale $\overline{\mathcal{P}}_{r}$-torsor which is a reduction of structure of $p_2^{-1}P_{\opn{dR}, s+1}^{\opn{an}}$. Furthermore, $p_2^{-1}\mathcal{IG}_{\infty}$ defines a reduction of structure of $\mathscr{G}$. Since $p_2$ is an isomorphism, we see from Proposition \ref{ReductionStructurePdRProposition} that we can increase $s$ (if necessary) so that $\mathscr{G} \times^{t^{-1} \overline{\mathcal{P}}_r t} \overline{\mathcal{P}}_{r}$ coincides with $p_2^{-1}\tilde{\mathcal{F}}_{r, s+1}$, and $p_2^{-1}\tilde{\mathcal{F}}_{r+1, s+1}$ is a reduction of structure of $\mathscr{G}$ via the embedding $\overline{\mathcal{P}}_{r+1} \subset t^{-1} \overline{\mathcal{P}}_r t$.
\end{proof}

Let $\kappa \colon T(\mbb{Z}_p) \to (R^+)^{\times}$ be an $r$-analytic character, and let $s \geq r$ be a sufficiently large integer such that the conclusion of Lemma \ref{PhiRestrictsToIsoLemma} holds. Let 
\[
\pi_1 \colon p_1^{-1}\tilde{\mathcal{F}}_{r, s} \to \mathcal{C}_s, \quad \quad \pi_2 \colon p_2^{-1}\tilde{\mathcal{F}}_{r+1, s+1} \to \mathcal{C}_s, \quad \quad \sigma \colon \mathscr{G} \to \mathcal{C}_s
\]
denote the structural maps of the torsors, where $\mathscr{G}$ is as in Lemma \ref{PhiRestrictsToIsoLemma}. Let $t = \left( \begin{smallmatrix} 1 & \\ & p^{-1} \end{smallmatrix} \right) \in T^{++}$. Then the morphism $\phi^* \otimes (t \star_l - )$ induces a cohomological correspondence
\[
\phi_{\kappa} \colon p_2^* \mathcal{V}_{r, s+1, \kappa} = \left( (\sigma)_* \mathcal{O}_{\mathscr{G}} \hatot V_{r, \kappa} \right)^{t^{-1} \overline{\mathcal{P}}_r t, \star_l} \to \left( (\pi_1)_* \mathcal{O}_{p_1^{-1}\tilde{\mathcal{F}}_{r, s}} \hatot V_{r, \kappa} \right)^{\overline{\mathcal{P}}_{r}, \star_l} = p_1^* \mathcal{V}_{r, s, \kappa} 
\]
where the first equality holds because $V_{r, \kappa}$ is a representation of $\overline{\mathcal{P}}_r$ and $\mathscr{G} \times^{t^{-1} \overline{\mathcal{P}}_r t} \overline{\mathcal{P}}_r = p_2^{-1}\tilde{\mathcal{F}}_{r, s+1}$. The $U_p$-operator can then be seen as the following composition:
\begin{align*}
    \opn{H}^0\left( \mathcal{X}_s, \mathcal{V}_{r, s, \kappa} \right) &\xrightarrow{\opn{res}} \opn{H}^0\left( \mathcal{X}_{s+1}, \mathcal{V}_{r, s+1, \kappa} \right) \\ 
     &\xrightarrow{p_2^*} \opn{H}^0\left( \mathcal{C}_{s}, p_2^*\mathcal{V}_{r, s+1, \kappa} \right) \\
     &\xrightarrow{\phi_{\kappa}} \opn{H}^0\left( \mathcal{C}_{s}, p_1^*\mathcal{V}_{r, s, \kappa} \right) \\
     &\xrightarrow{\opn{Tr}_{p_1}} \opn{H}^0\left( \mathcal{X}_s, \mathcal{V}_{r, s, \kappa} \right)
\end{align*}
where the first map is induced from restriction. These operators are compatible as one varies $r, s$, and the resulting operator on the limit $\mathscr{N}^{\dagger}_{\kappa} = \varinjlim_{r, s} \opn{H}^0\left( \mathcal{X}_s, \mathcal{V}_{r, s, \kappa} \right)$ is precisely the $U_p$-operator constructed in the previous sections. For $s \geq r+1$ sufficiently large, we have a factorisation:
\begin{equation} \label{FactorisationOfUpEqn}
\begin{tikzcd}
{\opn{H}^0\left( \mathcal{X}_{s+1}, \mathcal{V}_{r+1, s+1, \kappa} \right)} \arrow[r, "U_p"] \arrow[rd] & {\opn{H}^0\left( \mathcal{X}_{s+1}, \mathcal{V}_{r+1, s+1, \kappa} \right)}   \\
{\opn{H}^0\left( \mathcal{X}_s, \mathcal{V}_{r, s, \kappa} \right)} \arrow[r, "U_p"] \arrow[u]          & {\opn{H}^0\left( \mathcal{X}_s, \mathcal{V}_{r, s, \kappa} \right)} \arrow[u]
\end{tikzcd}
\end{equation}
because the cohomological correspondence $p_2^*\mathcal{V}_{r+1, s+1, \kappa} \xrightarrow{\phi_{\kappa}} p_1^*\mathcal{V}_{r+1, s, \kappa}$ factors through $p_1^* \mathcal{V}_{r, s, \kappa} \subset p_1^*\mathcal{V}_{r+1, s, \kappa}$ (see Remark \ref{CompactOperatorOnVrkappaRem}).

\begin{remark}
    Using this reinterpretation, one can easily see that $U_p$ preserves the filtration $\opn{Fil}_{\bullet}\mathscr{N}^{\dagger}_{\kappa}$.
\end{remark}

\subsubsection{Slope decompositions}

We now discuss the spectral theory of the $U_p$-operator on $\mathscr{N}^{\dagger}_{\kappa}$. Let $r \geq 1$ be an integer. Let $R_0^+$ be an admissible $\mbb{Z}_p$-algebra, $R_0 = R_0^+[1/p]$ and consider the affinoid adic space $\mathcal{W} \defeq \opn{Spa}(R_0, R_0^+)$. Suppose that we have an $r$-analytic analytic character $\kappa \colon T(\mbb{Z}_p) \to R_0^{\times}$. We have the following:

\begin{theorem} \label{SlopeDecompositionForNdaggerThm}
Let $x \in \mathcal{W}(\overline{\mbb{Q}}_p)$ and $h \in \mbb{Q}$. Then there exists an open affinoid neighbourhood $\opn{Spa}(R, R^+) \subset \mathcal{W}$ of $x$ such that $\mathscr{N}^{\dagger}_{\kappa}$ and $\opn{Fil}_{\bullet}\mathscr{N}_{\kappa}^{\dagger}$ have slope $\leq h$ decompositions with respect to $U_p$. Furthermore, there exists an integer $k$ such that
\[
(\mathscr{N}^{\dagger}_{\kappa})^{\leq h} = (\opn{Fil}_{k}\mathscr{N}^{\dagger}_{\kappa})^{\leq h} .
\]
\end{theorem}
\begin{proof}
    Since the inclusion $V_{r, \kappa} \to V_{r+1, \kappa}$ is compact and $\mathcal{X}_{s+1} \Subset \mathcal{X}_s$ is a strict inclusion, the factorisation in (\ref{FactorisationOfUpEqn}) (and the standard theory of slope decompositions) implies that $\mathscr{L}_{\kappa} \defeq \varinjlim_s\opn{H}^0\left(\mathcal{X}_s, \mathcal{V}_{r, s, \kappa} \right)$ admits a slope $\leq h$ decomposition over a neighbourhood $\opn{Spa}(R, R^+)$ of $x$ in $\mathcal{W}$. Furthermore, one can check that as $k \to +\infty$, the norm of the operator
    \[
    (t \star_l - ) \colon \opn{Gr}_k V_{r, \kappa} \to \opn{Gr}_k V_{r, \kappa}
    \]
    tends to zero. Hence there exists an integer $k \geq 0$ such that $\mathscr{L}_{\kappa}^{\leq h} = (\opn{Fil}_{k}\mathscr{N}^{\dagger}_{\kappa})^{\leq h}$. To conclude the proof, we note that the factorisation (\ref{FactorisationOfUpEqn}) implies that $U_p$ is pointwise nilpotent on $\mathscr{N}^{\dagger}_{\kappa}/\mathscr{L}_{\kappa}$, so $\mathscr{N}^{\dagger}_{\kappa}$ admits a slope $\leq h$ decomposition and $(\mathscr{N}^{\dagger}_{\kappa})^{\leq h} = \mathscr{L}_{\kappa}^{\leq h}$.
\end{proof}

\subsection{Overconvergent projectors}

We now have all the ingredients to define an overconvergent projector on $\mathscr{N}^{\dagger}_{\kappa}$. As in the previous subsection, let $R_0^+$ be an admissible $\mbb{Z}_p$-algebra, $R_0 = R_0^+[1/p]$ and set $\mathcal{W} = \opn{Spa}(R_0, R_0^+)$. Suppose that we have a locally analytic character $\kappa \colon T(\mbb{Z}_p) \to R_0^{\times}$.

\begin{theorem} \label{OCProjectorForNdaggerThm}
Let $x \in \mathcal{W}(\overline{\mbb{Q}}_p)$ and $h \in \mbb{Q}$. Then there exists an open affinoid neighbourhood $\opn{Spa}(R, R^+) \subset \mathcal{W}$ containing $x$ and an overconvergent projector
\[
\Pi^{\leq h, \opn{oc}} \colon \mathscr{N}^{\dagger}_{\kappa} \to \left(\mathscr{M}^{\dagger}_{\kappa} \right)^{\leq h} \otimes_R \opn{Frac}(R)
\]
with finitely-many poles, interpolating the overconvergent projectors in \cite[\S 3.3.4]{Urban} for classical specialisations.
\end{theorem}
\begin{proof}
By Theorem \ref{SlopeDecompositionForNdaggerThm}, it is enough to construct an overconvergent projector for $\left(\opn{Fil}_k \mathscr{N}^{\dagger}_{\kappa} \right)^{\leq h}$, and this follows from exactly the same arguments as in \cite[\S 3.9]{AI_LLL} (using Lemma \ref{NablaBarIsoUpToDenominators}).
\end{proof}

\subsection{\texorpdfstring{$p$}{p}-depletion}

In this section we discuss the relation between the action of $C^{\opn{la}}(\mbb{Z}_p, R)$ and $p$-depletion. We begin with the following lemma.

\begin{lemma} \label{OrdPdepletionLemma}
    The action of $1_{\mbb{Z}_p^{\times}} \in C_{\opn{cont}}(\mbb{Z}_p, R)$ coincides with the operator $1 - p^{-1}S_p^{-1}\varphi U_p$ on $p$-adic modular forms $\mathscr{M}$.  
\end{lemma}
\begin{proof}
    Without loss of generality, we may assume that $R = \mbb{Q}_p^{\opn{cycl}}$. The space $\mathscr{M}$ can be described as sections of $\mathcal{IG}_{K^p}$ invariant under right-translation by $U_{P'}^{\opn{int}}$. Then, by $p$-adic Fourier theory (see \cite[\S 7.1]{howe2020unipotent}), the action of $1_{p\mbb{Z}_p}$ is given by
    \[
    f \mapsto \frac{1}{p} \sum_{\zeta \in \mu_p} \tbyt{1}{\tilde{\zeta}}{}{1} \cdot f
    \]
    for $f \in \mathscr{M}$, where $\tilde{\zeta} \in \tilde{\mu_{p^{\infty}}}$ denotes any lift of $\zeta$. Here $\left(\begin{smallmatrix} 1 & \tilde{\zeta} \\ & 1 \end{smallmatrix}\right)$ is an element of $J_{\opn{ord}}(\mbb{Q}_p^{\opn{cycl}})$ acting on a function by right-translation of the argument. On the other hand, we have
    \begin{align*}
        (p^{-1}S_p^{-1}U_p)(f) &= \frac{1}{p} \sum_{\omega \in T_p(\mu_{p^{\infty}})/p T_p(\mu_{p^{\infty}})} \tbyt{p}{\omega}{}{1} \cdot f \\
        \varphi(f) &= \tbyt{p^{-1}}{}{}{1} \cdot f .
    \end{align*}
    Hence $p^{-1}S_p^{-1}\varphi U_p$ coincides with the action of $1_{p\mbb{Z}_p}$ and the result follows.
\end{proof}

We now prove an overconvergent version of this lemma.

\begin{proposition}
    The action of $1_{\mbb{Z}_p^{\times}} \in C^{\opn{la}}(\mbb{Z}_p, R)$ coincides with the operator $1 - p^{-1}S_p^{-1}\varphi U_p$ on nearly overconvergent modular forms $\mathscr{N}^{\dagger}$.
\end{proposition}
\begin{proof}
    We first claim that the action of $1_{\mbb{Z}_p^{\times}}$ preserves the filtration $\opn{Fil}_{\bullet}\mathscr{N}^{\dagger}$. This can be checked locally on sections of overconvergent extensions of some $\mathcal{U}_{\opn{HT}, n}$, for $n$ sufficiently large (see \S \ref{OrdinaryNHoodsSubSec}--\ref{OverconvergentNBHDSsection}). If $\mathcal{U}$ is such an overconvergent extension, then we have
    \[
    \opn{H}^0(\mathcal{U}, \mathcal{O}_{\mathcal{U}}) \hookrightarrow \opn{H}^0(\mathcal{U}_{\opn{HT}, n, A_{\opn{ord}, \infty}}, \mathcal{O}_{\mathcal{U}_{\opn{HT}, n, A_{\opn{ord}, \infty}}} ) = \bigoplus_{\lambda \in T(\mbb{Z}/p^n\mbb{Z})} \opn{H}^0\left( \mathcal{B}_{\infty}(\sigma_{\lambda}(g_{\infty}), p^{-n}), \mathcal{O}_{\mathcal{B}_{\infty}(\sigma_{\lambda}(g_{\infty}), p^{-n})} \right)
    \]
    respecting filtrations (it is equivariant for the action of $\overline{\mathfrak{n}}$), so it suffices to check the claim for each direct summand in the right-hand side. But the action of $1_{\mbb{Z}_p^{\times}}$ on such a direct summand is given by $\lim_{m \to +\infty} (\nabla^{\lambda})^{p^{m-1}(p-1)}$, and one can see from the explicit description of $\nabla^{\lambda}$ in the proof of Proposition \ref{PropExplicitForm} that this operator preserves the filtration (it acts as $1_{\mbb{Z}_p^{\times}}$ on the $A_{\opn{ord}, \infty}$-coefficients of the power series in the coordinates $X, Y, Z$).

    Since $\bigcup_{r \geq 0} \opn{Fil}_r \mathscr{N}^{\dagger}$ is dense in $\mathscr{N}^{\dagger}$ it therefore suffices to show that the operator $T \defeq 1 - p^{-1}S_p^{-1}\varphi U_p - 1_{\mbb{Z}_p^{\times}} = 1_{p\mbb{Z}_p} - p^{-1}S_p^{-1}\varphi U_p$ is zero on $\opn{Fil}_r\mathscr{N}^{\dagger}$, for any integer $r \geq 0$. But for any integer $r \geq 0$, we have a morphism
    \[
    \opn{Fil}_r\mathscr{N}^{\dagger} \twoheadrightarrow \opn{Gr}_r \mathscr{N}^{\dagger} \hookrightarrow \mathscr{M}
    \]
    which is equivariant for the action of $T$. Hence Lemma \ref{OrdPdepletionLemma} implies that $T$ is zero on the graded pieces of $\opn{Fil}_r\mathscr{N}^{\dagger}$ -- in particular, we must have $T^{r+1} = 0$ on $\opn{Fil}_r\mathscr{N}^{\dagger}$. Finally, we note that $T$ is idempotent. Indeed
    \[
    T^2 = 1_{p\mbb{Z}_p} - p^{-1}S_p^{-1}\varphi U_p 1_{p \mbb{Z}_p} - p^{-1} 1_{p\mbb{Z}_p} S_p^{-1}\varphi U_p + p^{-1}S_p^{-1}\varphi U_p = 1_{p\mbb{Z}_p} - p^{-1}S_p^{-1}\varphi U_p = T  
    \]
    where we have used the fact that $p^{-1}S_p^{-1}\varphi U_p$ is idempotent (because $U_p \varphi = p S_p$) and the relations in Lemma \ref{ClaHeckeIntertwiningLemma}. This proves that $T = T^{r+1}$ is zero on $\opn{Fil}_r \mathscr{N}^{\dagger}$ as required. 
\end{proof}

\subsection{Comparison with the work of Andreatta--Iovita}

We end this section by comparing the space of nearly overconvergent forms constructed in this article with the space constructed in \cite{AI_LLL}. For this, consider the following group
\[
\overline{\mathcal{P}}_r^{\opn{AI}} \defeq T(\mbb{Z}_p) \cdot \{ x \in \overline{\mathcal{P}} : m(x) \equiv 1 \text{ modulo } p^{r+1-\frac{1}{p-1}} \}
\]
where $m \colon \overline{\mathcal{P}} \to \mathcal{T}$ denotes the projection to the torus. With notation as in Proposition \ref{ReductionStructurePdRProposition}, consider the pushout $\tilde{\mathcal{F}}_{r, s}^{\opn{AI}} \defeq \tilde{\mathcal{F}}_{r, s} \times^{\overline{\mathcal{P}}_r} \overline{\mathcal{P}}_r^{\opn{AI}}$. Set $V_r^{\opn{AI}} = \mathcal{O}(\overline{\mathcal{P}}_r^{\opn{AI}})$ which carries an action of $\Sigma_r^{\opn{AI}} = \overline{\mathcal{P}}_r^{\opn{AI}} \cdot T^+ \cdot \overline{\mathcal{P}}_r^{\opn{AI}}$ by the exact same formulae as in Definition \ref{DefinitionOfVrwithmonoidaction}. If $\kappa \colon T(\mbb{Z}_p) \to R^{\times}$ is an $r$-analytic character, then we set 
\[
V_{r, \kappa}^{\opn{AI}} \defeq \opn{Hom}_{T(\mbb{Z}_p)}\left( - w_0\kappa, V_r^{\opn{AI}} \hatot R \right) .
\]
We let $\mathcal{V}_{r, s, \kappa}^{\opn{AI}} = \left( \pi_* \mathcal{O}_{\tilde{\mathcal{F}}_{r, s}^{\opn{AI}}} \hatot V_{r, \kappa}^{\opn{AI}} \right)^{\overline{\mathcal{P}}_r^{\opn{AI}}}$ denote the corresponding locally projective Banach sheaf on $\mathcal{X}_s$, where $\pi \colon \tilde{\mathcal{F}}_{r, s}^{\opn{AI}} \to \mathcal{X}_s$ denotes the structural map. Note that $\mathcal{V}_{r, s, \kappa}^{\opn{AI}}$ is independent of the choice of $r$ such that $\kappa$ is $r$-analytic (whereas the sheaf $\mathcal{V}_{r, s, \kappa}$ from Definition \ref{DefinitionOfVrwithmonoidaction} very much depends on $r$). 

Set $\mathscr{N}^{\dagger,\opn{AI}}_{\kappa} \defeq \varinjlim_s \opn{H}^0\left( \mathcal{X}_s, \mathcal{V}_{r, s, \kappa}^{\opn{AI}} \right)$, where the colimit is over the restriction maps. This space carries actions of $U_p$, $S_p$ and $\varphi$ by exactly the process as in \S \ref{TheUpOperatorSSec}, and $\mathscr{N}^{\dagger, \opn{AI}}_{\kappa}$ admits slope decompositions with respect to $U_p$ (up to possibly shrinking $\opn{Spa}(R, R^+)$). We have the following comparison result.

\begin{proposition} \label{PropComparisonWithAI}
    Let $I = [0, p^{b}]$ with $b \neq \infty$, and suppose that $k$ is the universal character of the open $\mathcal{W}_I$ of weight space introduced in \cite[p.2004]{AI_LLL}. Let $\kappa = (k; w) \colon T(\mbb{Z}_p) \to \mathcal{O}(\mathcal{W}_I)^{\times}$ denote any locally analytic character extending $k$. Then:
    \begin{enumerate}
        \item For $s$ sufficiently large, the sheaf $\mathcal{V}_{r, s, \kappa}^{\opn{AI}}$ is identified with the sheaf $\mbb{W}_{k, I}[1/p]$ of nearly overconvergent forms over $\mathcal{X}_s$ in \cite[\S 3.3]{AI_LLL}.
        \item The actions of $p^{-1}U_p$ and $S_p^{-1} \varphi$ on 
    \begin{equation} \label{AIidentificationEqn}
    \mathscr{N}^{\dagger, \opn{AI}}_{\kappa} = \varinjlim_s \opn{H}^0\left( \mathcal{X}_s, \mbb{W}_{k, I}[1/p] \right)
    \end{equation}
    coincide with the actions of $U$ and $V$ constructed in \cite[\S 3.6--\S 3.7]{AI_LLL}.
    \item The Gauss--Manin connection $\nabla \colon \mathscr{N}^{\dagger, \opn{AI}}_{\kappa} \to \mathscr{N}^{\dagger, \opn{AI}}_{\kappa + 2\rho}$ intertwines with the connection constructed in \cite[\S 3.4]{AI_LLL} under the identifications in (\ref{AIidentificationEqn}).
    \end{enumerate}
\end{proposition}
\begin{proof}
Let $\mbb{V}_0(\opn{H}^1_{\opn{dR}}(E)^{\#}, s_{\opn{can}}) \to \mathcal{X}_s$ denote the adic generic fibre of the formal vector bundle with marked section as constructed in \cite[\S 3.3]{AI_LLL}. Here, we use the notation $\Omega_E \subset \opn{H}^1_{\opn{dR}}(E)^{\#}$ (resp. $s_{\opn{can}} \in \Omega_E/p^{r+1-1/(p-1)}$) in place of the notation $\Omega_E \subset \opn{H}_E^{\#}$ (resp. $s \in \Omega_E/p^{r+1-1/(p-1)}$) used in \emph{loc.cit.} in order to avoid clashes with notation used previously in this article. The space $\mbb{V}_0(\opn{H}^1_{\opn{dR}}(E)^{\#}, s_{\opn{can}})$ carries a left action of the group $\mathcal{T}_r' \defeq \mbb{Z}_p^{\times} \left( 1 + p^{r+1-1/(p-1)}\mbb{G}_a^+ \right)$; since this group is abelian, we can (and do) view this as a right action.

Over $V \defeq \mbb{V}_0(\opn{H}^1_{\opn{dR}}(E)^{\#}, s_{\opn{can}})$, one has a universal morphism $\rho \colon \opn{H}^1_{\opn{dR}}(E)^{\#} \to \mathcal{O}^+_V$ which restricts to an isomorphism $\Omega_E \xrightarrow{\sim} \mathcal{O}^+_V$ mapping the marked section $s_{\opn{can}}$ to $1$ modulo $p^{r+1-1/(p-1)}$. In particular, the kernel $\mathcal{U}$ of the universal morphism $\rho$ induces a decomposition $\opn{H}^1_{\opn{dR}}(E)^{\#} = \Omega_E \oplus \mathcal{U}$ (and $\mathcal{U}$ is a locally free $\mathcal{O}^+_V$-module of rank one). We let $W \to V$ denote the (right) $\mbb{G}^+_m$-torsor parameterising bases $\{ e, f\}$ of $\opn{H}^1_{\opn{dR}}(E)^{\#}$ with $e \in \Omega_E$, $f \in \mathcal{U}$, and $e$ mapping to $1$ under the universal morphism $\rho$. Here $\mbb{G}_m^+(\opn{Spa}(R, R^+)) = (R^+)^{\times}$ and it acts on the basis by rescaling the vector $f$ by the \emph{inverse} of the element in $\mbb{G}_m^+$. Furthermore, the map $W \to V$ is $\mathcal{T}_r'$-equivariant, where the action on the source is given by rescaling $e$ by the inverse of the element in $\mathcal{T}_r'$. Finally, we equip $W$ with an action of the group 
\[
\overline{\mathcal{P}}_r' \defeq \tbyt{\mathcal{T}_r'}{}{\mbb{G}_a^+}{\mbb{G}_m^+}
\]
by declaring $g \cdot \{e, f \} = \{ a e, b e + c f\}$ for ${^tg^{-1}} = \tbyt{a}{b}{}{c}$. Note that $W \to \mathcal{X}_s$ is a $\overline{\mathcal{P}}_r'$-torsor. By using the fact that $\mathcal{H}_{\mathcal{E}} \cong \left(\opn{H}^1_{\opn{dR}}(E)^{\#}[1/p]\right)^{\vee}$ and the fact that $\opn{H}^1_{\opn{dR}}(E)^{\#}$ has a canonical decomposition $\Omega_E \oplus \mathcal{U}^{\opn{can}}$ (where $\mathcal{U}^{\opn{can}}$ denotes the unit root subsheaf) over the ordinary locus, one has morphisms
\[
\mathcal{IG} \hookrightarrow W \hookrightarrow P_{\opn{dR}}^{\opn{an}}
\]
with both maps defining reductions of structure. Therefore, by a similar argument as in Proposition \ref{ReductionStructurePdRProposition}, one can identify $W$ with $\tilde{\mathcal{F}}_{r, s}^{\opn{AI}} \times^{\overline{\mathcal{P}}_r^{\opn{AI}}} \overline{\mathcal{P}}_r'$ after possibly increasing $s$. We conclude that there are natural isomorphisms
\[
\mathcal{V}_{r, s, (k;-k)}^{\opn{AI}} \cong \opn{Hom}_{\mathcal{T}_r}((k;0), \pi_{W, *}\mathcal{O}_W ) \cong \opn{Hom}_{\mathcal{T}_r'}(k, \pi_{V, *} \mathcal{O}_V) \cong \mbb{W}_{k, I}[1/p] .
\]
Here $\pi_W \colon W \to \mathcal{X}_s$ and $\pi_V \colon V \to \mathcal{X}_s$ denote the structural maps; and we have used the fact that $-w_0 (k; -k) = (k; 0)$. 

The result for general $\kappa = (k; w)$ follows from the identification $\mathcal{V}_{r, s, \kappa}^{\opn{AI}} \cong \mathcal{V}_{r, s, (k;-k)}^{\opn{AI}}$ arising from identifying $E$ and $E^D$ via the principal polarisation. Parts (2) and (3) follow from tracing through the definitions and constructions in \S \ref{TheGMConnectionSubSec}, \S \ref{TheUpOperatorSSec} and \cite[\S 3.4, \S 3.6--\S 3.7]{AI_LLL} (again, identifying $E$ with its dual $E^D$).
\end{proof}

If we set $\mathscr{N}^{\dagger, \opn{AI}} \defeq \varinjlim_{r, s} \opn{H}^0\left( \tilde{\mathcal{F}}_{r, s}^{\opn{AI}}, \mathcal{O}_{\tilde{\mathcal{F}}_{r, s}^{\opn{AI}}} \right)$, then we see that we have an inclusion
\begin{equation} \label{AItoNDaggerEqn}
\mathscr{N}^{\dagger, \opn{AI}} \subset \mathscr{N}^{\dagger}
\end{equation}
and the inclusion $\mathscr{N}_{\kappa}^{\dagger, \opn{AI}} \subset \mathscr{N}^{\dagger}_{\kappa}$ is an isomorphism on finite-slope parts (with respect to the action of $U_p$). The space $\mathscr{N}^{\dagger}$ is much larger than $\mathscr{N}^{\dagger, \opn{AI}}$ though; the difference between the two spaces in (\ref{AItoNDaggerEqn}) is that $\mathscr{N}^{\dagger}$ incorporates ``congruences between the subspace $\mathcal{U}$ in the proof of Proposition \ref{PropComparisonWithAI} and the unit root subsheaf $\mathcal{U}^{\opn{can}}$'' -- roughly speaking, when restricting to the Igusa tower, the space $\mathscr{N}^{\dagger}$ locally looks like the colimit over $n$ of a number of copies of $A_{\opn{ord}, \infty}^+ \langle \frac{X-1}{p^n}, \frac{Y}{p^n}, \frac{Z-1}{p^n} \rangle $ (see Proposition \ref{PropExplicitForm}), whereas $\mathscr{N}^{\dagger, \opn{AI}}$ looks like the colimit over $n$ of a number of copies of $A_{\opn{ord}, \infty}^+ \langle \frac{X-1}{p^n}, Y, \frac{Z-1}{p^n} \rangle $. This is the key difference which allows us to extend $\nabla$ to an action of $C^{\opn{la}}(\mbb{Z}_p, \mbb{Q}_p)$. 

It does not seem possible to improve the results of \cite{AI_LLL} without incorporating such congruences with the unit root subsheaf. More precisely, as indicated in \cite[Remark 3.39]{AI_LLL}, in order to define an action of $\nabla^s$ on $\mathscr{N}_{\kappa}^{\dagger, \opn{AI}}$ for a locally analytic character $s$ of $\mbb{Z}_p^{\times}$, it is necessary to assume that $s$ and $\kappa$ are analytic up to finite-order twist when restricted to $1 + p\mbb{Z}_p$ (in the language of \emph{op.cit.}, this condition is written as $u_s \in \Lambda_{I_s}^0$ and $u_k \in \Lambda_I^0$). The reason for this is due to the presence of the ``unbounded denominators'' in the displayed equation of \cite[Remark 3.39]{AI_LLL}. In our definition of $\mathscr{N}^{\dagger}$, we allow the coordinate ``$V = Y(1+pZ)^{-1}$'' (with notation as in \emph{loc.cit.}) to be arbitrarily divisible by $p$, which cancels out these problematic denominators.\footnote{One could also try to bound the action of $\binom{\nabla}{k}$ on $\mathscr{N}^{\dagger, \opn{AI}}$ in a similar way as in this article (by the formula in Lemma \ref{Lemmaformulatheta} for example), however one runs into the same issue with ``unbounded denominators'' (because the local coordinate $Y$ is not arbitrarily divisible by powers of $p$ in the definition of $\mathscr{N}^{\dagger, \opn{AI}}$).} 

\begin{remark}
    Let $\widehat{\mathcal{N}_{\mathfrak{U}}^{\infty, \rho}}(N)$ denote the space of nearly overconvergent modular forms defined in \cite[\S 3.3.2]{Urban} (with notation as in \emph{loc.cit.}). As explained in \cite[\S B.2, p.2077]{AI_LLL}, the construction of Andreatta--Iovita corresponds to the integral structure alluded to in \cite[Remark 10]{Urban}; in particular, one has 
    \[
    \varinjlim_{\rho}\widehat{\mathcal{N}_{\mathfrak{U}}^{\infty, \rho}}(N) \cong \varinjlim_s \opn{H}^0(\mathcal{X}_s, \mbb{W}_{k_{\mathfrak{U}}}[1/p]) = \mathscr{N}^{\dagger, \opn{AI}}_{\kappa_{\mathfrak{U}}}
    \]
    where $k_{\mathfrak{U}}$ denotes the universal character of $\mathfrak{U}$, $\kappa_{\mathfrak{U}}$ is any locally analytic character of $T(\mbb{Z}_p)$ extending $k_{\mathfrak{U}}$, and $\mbb{W}_{k_{\mathfrak{U}}}[1/p]$ denotes the restriction of $\mbb{W}_{k, I}[1/p]$ to $\mathfrak{U}$ (for any choice of interval $I = [0, p^b]$ such that $\mathfrak{U} \subset \mathcal{W}_I$, which always exists because $\mathfrak{U}$ is assumed to be quasi-compact). We therefore see that the space $\widehat{\mathcal{N}_{\mathfrak{U}}^{\infty, \rho}}(N)$ is also not large enough for $p$-adically interpolating the Gauss--Manin connection in the level of generality of this article (i.e., extending $\nabla$ to an action of $C^{\opn{la}}(\mbb{Z}_p, \mbb{Q}_p)$) for the same reasons as above.
\end{remark}

\section{\texorpdfstring{$p$}{p}-adic \texorpdfstring{$L$}{L}-functions} \label{SectpadicLfunctions}

In this final section, we describe an application of our theory to the construction of triple product and Rankin--Selberg $p$-adic $L$-functions in families, without the restriction on the weight imposed in \cite{AI_LLL}. Throughout, we let $p > 2$ and let $N \geq 4$ be an integer which is prime to $p$. We take $K = \opn{GL}_2(\mbb{Z}_p) K^p$ to be the compact open subgroup of all matrices which lie in the upper-triangular unipotent modulo $N \widehat{\mbb{Z}}$.

We fix an isomorphism $\iota_p \colon {\overline{\mathbb{Q}}_p} \xrightarrow{\sim} \mbb{C}$ throughout, which gives a canonical choice of $N$-th root of unity $\zeta_N \defeq \iota_p^{-1}(e^{2\pi i/N})$ (after fixing $i = \sqrt{-1}$). Any $q$-expansion in this section will refer to the $q$-expansion at the cusp $(\opn{Tate}(q), \alpha_N^{\opn{can}})$, where $\opn{Tate}(q)$ is the Tate curve over $\mbb{Z}[1/N](\!(q)\!)$, and $\alpha_{N}^{\opn{can}}$ is the canonical level $\Gamma_1(N)$-structure given by $\zeta_N$. Note that the map from overconvergent modular forms to $q$-expansions at this cusp is injective, as can be seen by using the $q$-expansion principle for $p$-adic modular forms (\cite{Katz}) and the inclusion of overconvergent modular forms into $p$-adic modular forms. Furthermore, we base-change everything in this section over a finite Galois extension $L / \mbb{Q}_p$ containing $\zeta_N$, but omit this from the notation.

\begin{convention}
As is customary in the literature, our weights will be locally analytic characters of $\mbb{Z}_p^{\times}$ and not $T(\mbb{Z}_p)$. This amounts to choosing a normalisation: for any locally analytic character $\kappa \colon \mbb{Z}_p^{\times} \to R^{\times}$, we set
\[
\mathscr{N}^{\dagger}_{\kappa} \defeq \mathscr{N}^{\dagger}_{(\kappa; w(\kappa))}
\]
where $w(\kappa) = -\kappa/2$ (resp. $w(\kappa) = (1-\kappa)/2$) if $\kappa(-1) = 1$ (resp. $\kappa(-1) = -1$). These weight spaces are stable under the action of $C^{\opn{la}}(\mbb{Z}_p, R)$. We adopt similar notations for overconvergent modular forms. 

Finally, we set $U_p^{\circ} = p^{-1}U_p$, which is the normalisation giving the usual description of the $U_p$-Hecke operator on $q$-expansions. All slope decompositions will be with respect to $U_p^{\circ}$.
\end{convention}

\subsection{The eigencurve}

Let $\mathcal{C}$ denote the Buzzard--Coleman--Mazur cuspidal eigencurve (over $\opn{Spa}(L, \mathcal{O}_L)$) of tame level $\Gamma_1(N)$, which comes equipped with a weight map $w \colon \mathcal{C} \to \mathcal{W}$, where $\mathcal{W}$ denotes the weight space parameterising continuous characters on $\mbb{Z}_p^{\times}$. Over $\mathcal{C}$, we have a universal eigenform
\[
\mathcal{F}^{\opn{univ}} = \sum_{n \geq 1} a_n q^n \quad \in \mathcal{O}(\mathcal{C})[\![q]\!]
\]
with $a_1 = 1$ and $a_p \in \mathcal{O}(\mathcal{C})^{\times}$, satisfying the following universal property: for any affinoid $U$ with a weight morphism $\kappa : U \to \mathcal{W}$, and any family $\mathcal{F}_U$ of finite slope eigenforms over $U$ of tame level $\Gamma_1(N)$ and weight $\kappa$, there exists a unique morphism $U \to \mathcal{C}$ lifing $\kappa$ such that $\mathcal{F}_U$ is the pullback of $\mathcal{F}^{\rm univ}$ (cf. the discussion just after \cite[Definition 3.5]{LoefflerRSNote}). For any quasi-compact open affinoid $V \subset \mathcal{C}$, we let $\mathcal{F}_V \in \mathcal{O}(V)[\![q]\!]$ denote the pullback of $\mathcal{F}^{\opn{univ}}$ to $V$. 

Let $\mathscr{S}^{\dagger} \subset \mathscr{M}^{\dagger}$ denote the subspace of cuspidal overconvergent modular forms, i.e. those which vanish at all the cusps (not just the specific choice above). This space is stable under the actions of $T(\mbb{Z}_p)$ and $U_p$. Over $\mathcal{W}$ we have a universal sheaf of cuspidal overconvergent modular forms $\mathcal{S}^{\dagger}$, such that for any quasi-compact open $U \subset \mathcal{W}$, we have $\mathcal{S}^{\dagger}(U) = \mathscr{S}^{\dagger}_{\kappa}$ (where $\kappa$ is the universal character of $U$). We let $\mathcal{S}^{\dagger}_{\mathcal{C}} \defeq w^*\mathcal{S}^{\dagger}$ denote the pullback to $\mathcal{C}$. 

We have the following proposition:

\begin{proposition} \label{UniversalClassProp}
Let $V \subset \mathcal{C}$ be a quasi-compact open affinoid subspace, and let $U = w(V) \subset \mathcal{W}$ denote its image in weight space with universal character $\kappa \colon \mbb{Z}_p^{\times} \to \mathcal{O}(U)^{\times}$. Then there exists a ``universal cohomology class''
\[
\omega_{V} \in \mathcal{S}^{\dagger}_{\mathcal{C}}(V) = \mathscr{S}^{\dagger}_{\kappa} \hatot_{\mathcal{O}(U)} \mathcal{O}(V)
\]
with $q$-expansion given by $\mathcal{F}_{V}$. Moreover the restriction of $\omega_{V}$ to $\mathcal{S}^{\dagger}_{\mathcal{C}}(V')$ is equal to $\omega_{V'}$ for $V' \subset V$, so these classes glue to give a universal cohomology class $\omega_{\mathcal{C}} \in \mathcal{S}^{\dagger}_{\mathcal{C}}(\mathcal{C})$.
\end{proposition}
\begin{proof}
This follows from the duality between cuspidal overconvergent forms and Hecke algebras (cf. for instance, \cite[\S 4.2]{Urban} for trivial tame level, or \cite[Proposition B5.6]{Coleman}). Observe that, since we are working at tame level $\Gamma_1(N)$, here we consider the Hecke algebra generated by the Hecke operators $T_\ell$ for $\ell \nmid N p$, $U_\ell$ and the diamond operators $\langle \ell \rangle$ for $\ell \mid N$, and the Hecke operator $U_p$.
\end{proof}

\subsubsection{The dual class} 

Let $V \subset \mathcal{C}$ be a quasi-compact open affinoid subspace, and let $\omega_{V} \in \mathcal{S}^{\dagger}_{\mathcal{C}}(V)$ denote the universal cohomology class provided by Proposition \ref{UniversalClassProp}. We will now explain how to associate a ``dual class'' to $\omega_{V}$. As explained in \cite[\S 5.2]{AI_LLL}, one has an Atkin--Lehner involution $w_N$ on $X = X_1(N)$ (as $\zeta_N \in \mathcal{O}_L$), which extends to involutions on $M_{\opn{dR}}$ and $\mathcal{IG}_{\infty}$, compatible with the morphism $\mathcal{IG}_{\infty} \to M_{\opn{dR}}^{\opn{an}}$. This induces an involution $w_N \colon \mathscr{M}^{\dagger} \to \mathscr{M}^{\dagger}$ which commutes with the $T(\mbb{Z}_p)$ and $U_p$ actions and preserves $\mathscr{S}^{\dagger}$.

\begin{definition}
Let $\omega_{V}^c$ denote the ``dual class'' given by
\[
\omega^c_{V} \defeq w_N(\omega_{V}) \in \mathcal{S}^{\dagger}_{\mathcal{C}}(V) .
\]
These classes are compatible for $V' \subset V$ because the same is true for the Atkin--Lehner involution $w_N$, and we denote by $\omega_{\mathcal{C}}^c \in \mathcal{S}^{\dagger}_{\mathcal{C}}(\mathcal{C})$ the dual class obtained by gluing. 
\end{definition}

\subsubsection{A linear functional} \label{ALinearFunctionalSection}

For any integer $M | N$, let $\mathcal{C}(M)$ denote the Buzzard--Coleman--Mazur cuspidal eigencurve of tame level $\Gamma_1(M)$, and we let $\mathscr{S}^{\dagger}_{\kappa_1}(M)$ denote the space of cuspidal overconvergent modular forms of weight $\kappa_1$ and tame level $\Gamma_1(M)$. If $M=N$, we omit this from the notation and we keep the notations of the previous sections.

\begin{definition} \label{ClassicalPregDef}
    We say that a point $x \in \mathcal{C}(M)$ is:
    \begin{itemize}
        \item \emph{classical} if $w(x) \in \mbb{Z}_{\geq 2}$ and the specialisation $\mathcal{F}^{\opn{univ}}_x$ of the universal eigenform $\mathcal{F}^{\opn{univ}}$ over $\mathcal{C}(M)$ at the point $x$ is the $q$-expansion of a normalised cuspidal modular form of weight $w(x)$ and level $\Gamma_1(M) \cap \Gamma_0(p^r)$, for some integer $r \geq 1$;
        \item \emph{crystalline} if $x$ is classical and the newform associated with the modular form $\mathcal{F}^{\opn{univ}}_x$ has level $\Gamma_1(M')$ for some $M' | M$;
        \item \emph{$p$-regular crystalline} if $x$ is crystalline and the roots of the Hecke polynomial at $p$
        \[
        X^2 - a_p(f)X + \varepsilon_f(p)p^{w(x)-1}
        \]
        are distinct, where $f$ denotes the newform associated with $\mathcal{F}^{\opn{univ}}_x$ with nebentypus $\varepsilon_f$.
    \end{itemize}
\end{definition}

We make the following assumption:

\begin{assumption} \label{NobleAssumption}
Let $U_1 \subset \mathcal{W}$ be a connected quasi-compact open affinoid subspace in weight space with universal character $\kappa_1 \colon \mbb{Z}_p^{\times} \to \mathcal{O}(U_1)^{\times}$. We suppose that there exists an integer $N_{\mathcal{F}}$ dividing $N$ and an open $V_1 \subset \mathcal{C}(N_{\mathcal{F}})$ lying over $U_1$ such that
\begin{itemize}
    \item The weight map $w \colon V_1 \to U_1$ is an isomorphism.
    \item Every specialisation of the Coleman family $\mathcal{F} := \mathcal{F}_{V_1}$ at a classical point $x \in V_1$ with $w(x) \in \mbb{Z}_{\geq 2}$, is \emph{noble}; i.e. it is the $p$-stabilisation of a normalised cuspidal newform of level $\Gamma_1(N_{\mathcal{F}})$ that is $p$-regular, with the additional condition that its local Galois representation is not the sum of two characters if it is of critical slope (see \cite[Definition 4.6.3]{LZ16Coleman}).
\end{itemize}
\end{assumption}

\begin{remark}
Many such opens $V_1$ exist by starting with a noble eigenform and taking a sufficiently small neighbourhood of its corresponding point in the eigencurve (see \cite{bellpadic}). 
\end{remark}

The Coleman family $\mathcal{F}$ will constitute the first variable of the triple $p$-adic $L$-function.

\begin{lemma}
Let $\omega_{V_1} \in \mathscr{S}^{\dagger}_{\kappa_1}(N_{\mathcal{F}})$ denote the universal class over $V_1$, as in Proposition \ref{UniversalClassProp}. After possibly shrinking $V_1$, there exists a unique $\mathcal{O}(V_1)$-linear Hecke equivariant (for the Hecke operators $\{ T_{\ell} : \ell \nmid p N_{\mathcal{F}} \} \cup \{U_p \}$) map
\begin{equation} \label{LevelNFmap}
\mathscr{S}^{\dagger}_{\kappa_1}(N_{\mathcal{F}}) \to \mathcal{O}(V_1)
\end{equation}
sending the universal class $\omega_{V_1}$ to $1$, where the action of the Hecke algebra on $\mathcal{O}(V_1)$ is through the Hecke eigensystem corresponding to the Coleman family $\mathcal{F}$.
\end{lemma}
\begin{proof}
If $\alpha_{\mathcal{F}}$ denotes the Hecke eigencharacter associated with the Coleman family $\mathcal{F}$, then $\omega_{V_1}$ is an eigenclass for the Hecke action with eigencharacter $\alpha_{\mathcal{F}}$. We claim that this appears as a direct summand in $\mathscr{S}^{\dagger}_{\kappa_1}(N_{\mathcal{F}})$ with multiplicity one. Indeed, it suffices to show this in $\mathscr{S}^{\dagger}_{\kappa_1}(N_{\mathcal{F}})^{\leq h}$ for sufficiently large $h \in \mbb{N}$, since the slope $\leq h$ part is a direct summand of $\mathscr{S}^{\dagger}_{\kappa_1}(N_{\mathcal{F}})$. Let $\mbb{T}_h = \mbb{T}(\mathscr{S}^{\dagger}_{\kappa_1}(N_{\mathcal{F}})^{\leq h})$ denote the Hecke algebra over $\mathcal{O}(U_1)$ (generated by the same Hecke operators as in the proof of Proposition \ref{UniversalClassProp}) acting faithfully on the module $\mathscr{S}^{\dagger}_{\kappa_1}(N_{\mathcal{F}})^{\leq h}$. 

Since the weight map induces an isomorphism on $V_1$ (Assumption \ref{NobleAssumption}), there exists an idempotent $e$ in $\mbb{T}_h$ such that $e\mbb{T}_h = \mathcal{O}(V_1) \cong \mathcal{O}(U_1)$ and $e\mbb{T}_h \subset \mbb{T}_h$ is the (generalised) eigenspace for the character $\alpha_{\mathcal{F}}$. We note that we also have a stronger property, namely, $e \mbb{T}_h \subset \mbb{T}_h$ is the (generalised) eigenspace for the action of the Hecke operators $\{ T_{\ell} : \ell \nmid p N_{\mathcal{F}} \} \cup \{U_p \}$ with eigencharacter $\alpha_{\mathcal{F}}$. This is because any Coleman family over $U_1$ with the same prime-to-$N_{\mathcal{F}}$ Hecke eigensystem as $\mathcal{F}$ must coincide with $\mathcal{F}$ (this follows from the second bullet point in Assumption \ref{NobleAssumption}).

Finally, by the duality between overconvergent modular forms and Hecke algebras (see the proof of Proposition \ref{UniversalClassProp}), the same properties are true for $\mathscr{S}^{\dagger}_{\kappa_1}(N_{\mathcal{F}})^{\leq h}$, namely: one has $e \mathscr{S}^{\dagger}_{\kappa_1}(N_{\mathcal{F}})^{\leq h} \cong \mathcal{O}(U_1)$ and $e \mathscr{S}^{\dagger}_{\kappa_1}(N_{\mathcal{F}})^{\leq h} \subset \mathscr{S}^{\dagger}_{\kappa_1}(N_{\mathcal{F}})^{\leq h}$ is the (generalised) eigenspace for the action of the Hecke operators $\{ T_{\ell} : \ell \nmid p N_{\mathcal{F}} \} \cup \{U_p \}$ with eigencharacter $\alpha_{\mathcal{F}}$. This completes the proof of the lemma. 
\end{proof}

We extend this to a linear functional on $\mathscr{S}^{\dagger}_{\kappa_1}$ as follows. For any divisor $a \geq 1$ of $N/N_{\mathcal{F}}$, we have a finite \'{e}tale morphism $[a] \colon X_1(N) \to X_1(N_{\mathcal{F}})$ as described at the start of \cite[\S 5]{AI_LLL}. This induces a trace morphism $[a]_* \colon \mathscr{S}^{\dagger}_{\kappa_1} \to \mathscr{S}^{\dagger}_{\kappa_1}(N_{\mathcal{F}})$ and a pullback map $[a]^* \colon \mathscr{S}^{\dagger}_{\kappa_1}(N_{\mathcal{F}}) \to \mathscr{S}^{\dagger}_{\kappa_1}$, with the latter given by translation of the argument by $a$ on classical modular forms.

\begin{definition} \label{DefOfLinearFunctional}
Let $\underline{\mu} = (\mu_a)_{a | N/N_{\mathcal{F}}}$ be a tuple of elements of $L \cap \iota_p^{-1}\overline{\mbb{Q}}$ indexed by the divisors of $N/N_{\mathcal{F}}$. We set $\omega_{\mathcal{F}^{\circ}} = \sum_{a | N/N_{\mathcal{F}}} \mu_a \cdot [a]^*(\omega_{V_1}) \in \mathscr{S}^{\dagger}_{\kappa_1}$. Set $\omega_{\mathcal{F}^{\circ}}^c = w_N(\omega_{\mathcal{F}^{\circ}})$. We define
\[
\lambda_{\mathcal{F}^{\circ, c}} \colon \mathscr{S}^{\dagger}_{\kappa_1} \to \mathcal{O}(V_1)
\]
to be the $\mathcal{O}(V_1)$-linear map sending a class $\delta \in \mathscr{S}^{\dagger}_{\kappa_1}$ to the image of $[\Gamma_1(N_{\mathcal{F}}): \Gamma_1(N)]^{-1} \cdot \sum_{a | N/N_{\mathcal{F}}} \bar{\mu}_a \cdot [a]_* w_N(\delta)$ under the map (\ref{LevelNFmap}), where $\bar{\cdot}$ denotes complex conjugation. This is Hecke equivariant away from $Np$ and $U_p$-equivariant, where the Hecke action on the target is given by the eigensystem for $\mathcal{F}^{\circ, c}$ (see \cite[Lemma 3.4]{LoefflerRSNote}).
\end{definition}

We have the following specialisation formula:

\begin{lemma} \label{LemmaSpecialisationOfLinearFunctional}
Let $x_1 \in w^{-1}(\mbb{Z}_{\geq 2}) \cap V_1$, let $\mathcal{F}_{x_1}$ denote the specialisation of the Coleman family at $x_1$ and let 
\[
\lambda^{\opn{cl}}_{\mathcal{F}_{x_1}^{\circ, c}} \colon S_k(\Gamma_1(N) \cap \Gamma_0(p); \mbb{C}) \to \mbb{C}
\]
denote the restriction of the specialisation of $\lambda_{\mathcal{F}^{\circ, c}}$ at $x_1$ to the space of classical cusp forms of weight $k \defeq w(x_1)$ and level $\Gamma_1(N) \cap \Gamma_0(p)$ (using $\iota_p$ to extend scalars to $\mbb{C}$).\footnote{Note that one has a natural map $\mathcal{X}_r \to \mathcal{X}_0(p)$ (where the latter denotes the modular curve of level $\Gamma_1(N) \cap \Gamma_0(p)$) which induces a Hecke equivariant map from cuspforms of level $\Gamma_1(N) \cap \Gamma_0(p)$ to cuspidal overconvergent modular forms (see \cite[\S 5.4.4]{BoxerPilloni}).} Let $\mathcal{F}_{x_1}^{\circ, c} = \mathcal{F}^{\circ}_{x_1} \otimes \varepsilon_{\mathcal{F}_{x_1}}^{-1}$, where $\varepsilon_{\mathcal{F}_{x_1}}$ denotes the nebentypus of $\mathcal{F}_{x_1}$ and $\mathcal{F}^{\circ}_{x_1} = \sum_{a | N/N_{\mathcal{F}}} \mu_a \cdot [a]^*\mathcal{F}_{x_1}$. Then
\[
\lambda^{\opn{cl}}_{\mathcal{F}_{x_1}^{\circ, c}}(g) = \frac{\langle \mathcal{F}_{x_1}^{\circ, c}, g \rangle_{N, p}}{\langle \mathcal{F}_{x_1}^{c}, \mathcal{F}_{x_1}^{c} \rangle_{N, p}}
\]
for any $g \in S_k(\Gamma_1(N) \cap \Gamma_0(p); \mbb{C})$, where $\langle \cdot, \cdot \rangle_{N, p}$ denotes the Petersson inner product of level $\Gamma_1(N) \cap \Gamma_0(p)$ as in \cite[Eq. (35)]{DarmonRotgerGrossZagier1} (which is Hermitian linear in the first variable).
\end{lemma}
\begin{proof}
The restriction of the specialisation of (\ref{LevelNFmap}) at $x_1$ to complex-valued cusp forms of weight $k$ and level $\Gamma_1(N_{\mathcal{F}}) \cap \Gamma_0(p)$ defines a Hecke equivariant map $S_k(\Gamma_1(N_{\mathcal{F}}) \cap \Gamma_0(p); \mbb{C}) \to \mbb{C}$ factoring through the eigenspace associated with $\mathcal{F}_{x_1}$, and sending $\mathcal{F}_{x_1}$ to $1$. Set $M = N_{\mathcal{F}}$. Since $\mathcal{F}_{x_1}$ is a noble, this map must therefore be equal to
\[
\frac{\langle \mathcal{F}_{x_1}, - \rangle_{M, p}}{\langle \mathcal{F}_{x_1}, \mathcal{F}_{x_1} \rangle_{M, p}} = \frac{\langle \mathcal{F}_{x_1}, - \rangle_{M, p}}{\langle \mathcal{F}_{x_1}^c, \mathcal{F}_{x_1}^c \rangle_{M, p}} = [\Gamma_1(M) : \Gamma_1(N)] \frac{\langle \mathcal{F}_{x_1}, - \rangle_{M, p}}{\langle \mathcal{F}_{x_1}^c, \mathcal{F}_{x_1}^c \rangle_{N, p}}.
\]
Now one uses the fact that $[a]_*$ (resp. $w_N$) and $[a]^*$ (resp. $w_N$) are adjoint under the Petersson inner product, and that it is Hermitian linear in the first variable.
\end{proof}

\subsection{Triple product \texorpdfstring{$p$}{p}-adic \texorpdfstring{$L$}{L}-functions}

We now construct a pairing over two copies of the eigencurve which will be used in our construction of triple product and Rankin-Selberg $p$-adic $L$-functions. Fix a Coleman family $\mathcal{F}$ over the open $V_1 \subset \mathcal{C}(N_{\mathcal{F}})$ (which is isomorphic to an open $U_1 \subset \mathcal{W}$ via the weight map) as in \S \ref{ALinearFunctionalSection} and Assumption \ref{NobleAssumption}.

\begin{notation} 
Let $\left( V_1 \times \mathcal{C} \times \mathcal{C} \right)^+ \subset V_1 \times \mathcal{C} \times \mathcal{C}$ denote the open and closed subspace defined by the condition: for any $(x_1, x_2, x_3) \in V_1({\overline{\mathbb{Q}}_p}) \times \mathcal{C}({\overline{\mathbb{Q}}_p}) \times \mathcal{C}({\overline{\mathbb{Q}}_p})$, the weight
\[
w(x_1) - w(x_2) - w(x_3)
\]
is even, i.e. there exists $u \colon \mbb{Z}_p^{\times} \to \overline{\mathbb{Q}}_p^{\times}$ such that $w(x_1) - w(x_2) - w(x_3) = 2u$.
\end{notation} 

Let $V_2, V_3 \subset \mathcal{C}$ be two quasi-compact open affinoid subspaces with images $U_2 = w(V_2), U_3 = w(V_3) \subset \mathcal{W}$ and such that $V_1 \times V_2 \times V_3 \subseteq (V_1 \times \mathcal{C} \times \mathcal{C})^+$. Let $a_p(\mathcal{F})$ denote the coefficient of $q^p$ in the series $\mathcal{F} = \mathcal{F}_{V_1}$, and suppose that the $p$-adic Banach norm of $a_p(\mathcal{F})$ is $\leq h$, for some integer $h \geq 0$. By shrinking $V_i$ ($i=1, 2, 3$) if necessary we have the following:
\begin{itemize}
    \item For $i=1, 2, 3$, let $\kappa_i \colon \mbb{Z}_p^{\times} \to \mathcal{O}(U_i)^{\times}$ denote the universal character of $U_i$ and $\hat{\kappa}_i$ its pullback to $U = U_1 \times U_2 \times U_3$. Then, we may assume that: 
    \[
    \hat{\kappa}_1 - \hat{\kappa}_2 - \hat{\kappa}_3
    \]
    is even, i.e. there exists $u \colon \mbb{Z}_p^{\times} \to \mathcal{O}(U)^{\times}$ such that $\hat{\kappa}_1 - \hat{\kappa}_2 - \hat{\kappa}_3 = 2u$.
    \item The space of nearly overconvergent modular forms $\mathscr{N}^{\dagger}_{\hat{\kappa}_1}$ has a slope $\leq h$ decomposition and hence a $\mathcal{O}(U)$-linear overconvergent projector 
    \[
    \Pi^{\opn{oc}, \leq h} \colon \mathscr{N}^{\dagger}_{\hat{\kappa}_1} \to \mathscr{M}^{\dagger, \leq h}_{\hat{\kappa}_1} \otimes_{\mathcal{O}(U)} (F_{U_1} \hatot \mathcal{O}(U_2) \hatot \mathcal{O}(U_3) )
    \]
    where $F_{U_1}$ denotes the fraction ring of $U_1$ (see Theorem \ref{SlopeDecompositionForNdaggerThm} and Theorem \ref{OCProjectorForNdaggerThm}). Furthermore, the space of nearly overconvergent modular forms $\mathscr{N}^{\dagger}$ over $\mathcal{O}(U)$ comes equipped with an action
    \[
    C^{\opn{la}}(\mbb{Z}_p, \mathcal{O}(U)) \times \mathscr{N}^{\dagger} \to \mathscr{N}^{\dagger} 
    \]
    as in Theorem \ref{SheafVersionMainThm}, extending the Gauss--Manin connection. We denote this action by $\star$.
\end{itemize}

\begin{definition} \label{DefXiBilinearPairing}
With notation as above, for any $\phi \in \mathcal{S}^{\dagger}_{\mathcal{C}}(V_2)$ and $\psi \in \mathcal{S}^{\dagger}_{\mathcal{C}}(V_3)$, we define
\[
\Xi (\phi, \psi) = \Pi^{\opn{oc}, \leq h_1}( (u \cdot 1_{\mbb{Z}_p^{\times}}) \star \hat{\phi} \times \hat{\psi} ) \in \mathscr{S}^{\dagger, \leq h}_{\kappa_1} \hatot_{\mathcal{O}(V_1)} \left( F_{V_1} \hat{\otimes} \mathcal{O}(V_2) \hat{\otimes} \mathcal{O}(V_3)\right)
\]
where $1_{\mbb{Z}_p^{\times}}$ denotes the indicator function of $\mbb{Z}_p^{\times}$, and $\hat{\phi}$ and $\hat{\psi}$ denote the pullback of the classes to $V_1 \times V_2 \times V_3$.
\end{definition}

We have the following lemma:

\begin{lemma}
With notation as above:
\begin{enumerate}
    \item The construction $\Xi(-, -)$ is $\mathcal{O}(V_2)$-linear in the first variable and $\mathcal{O}(V_3)$-linear in the second variable.
    \item If $V_i' \subset V_i$ ($i = 2, 3$) then $\Xi(-, -)$ and $\Xi'(-, -)$ are compatible under restriction, where $\Xi'(-, -)$ denotes the construction in Definition \ref{DefXiBilinearPairing} over $V_1 \times V_2' \times V_3'$. Furthermore, the constructions are compatible for different choices of $h$.
    \item The constructions glue to give a morphism of sheaves\footnote{For a pair of sheaves $\mathcal{F}, \mathcal{G}$ of Banach modules on two spaces $X, Y$, we denote by $\mathcal{F} \; \hat{\boxtimes} \; \mathcal{G}$ the sheaf on $X \times Y$ obtained by sheafifying the assignment $U \times V \mapsto \mathcal{F}(U) \; \hat{\otimes} \; \mathcal{G}(V)$.} over $(V_1 \times \mathcal{C} \times \mathcal{C})^+$
    \[
    {\boldsymbol\Xi} \colon \mathcal{O}_{V_1} \;  \hat{\boxtimes} \; \mathcal{S}^{\dagger}_{\mathcal{C}} \; \hat{\boxtimes} \; \mathcal{S}^{\dagger}_{\mathcal{C}} \to (\mathscr{S}^{\dagger, \leq h}_{\kappa_1} \hatot_{\mathcal{O}(V_1)} \mathcal{K}_{V_1} ) \; \hat{\boxtimes} \; \mathcal{O}_{\mathcal{C}} \; \hat{\boxtimes} \; \mathcal{O}_{\mathcal{C}}
    \]
    where $\mathcal{K}_{V_1}$ denotes the sheaf of meromorphic functions on $V_1$.
\end{enumerate}
\end{lemma}
\begin{proof}
The first part is clear and the third part follows from the second. Therefore, we just need to prove part (2). But this just follows from the linearity of $\star$ and $\Pi^{\opn{oc}, \leq h}$, and the compatibility of the slope $\leq h$ projectors as $h$ varies.
\end{proof}

We consider the following pairing:

\begin{definition} \label{DefJmu}
Fix a tuple $\underline{\mu} = (\mu_a)_{a | N/N_{\mathcal{F}}}$ as in Definition \ref{DefOfLinearFunctional}. Then, for any pair of opens $V_2, V_3 \subset \mathcal{C}$ (not necessarily quasi-compact nor affinoid) such that $V_1 \times V_2 \times V_3 \subset (V_1 \times \mathcal{C} \times \mathcal{C})^+$, we consider the following bilinear pairing:
\begin{align*} 
J_{\underline{\mu}}( -, -) \colon  \mathcal{S}^{\dagger}_{\mathcal{C}}(V_2) \times \mathcal{S}^{\dagger}_{\mathcal{C}}(V_3)  &\to  \opn{Mer}(V_1 \times V_2 \times V_3) \\
 (\phi, \psi) &\mapsto \lambda_{\mathcal{F}^{\circ, c}} \left( {\boldsymbol\Xi}(\phi, \psi)\right)
\end{align*}
where $\opn{Mer}(V_1 \times V_2 \times V_3)$ denotes the module of meromorphic functions on $V_1 \times V_2 \times V_3$. This pairing glues to a bilinear pairing $\mathcal{S}^{\dagger}_{\mathcal{C}}(\mathcal{C}) \times \mathcal{S}^{\dagger}_{\mathcal{C}}(\mathcal{C}) \to \opn{Mer}(V_1 \times \mathcal{C} \times \mathcal{C})^+$, which we continue to denote by $J_{\underline{\mu}}( -, -)$. 
\end{definition}

We define the universal triple product $p$-adic $L$-function as follows:

\begin{definition}
    Fix a Coleman family $\mathcal{F}$ over $V_1$ satisfying Assumption \ref{NobleAssumption} as in \S \ref{ALinearFunctionalSection}, and fix a tuple $\underline{\mu} = (\mu_a)_{a | N/N_{\mathcal{F}}}$ as in Definition \ref{DefOfLinearFunctional}. We define $\mathscr{L}_p \in \opn{Mer}(V_1 \times \mathcal{C} \times \mathcal{C})^+$ to be the meromorphic function given by
    \[
    \mathscr{L}_p \defeq J_{\underline{\mu}}(\omega_{\mathcal{C}}, \omega_{\mathcal{C}})
    \]
    where $\omega_{\mathcal{C}}$ denotes the universal cohomology class of Proposition \ref{UniversalClassProp}.
\end{definition}

\subsubsection{Interpolation property}

We now describe the interpolation property for the universal triple product $p$-adic $L$-function. Throughout this section, fix a Coleman family $\mathcal{F}$ over $V_1 \subset \mathcal{C}(N_{\mathcal{F}})$ as in \S \ref{ALinearFunctionalSection} and Assumption \ref{NobleAssumption}, and let $U_1 = w(V_1) \subset \mathcal{W}$. Fix a tuple $\underline{\mu} = (\mu_a)_{a | N/N_{\mathcal{F}}}$ as in Definition \ref{DefOfLinearFunctional}.

\begin{definition}
    We say that a point $x = (x_1, x_2, x_3) \in V_1 \times \mathcal{C} \times \mathcal{C}$ is \emph{unbalanced crystalline} if:
    \begin{itemize}
        \item $x_1, x_2, x_3$ are classical and $p$-regular crystalline (Definition \ref{ClassicalPregDef});
        \item the weights $(w(x_1), w(x_2), w(x_3))$ satisfy $w(x_1) \geq w(x_2) + w(x_3)$ and the nebentypes $\varepsilon_f, \varepsilon_g, \varepsilon_h$ of the newforms $(f, g, h)$ associated with $(x_1, x_2, x_3)$ satisfy $\varepsilon_f \cdot \varepsilon_g \cdot \varepsilon_h = 1$.\footnote{This implies that there exists an integer $t \geq 0$ such that $w(x_1) - w(x_2) - w(x_3) = 2t$.}
    \end{itemize}
    We denote the set of unbalanced crystalline points by $\Sigma_{\opn{cris}}$ (which depends on the fixed choice of $\mathcal{F}$).
\end{definition}

Given an unbalanced crystalline point $x = (x_1, x_2, x_3) \in \Sigma_{\opn{cris}}$, we can fix the following data.
\begin{itemize}
    \item (Weights) We set $(k, \ell, m) = (w(x_1), w(x_2), w(x_3))$.
    \item ($U_p^{\circ}$-eigenvalues) We let $(\alpha_f, \alpha_{g}, \alpha_h)$ denote the coefficients of $q^p$ in $\mathcal{F}_{x_1}, \mathcal{F}^{\opn{univ}}_{x_2}, \mathcal{F}^{\opn{univ}}_{x_3}$ respectively.
    \item (Newforms) We let $(f, g, h)$ denote the newforms of levels $\Gamma_1(N_f)$, $\Gamma_1(N_g)$, $\Gamma_1(N_h)$, associated with $\mathcal{F}_{x_1}, \mathcal{F}^{\opn{univ}}_{x_2}, \mathcal{F}^{\opn{univ}}_{x_3}$ respectively. Note that, since the point $x$ is unbalanced crystalline, all of the integers $N_f$, $N_g$, $N_h$ divide $N$. Furthermore, by Assumption \ref{NobleAssumption}, one has $N_f = N_{\mathcal{F}}$.
    \item (Test data) We obtain a triple $(f^{\circ}, g^{\circ}, h^{\circ})$ of cuspidal modular forms, where: $g^{\circ}$ and $h^{\circ}$ are the unique eigenforms of level $\Gamma_1(N)$ such that $\mathcal{F}^{\opn{univ}}_{x_2}, \mathcal{F}^{\opn{univ}}_{x_3}$ are the $p$-stabilisations of $g^{\circ}$, $h^{\circ}$ with respect to the roots $\alpha_g$, $\alpha_h$; and we define $f^{\circ} = \sum_{a | N/N_{\mathcal{F}}} \mu_a \cdot [a]^*(f)$. Note that $\mathcal{F}^{\circ}_{x_1}$ is the $p$-stabilisation of $f^{\circ}$ at the root $\alpha_f$.
\end{itemize}

To state the interpolation property for $\mathscr{L}_p$, we need to introduce several Euler factors. Let $(\beta_f, \beta_g, \beta_h)$ denote the roots of the Hecke polynomials at $p$ that are different from $(\alpha_f, \alpha_g, \alpha_h)$. Set $c = (k+ \ell + m - 2)/2$ and define:
\begin{align*}
    \mathcal{E}(f, g, h) &= (1 - \beta_f \alpha_g \alpha_h p^{-c})(1 - \beta_f \alpha_g \beta_h p^{-c})(1 - \beta_f \beta_g \alpha_h p^{-c})(1 - \beta_f \beta_g \beta_h p^{-c}), \\
    \mathcal{E}_0(f) &= (1- \beta_f^2 \varepsilon_f(p)^{-1} p^{1-k}), \\
    \mathcal{E}_1(f) &= (1- \beta_f^2 \varepsilon_f(p)^{-1} p^{-k}).
\end{align*}
For any integer $j \geq 0$, let $\delta_j = \frac{1}{2 \pi i}\left(\frac{d}{dz} + \frac{j}{z - \bar{z}} \right)$ denote the Maass--Shimura differential operator acting on weight $j$ nearly holomorphic modular forms. For any $t \geq 0$, we set $\delta_j^t = \delta_{j+2(t-1)} \circ \delta_{j + 2(t-2)} \circ \cdots \circ \delta_j$. Let $\langle \cdot, \cdot \rangle_N$ denote the Petersson inner product of level $\Gamma_1(N)$ on (complex-valued) nearly holomorphic modular forms. 

\begin{theorem} \label{TripleProductJinterpolation}
Let $x = (x_1, x_2, x_3) \in \Sigma_{\opn{cris}}$ with associated newforms $(f, g, h)$ and test data $(f^{\circ}, g^{\circ}, h^{\circ})$. Then $x$ is not a pole for $\mathscr{L}_p$ and the specialisation of $\mathscr{L}_p$ at $x$ satisfies:
\begin{equation} \label{MainInterpolationFormula}
\mathscr{L}_p(x) = \frac{\mathcal{E}(f, g, h)}{\mathcal{E}_0(f)\mathcal{E}_1(f)} \cdot \frac{I(f^{\circ}, g^{\circ}, h^{\circ})}{\langle f, f \rangle_N }
\end{equation}
where we view the left-hand side of (\ref{MainInterpolationFormula}) in $\mbb{C}$ via $\iota_p$, and $I(f^{\circ}, g^{\circ}, h^{\circ}) = \langle (f^{\circ})^*, \delta_{\ell}^t g^{\circ} \times h^{\circ} \rangle_N$ with $(f^{\circ})^*(\tau) = \overline{f^{\circ}(-\bar{\tau})}$.
\end{theorem}

\begin{proof}
The proof of this is very similar to \cite{DarmonRotgerGrossZagier1} and \cite{AI_LLL}. Let $\omega_{g^{\circ}}^{\alpha}$ and $\omega_{h^{\circ}}^{\alpha}$ denote the specialisations of $\omega_{\mathcal{C}}$ at $x_2$ and $x_3$ respectively, and recall $t = (k-\ell-m)/2$. Firstly, we note that $\mathscr{L}_p(x) = J_{\underline{\mu}}( \omega_{\mathcal{C}}, \omega_{\mathcal{C}})_x$ is equal to 
\[
\lambda_{\mathcal{F}^{\circ, c}_{x_1}} ( \Pi^{\opn{oc}, \leq h} \left( (x^t\cdot 1_{\mbb{Z}_p^{\times}} ) \star \omega_{g^{\circ}}^{\alpha} \times \omega_{h^{\circ}}^{\alpha} \right) ) = \lambda_{\mathcal{F}^{\circ, c}_{x_1}} ( e_{f^{\circ, c}, \alpha} \Pi^{\opn{oc}, \leq h} \left( (x^t \cdot 1_{\mbb{Z}_p^{\times}} ) \star \omega_{g^{\circ}}^{\alpha} \times \omega_{h^{\circ}}^{\alpha} \right) ) .
\]
where $e_{f^{\circ, c}, \alpha} \colon \mathscr{S}^{\dagger, \leq h}_{k} \twoheadrightarrow \mathscr{S}^{\dagger, \leq h}_{k}[\pi_{f^*}, U_p^{\circ} = \varepsilon_f(p)^{-1}\alpha_f]$ is the projection to the eigenspace for $(f^{\circ, \alpha})^{c}$ (for the Hecke operators away from $Np$ and $U_p$). Here $f^{\circ, \alpha} = \mathcal{F}^{\circ}_{x_1}$ and $\pi_{f^*}$ denotes the automorphic representation associated with $f^* = f \otimes \varepsilon_f^{-1}$. Set $\nu = (x^t\cdot 1_{\mbb{Z}_p^{\times}} ) \star \omega_{g^{\circ}}^{\alpha} \times \omega_{h^{\circ}}^{\alpha} \in \mathscr{N}^{\dagger}_k$. This class has the same $q$-expansion as the class
\[
\nu' = \nabla^t \omega^{\alpha, [p]}_{g^{\circ}} \times \omega_{h^{\circ}}^{\alpha} \in \opn{Fil}_t \mathscr{N}^{\dagger}_k
\]
where $(-)^{[p]}$ is $p$-depletion. Any filtered piece $\opn{Fil}_{b}\mathscr{N}^{\dagger}_k$ injects into the space of $p$-adic modular forms (and hence the space of $q$-expansions), and since $\Pi^{\opn{oc}, \leq h}$ is $U_p^{\circ}$-equivariant, and $\alpha_f \neq 0$, we therefore see that $\mathscr{L}_p(x)$ is equal to
\[
\mathscr{L}_p(x) = \lambda_{\mathcal{F}_{x_1}^{\circ, c}}(e_{f^{\circ, c}, \alpha} \Pi^{\opn{oc}, \leq h} \nu').
\]
By the computations in the proof of Lemma 5.9 and Lemma 5.10 in \cite{AI_LLL}, we see that
\[
e_{f^{\circ, c}, \alpha} \Pi^{\opn{oc}, \leq h} \nu' = \frac{\mathcal{E}(f, g, h)}{\mathcal{E}_1(f)} e_{f^{\circ, c}, \alpha} \Pi^{\opn{oc}, \leq h} (\nabla^t \omega_{g^{\circ}} \times \omega_{h^{\circ}} )
\]
where $\omega_{g^{\circ}} \in \mathscr{S}^{\dagger}_{\ell}$, $\omega_{h^{\circ}} \in \mathscr{S}^{\dagger}_m$ denote the classes attached to $g^{\circ}$ and $h^{\circ}$. But the argument in the right-hand side is classical: if we denote $\nu'' = e_{f^{\circ, c}} \Pi^{\opn{hol}}(\delta_{\ell}^t g^{\circ} \times h^{\circ})$ the projection to the $f^{\circ, c}$-eigenspace of the holomorphic projection of $\delta_{\ell}^t g^{\circ} \times h^{\circ}$, then we have
\[
\mathscr{L}_p(x) =  \frac{\mathcal{E}(f, g, h)}{\mathcal{E}_1(f)} \lambda_{\mathcal{F}_{x_1}^{\circ, c}}((\nu'')^{\alpha})
\]
where $(\nu'')^{\alpha}$ is the $p$-stabilisation of $\nu''$ with respect to the eigenvalue $\varepsilon_f(p)^{-1} \alpha_f$. By the specialisation property in Lemma \ref{LemmaSpecialisationOfLinearFunctional}, we have
\begin{align*}
    \mathscr{L}_p(x) &= \frac{\mathcal{E}(f, g, h)}{\mathcal{E}_1(f)} \lambda_{\mathcal{F}_{x_1}^{\circ, c}}( (\nu'')^{\alpha}) \\
    &= \frac{\mathcal{E}(f, g, h)}{\mathcal{E}_1(f)} \frac{\langle (f^{\circ, c})^{\alpha}, (\nu'')^{\alpha} \rangle_{N, p}}{\langle (f^c)^{\alpha}, (f^c)^{\alpha} \rangle_{N, p}} \\
    &= \frac{\mathcal{E}(f, g, h)}{\mathcal{E}_0(f)\mathcal{E}_1(f)} \frac{\langle f^{\circ, c}, \nu'' \rangle_{N}}{\langle f^c, f^c \rangle_{N}} \\
    &= \frac{\mathcal{E}(f, g, h)}{\mathcal{E}_0(f)\mathcal{E}_1(f)} \frac{\langle (f^{\circ})^*, \delta_\ell^t g^{\circ} \times h^{\circ} \rangle_{N}}{\langle f, f \rangle_{N}}
\end{align*}
where the third equality follows from an explicit calculation (c.f. \cite[Lemma 5.12]{AI_LLL}) and we've used the fact that $(f^{\circ})^* = f^{\circ, c}$ because the conductor of $\varepsilon_f$ is not divisible by $p$.
\end{proof}

\subsubsection{Relation to central $L$-values} \label{ChoiceOfTestdataSec}

The trilinear period in Theorem \ref{TripleProductJinterpolation} is closely related to central critical $L$-values of the Garrett--Rankin triple product $L$-function. More precisely:

\begin{theorem}[{\cite[Theorem 4.2]{DarmonRotgerGrossZagier1}}]
There exists a constant $C$ (defined over the number field generated by the Fourier coefficients of $f, g, h$ and depending on $f^{\circ}$, $g^{\circ}$, $h^{\circ}$) such that
\[
|I(f^{\circ}, g^{\circ}, h^{\circ})|^2 = \frac{C}{\pi^{2k}} L(f, g, h, c)
\]
where $c = (k+\ell+m - 2)/2$ and $L(f, g, h, -)$ is the Garrett--Rankin triple product $L$-function.
\end{theorem}

Therefore, one can view $\mathscr{L}_p$ as a universal ``square-root'' triple product $L$-function. Unfortunately, it does not seem possible (with the current methods) to construct a three-variable $p$-adic $L$-function over $(\mathcal{C} \times \mathcal{C} \times \mathcal{C})^+$ as we don't know if the linear functional $\lambda_{\mathcal{F}^{\circ, c}}$ globalises. Furthermore, it can happen that the constant $C$ is equal to zero for the test vectors $(f^{\circ}, g^{\circ}, h^{\circ})$ (for example, if $N_{\mathcal{F}} = N$ and there exists a prime dividing $N$, but not $N_g \cdot N_h$ -- see \cite[Remark 4.3]{DarmonRotgerGrossZagier1}). One has more freedom for choosing test vectors so that this constant is non-zero in small opens of $\mathcal{C}$, however we are not sure if this is possible globally. 

\subsection{Rankin--Selberg \texorpdfstring{$p$}{p}-adic \texorpdfstring{$L$}{L}-functions}

We close with an application to the construction of Rankin-Selberg $p$-adic $L$-functions in three variables. Our strategy essentially follows the one in \cite[Appendix B]{AI_LLL} circumventing the restriction on the weight space, by replacing \cite[Theorem 4.3]{AI_LLL} by our generalised version Theorem \ref{SheafVersionMainThm} as the input into the construction. Let us give a very brief sketch of this construction.

Fix a Coleman family $\mathcal{F}$ as in \S \ref{ALinearFunctionalSection} and Assumption \ref{NobleAssumption}, defined over a quasi-compact affinoid open subspace $U_1 \subset \mathcal{W}$. We first observe that Definition \ref{DefJmu} still makes sense for non cuspidal forms $\psi$.

\begin{definition}
Fix a tuple $\underline{\mu}$ as in Definition \ref{DefOfLinearFunctional}. For any open subspaces $U_2 \subseteq \mathcal{W}, V_3 \subseteq \mathcal{C}$ such that $V_1 \times U_2 \times V_3 \subseteq (V_1 \times \mathcal{W} \times \mathcal{C})^+$, we define a bilinear pairing
\begin{eqnarray*}
J_{\underline{\mu}}(-,-): \mathcal{M}^\dagger(U_2) \times \mathcal{S}^\dagger_\mathcal{C}(V_3) &\to& \mathrm{Mer}(V_1 \times U_2 \times V_3) \\
(\phi, \psi) &\mapsto& \lambda_{\mathcal{F}^{\circ, c}}(\Xi(\phi, \psi)).
\end{eqnarray*}
in exactly the same way as in Definition \ref{DefJmu}, where $\mathcal{M}^{\dagger}$ denotes the sheaf over $\mathcal{W}$ of overconvergent modular forms. This pairing glues to a pairing $\mathcal{M}^{\dagger}(\mathcal{W}) \times \mathcal{S}_{\mathcal{C}}^{\dagger}(\mathcal{C}) \to \opn{Mer}(V_1 \times \mathcal{W} \times \mathcal{C})^+$ which we continue to denote by $J_{\underline{\mu}}(-,-)$.
\end{definition}

This bilinear pairing can be used to construct three-variable Rankin--Selberg $p$-adic $L$-functions, by taking $\phi$ to be a $p$-adic family of Eisenstein series. More precisely following \cite[Lemma 3.2]{LoefflerRSNote}, there exists a $p$-adic family of Eisenstein series $E^{[p]}_{\kappa} \in \mathcal{M}^\dagger(\mathcal{W})$ with $q$-expansion
\[
E^{[p]}_{\kappa}(q) = \sum_{\substack{n \geq 1 \\ (n, p) = 1}} \left( \sum_{d | n} d^{\kappa - 1} (\zeta_N^d + (-1)^{\kappa} \zeta_N^{-d} ) \right) q^n
\]
where $\kappa$ denotes the universal character. Recall that $\omega_{\mathcal{C}} \in \mathcal{S}^\dagger_\mathcal{C}(\mathcal{C})$ is the universal class over the eigencurve.

\begin{definition}
We define the universal Rankin-Selberg $p$-adic $L$-function to be the meromorphic function on $(V_1 \times \mathcal{W} \times \mathcal{C})^+$ given by
\[ \mathscr{L}_p \defeq J_{\underline{\mu}}(E^{[p]}_{\kappa}, \omega_\mathcal{C}). \]
\end{definition}

\begin{theorem} \label{RankinSelberginterpolation}
We have the following interpolation property: for any integer $k_2 \geq 2$ and noble classical points $x_1 \in V_1$, $x_3 \in \mathcal{C}(\overline{\mathbb{Q}}_p)$ such that $k_1 \defeq w(x_1) = k_2 + w(x_3) + 2t$ for some $t \geq 0$, one has
\[ 
\mathscr{L}_p(x_1, k_2, x_3) = (\star) \cdot L^{\opn{imp}}(f, h, k_1 - 1 - t)
\]
for some explicitly computable factor $(\star)$, where 
\begin{itemize}
    \item $L^{\opn{imp}}(f, h, s)$ denotes the imprimitive Rankin--Selberg $L$-function as in \cite[Definition 2.1]{LoefflerRSNote}.
    \item $f$ (resp. $h$) is the newform associated with the specialisation of $\mathcal{F}$ at $x_1$ (resp. the eigenform corresponding to the point $x_3$).
\end{itemize}
\end{theorem}

\begin{proof}
As in Theorem \ref{TripleProductJinterpolation}, we can express classical specialisations of $\mathscr{L}_p$ as a Petersson inner product between $(f^{\circ})^*$ and $\delta_{k_2}^t E^{[p]}_{k_2} \cdot h^{\circ}$; the interpolation formula then follows from the calculations in \cite{LoefflerRSNote}.
\end{proof}

\begin{remark} 
As in \S \ref{ChoiceOfTestdataSec}, the factor $(\star)$ depends on the test data $f^{\circ}$ and $h^{\circ}$. Furthermore, we note that this three variable $p$-adic $L$-function was essentially constructed in \cite{LoefflerRSNote}, but its construction relies on the $p$-adic variation of Beilinson--Flach Euler system classes in \cite{LZ16Coleman} and imposes an ordinarity assumption on one of the families of overconvergent cuspidal forms. Our construction is completely independent of Euler systems.
\end{remark}

\appendix

\section{Classical nearly holomorphic modular forms} \label{AppendixClassicalNHMFs}

In this appendix we describe the relation between $(\mathfrak{g}, P)$-representations and $\mathcal{D}$-modules on Shimura varieties. We note that similar constructions can be found in \cite{ZLiu19}.

\subsection{Preliminaries on the flag variety} \label{PrelimsOnFLAppendix}

Let $k$ be a field of characteristic zero. Let $G$ be a reductive group over $k$ and $P \subset G$ a parabolic subgroup. Let $\opn{FL} = G/P$ be the corresponding partial flag variety. It carries an action of $G$ by left translation. 

\subsubsection{\texorpdfstring{$G$}{G}-equivariant sheaves}

We let $\mathrm{QCoh}_{G}(\opn{FL})$ be the category of $G$-equivariant quasi-coherent sheaves over $\opn{FL}$. We let $\mathrm{Rep}(P)$ be the category of algebraic representations of $P$ on $k$-vector spaces. We define a functor:
\begin{eqnarray*}
F :  \mathrm{Rep}(P)  & \rightarrow & \mathrm{QCoh}_{G}(\opn{FL})\\
V & \mapsto & \mathcal{V}=F(V)
\end{eqnarray*}
as follows. The group $G$ acts on itself by left and right translation. It follows that $\mathcal{O}_G$ carries two $G$-actions, denoted $\star_l$ and $\star_r$, given by the rule $g \star_l f(-) = f(g^{-1} \cdot -)$ and $g \star_r f(-) = f(- \cdot g)$. We consider the projection $\pi : G \mapsto G/P$.  For $(V, \rho)$ an object of $\mathrm{Rep}(P)$, we let
\[ F(V) = (\pi_\star \mathcal{O}_{G} \otimes_k V)^P, \] where the action of $P$ on the tensor product is the diagonal action given by $\star_r$ on the left factor and by $\rho$ on the right factor. The action $\star_l$ provides a $G$-equivariant structure on $F(V)$. 

On the other hand, consider $1 \cdot P=1 \in \opn{FL}$, and let $i_1 : \{1\} \rightarrow \opn{FL}$ be the inclusion. Let $\mathcal{V} \in \mathrm{QCoh}_{G}(\opn{FL})$. The stalk at $1 \in \opn{FL}$, denoted $i_1^{-1}\mathcal{V}$, is an $\mathcal{O}_{\opn{FL},1}$-module with a  semi-linear action of $P$. The maximal ideal $\mathfrak{m}_{\mathcal{O}_{\opn{FL},1}} \subset \mathcal{O}_{\opn{FL},1}$ is stable under the action of $P$ and the fibre at $1 \in \opn{FL}$, denoted $i_1^* \mathcal{V}$, is a $k$-linear representation of $P$. Hence we obtain a pullback functor: 
\begin{eqnarray*}
i_1^* : \mathrm{QCoh}_{G}(\opn{FL}) & \rightarrow & \mathrm{Rep}(P)\\
\mathcal{V} & \mapsto & i_1^* \mathcal{V}.
\end{eqnarray*} 

\begin{proposition} 
The functors $F$ and $i_1^*$ are equivalences of categories, quasi-inverse of each other.
\end{proposition}

\begin{example} \leavevmode
\begin{enumerate}
\item If $V \in \mathrm{Rep}(G)$, one can check that $F(V) = \mathcal{O}_{\opn{FL}} \otimes_k V$.
\item Let $\mathfrak{p} \subset \mathfrak{g}$ be the Lie algebras of $P \subset G$. We let $\mathfrak{p}^0 = F(\mathfrak{p}) \subseteq F(\mathfrak{g}) = \mathfrak{g}^0 = \mathcal{O}_{FL} \otimes \mathfrak{g}$. The tangent sheaf on $\opn{FL}$ is given by $T_{\opn{FL}} = F(\mathfrak{g}/\mathfrak{p}) = \mathfrak{g}^0/\mathfrak{p}^0$. 
\item Let $\mathcal{D}_{\opn{FL}}$ be the sheaf of differential operators over $\opn{FL}$. This is an object of $\mathrm{QCoh}_{G}(\opn{FL})$. There is a surjective map $\mathcal{O}_{\opn{FL}} \otimes \, \mathcal{U}(\mathfrak{g}) \rightarrow \mathcal{D}_{\opn{FL}}$ inducing an isomorphism $\mathcal{O}_{\opn{FL}} \otimes \, \mathcal{U}(\mathfrak{g})/(\mathfrak{p}^0(\mathcal{O}_{\opn{FL}} \otimes \mathcal{U}(\mathfrak{g})) \rightarrow \mathcal{D}_{\opn{FL}}$.  One thus checks that $\mathcal{D}_{\opn{FL}} = F(\mathcal{U}(\mathfrak{g})\otimes_{\mathcal{U}(\mathfrak{p})} k)$. 
\end{enumerate}
\end{example}

\subsubsection{\texorpdfstring{$G$}{G}-equivariant \texorpdfstring{$\mathcal{D}_{\opn{FL}}$}{D}-modules}

As above, let $\mathcal{D}_{\opn{FL}}$ be the sheaf of differential operators over $\opn{FL}$. We let $\mathcal{D}_{\opn{FL}}\mathrm{-Mod}_G$ be the category of $G$-equivariant $\mathcal{D}_{\opn{FL}}$-modules. Its objects are objects of $\mathrm{QCoh}_{G}(\opn{FL})$ together with an action of $\mathcal{D}_{\opn{FL}}$ which is compatible with the $G$-equivariant structure.

We let $\mathrm{Rep}(( \mathfrak{g},P))$ be the category of $k$-vector spaces $V$ equipped with an algebraic representation of $P$, and a representation of $\mathfrak{g}$ satisfying the following compatibility: 
\begin{enumerate}
\item For any $v \in V, g \in \mathfrak{g}$ and $p \in P$, we have $(\mathrm{Ad}(p) g) v = p g p^{-1} v$,
\item The action of $P$ induces an action of the Lie algebra $\mathfrak{p}$. This action coincides with the  restriction to $\mathfrak{p}$ of the action of $\mathfrak{g}$. 
\end{enumerate}

\begin{lemma} \label{FEnrichedLemmaAppendix}
The functor $F$ can be enriched to a functor: 
\begin{eqnarray*}
F :  \mathrm{Rep}((\mathfrak{g},P))  & \rightarrow & \mathcal{D}_{\opn{FL}}\mathrm{-Mod}_G \\
V & \mapsto & \mathcal{V}=F(V)
\end{eqnarray*}
\end{lemma}

\begin{proof} 
Let $(V, \rho)$ be an object of $\mathrm{Rep}((\mathfrak{g},P))$ (where $\rho$ stands for the action of both $P$ and $\mathfrak{g}$). 
We first observe that the $G$-equivariant structure induces an action $\star_l$ of  $\mathfrak{g}$ on $F(V)$ by derivations. This action can be extended linearly to an action of $\mathfrak{g}^0 $ which extends the action by derivations on $\mathcal{O}_{\opn{FL}}$. However it will not in general factor through an action of $T_{\opn{FL}}$ on $F(V)$ and induce a $\mathcal{D}_{\opn{FL}}$-module structure. 

Instead, we construct a second action of $\mathfrak{g}$ on $F(V)$ by using the action $\rho$ of $\mathfrak{g}$ on $V$ as follows. Let $f \in \mathcal{O}_G \otimes V$. We let $h \in \mathfrak{g}$ and define $h\star f = [g \mapsto g^{-1} h g \star_\rho f(g)]$.  We check that this descends to an action on $F(V)$ as follows. For $p \in P$ we have:
\begin{eqnarray*}
p (\star_r \otimes \rho)   (h \star f)(g)& = &p (h\star f) (g p)  \\
&=& p p^{-1} g^{-1} h g p f(gp)\\
&=& h \star (p (\star_r \otimes \rho) f)(g).
\end{eqnarray*}
We check that this action is $G$-equivariant as follows. For $t$ in $G$ we have
\begin{eqnarray*}
t \cdot h \star (t \star_l f)(g) & = &  t ht^{-1} \star f(t^{-1} g ) \\
&=& g^{-1} t h t^{-1} g f(t^{-1} g ) \\
&=& t \star_l (  h \star f) (g).
\end{eqnarray*}
The $\star$ action is $\mathcal{O}_{\opn{FL}}$-linear. 
The difference $\star_{\mathcal{D}} = \star_l - \star$ defines an action  by derivations of $\mathfrak{g}$ on $F(V)$ which extends the  action by derivations on $\mathcal{O}_{\opn{FL}}$. 
We claim that  the action $\star_l - \star$ induces a $\mathcal{D}_{\opn{FL}}$-module structure on $F(V)$.
The main step is  to check that $\mathfrak{p}^0$ acts trivially.  As the restriction of the action $\star_{\mathcal{D}}$ of $\mathfrak{p}^0$ is $\mathcal{O}_{\opn{FL}}$-linear and $G$-equivariant it suffices to check 
that the induced action  is trivial on $i_1^{*}$. The action $\star_l$ restricts to the natural action of $P$ on $V$ via $i_1^*$ and induces an action of the Lie algebra $\mathfrak{p}$. The action $\star$ restricts by construction to the natural action of $\mathfrak{p}$ on $V$, induced by the inclusion $\mathfrak{p} \hookrightarrow \mathfrak{g}$. By the compatibility imposed on both actions, the difference is the  trivial action of $\mathfrak{p}$. 
\end{proof}

\begin{lemma} 
The functor  $i_1^{*}$ can be enriched to a functor 
\begin{eqnarray*}
i_1^* \colon \mathcal{D}_{\opn{FL}}\mathrm{-Mod}_G & \rightarrow & \mathrm{Rep}((\mathfrak{g},P))\\
\mathcal{V} & \mapsto & i_1^* \mathcal{V}
\end{eqnarray*}
\end{lemma}
\begin{proof} Let $\mathcal{V}$ be an object of $\mathcal{D}_{\opn{FL}}\mathrm{-Mod}_G$. The action of $G$ induces an action $\star_l$ of $\mathfrak{g}$. The $\mathcal{D}_{\opn{FL}}$-module action induces an action $\star_{\mathcal{D}}$ of $\mathfrak{g}$. Both actions extend the action by derivations on $\mathcal{O}_{\opn{FL}}$, therefore the difference is an $\mathcal{O}_{\opn{FL}}$-linear action which induces an action of $\mathfrak{g}$ on $i_1^* \mathcal{V}$.
\end{proof}

We immediately obtain the following proposition.

\begin{proposition} 
The categories $\mathcal{D}_{\opn{FL}}\mathrm{-Mod}_G$ and $\mathrm{Rep}((\mathfrak{g},P))$ are equivalent via the quasi-inverse functors $i_1^*$ and $F$.
\end{proposition}

\begin{example} \leavevmode
\begin{enumerate}
\item The forgetful functor $\mathrm{Rep}((\mathfrak{g},P))  \rightarrow \mathrm{Rep}(P)$ has a left adjoint $V \mapsto \mathcal{U}(\mathfrak{g}) \otimes_{\mathcal{U}(\mathfrak{p})} V$. Similarly, the forgetful functor $ \mathcal{D}_{\opn{FL}}\mathrm{-Mod}_G \rightarrow \mathrm{QCoh}_{G}(\opn{FL})$ has a left adjoint $\mathcal{V} \mapsto \mathcal{D}_{\opn{FL}} \otimes_{\mathcal{O}_{\opn{FL}}} \mathcal{V}$. 
This implies that for any $\mathcal{V} = F(V)$ with $V \in \mathrm{Rep}(P)$, one has $\mathcal{D}_{\opn{FL}} \otimes_{\mathcal{O}_{\opn{FL}}} \mathcal{V} = F(\mathcal{U}(\mathfrak{g}) \otimes_{\mathcal{U}(\mathfrak{p})} V)$. 
\item Let $E,G \in \mathrm{Rep}(P)$ be finite dimensional representations. Let $\mathcal{E}=F(E)$, $\mathcal{G}= F(G)$.  Let $\mathrm{Diff}_G(\mathcal{E}, \mathcal{G})$ be the space of  $G$-equivariant differential operators $\mathcal{E} \rightarrow \mathcal{G}$. By definition, a $G$-equivariant differential operator is a  global $G$-invariant section of 
$\mathcal{F} \otimes_{\mathcal{O}_{\opn{FL}}} \mathcal{D}_{\opn{FL}} \otimes_{\mathcal{O}_{\opn{FL}}} \mathcal{E}^\vee$, or equivalently a 
$G$-equivariant map 
$\mathcal{F}^\vee \rightarrow \mathcal{D}_{\opn{FL}} \otimes_{\mathcal{O}_{\opn{FL}}} \mathcal{E}^\vee $. We deduce that 
\begin{eqnarray*}
\mathrm{Diff}_G(\mathcal{E}, \mathcal{G}) &=& \mathrm{Hom}_{\mathrm{Rep}(P)}( F^\vee , \mathcal{U}(\mathfrak{g}) \otimes_{\mathcal{U}(\mathfrak{p})} E^\vee) \\
&=& \mathrm{Hom}_{\mathrm{Rep}((\mathfrak{g},P))}( \mathcal{U}(\mathfrak{g}) \otimes_{\mathcal{U}(\mathfrak{p})}F^\vee , \mathcal{U}(\mathfrak{g}) \otimes_{\mathcal{U}(\mathfrak{p})} E^\vee).
\end{eqnarray*}
\end{enumerate}

\end{example}

\subsubsection{The case of \texorpdfstring{$\mathrm{GL}_2$}{GL(2)}}

In this subsection, we freely use the notation in \S \ref{TheMainResultIntroSubSec}. Let $\overline{\mathfrak{n}} \subset \mathfrak{gl}_2$ denote the lower triangular nilpotent Lie algebra, and let $w_0$ denote the longest Weyl element of $G=\opn{GL}_2$. We will consider certain $(\mathfrak{g}, \overline{B})$-representations which belong to category $\mathcal{O}$. Since we are interested in representations of the lower triangular Borel subgroup (but define the positive roots to be those which lie in $B$) our conventions are slightly different from those found in the literature. For example, our convention is that Verma modules are \emph{lowest} weight modules. 

\begin{definition}
For any $\kappa \in X^*(T)$, the Verma module of lowest weight $w_0\kappa$ is $M_\kappa = \mathcal{U}(\mathfrak{g}) \otimes_{\mathcal{U}(\overline{\mathfrak{b}})} w_0\kappa$.  We let $M_\kappa^\vee$ be its dual in category $\mathcal{O}$.\footnote{Here is how duality is defined. Let $M \in \mathcal{O}$. Then $M = \oplus_{\chi \in X^*(T)_{\Q}} M_\chi$. We let $M^\vee = \oplus_{\chi \in X^*(T)_{\Q}} M_\chi^\vee \subset \mathrm{Hom}(M, \Q)$. The action on $M^\vee$ is given by $g f (-) = f( {^tg} \cdot -)$.} 
\end{definition}

If $\kappa \in X^*(T)^+$, let $V_\kappa$ be the irreducible representation of $G$ of highest weight $\kappa$. Then there is an exact sequence 
\[
0 \rightarrow V_\kappa \rightarrow M^\vee_\kappa \rightarrow M_{w_0\kappa -2\rho} \rightarrow 0 .
\]
If $\kappa \notin X^*(T)^+$ then $M^\vee_\kappa = M_\kappa$ is irreducible. 
The character of $M_\kappa^\vee$ is $\sum_{n \geq 0} w_0\kappa +2n \rho$ and we have a filtration $M_\kappa^\vee = \bigcup_{r \geq 0} \mathrm{Fil}_{r} M_\kappa^\vee$ where $\mathrm{Fil}_{r} M_\kappa^\vee = M_\kappa^\vee[\overline{\mathfrak{n}}^{r+1}]$ is the subspace of elements killed by $\overline{\mathfrak{n}}^{r+1}$. Moreover, there is an isomorphism of $\overline{B}$-modules: 
\[
M_\kappa^\vee[\overline{\mathfrak{n}}^{r+1}] \cong \mathrm{Sym}^r \opn{St} \otimes \Q(w_0\kappa + r(1; -1))
\]
where $\opn{St}$ denotes the standard representation of $G$, and $\Q(w_0\kappa + r(1; -1))$ is the one-dimensional line on which $\overline{B}$ acts through the character $w_0\kappa + r(1; -1)$.

 We let $\oscr_{\overline{B}}$ be the space of functions on $\overline{B}$. We equip it with an action of $\overline{B}$ via $b  \star_l f(-) = f(b^{-1} \cdot -)$. We can turn it into a $(\mathfrak{g},\overline{B})$-module  as follows. We have an open immersion $\overline{B} \hookrightarrow G/U$, where $U$ denotes the upper triangular unipotent, and the left translation action of $G$ on $G/U$ induces an action of $\mathfrak{g}$ stabilising $\oscr_{\overline{B}}$.  We also have  another action of $\overline{B}$ via $b\star_r f(-)= f(- \cdot b)$. 

We let $x, t_1, t_2$ be coordinates on $\overline{B}$, where the universal element of $\overline{B}$ is written as 
$$\begin{pmatrix} 
         t_2 & 0 \\
         x & t_1 \\
      \end{pmatrix}.$$
Using this description, we have $\oscr_{\overline{B}} = k[x, t_1, t_2, t_1^{-1}, t_2^{-1}]$ and the actions $\star_r$ and $\star_l$ of $\overline{B}$ and that of $\mathfrak{g}$ are easy to describe. Indeed, the infinitesimal action is given by
\begin{align*}
    \begin{pmatrix}
            0 & 1 \\
            0 & 0 \\
   \end{pmatrix} &\mapsto x t_1 t_2^{-1} \partial_{t_1} - x \partial_{t_2}, \quad \quad \begin{pmatrix}
            1 & 0 \\
            0 & 0 \\
   \end{pmatrix} \mapsto -\partial_{t_2} \\
   \begin{pmatrix}
            0 & 0 \\
            0 & 1 \\
   \end{pmatrix} &\mapsto -(\partial_x +\partial_{t_1}), \quad \quad \; \quad \quad \begin{pmatrix}
            0 & 0 \\
            1 & 0 \\
   \end{pmatrix} \mapsto -t_2 \partial_{x} .
\end{align*}
The action of the lower triangular unipotent $\overline{U}$ via $\star_r$ or $\star_l$ induces a filtration $\oscr_{\overline{B}} = \bigcup_{r \geq 0} \mathrm{Fil}_r(\oscr_{\overline{B}})$ where $\mathrm{Fil}_r(\oscr_{\overline{B}})$ is the sub-module of vectors killed by $\overline{\mathfrak{n}}^{r+1}$. This is the subspace of $\oscr_{\overline{B}}$ of elements with degree in $x$ bounded by $r$. The action of $T$ via $\star_r$ commutes with the actions of $\mathfrak{g}$ and $\overline{B}$ via $\star_l$ and the space   $\oscr_{\overline{B}}$ decomposes as a sum of eigenspaces: 

\begin{lemma}\label{lem-Brep}  
Via the action of $T$ by $\star_r$, the space $\oscr_{\overline{B}}$ decomposes as $\oplus_{\kappa \in X^*(T)} (M_{\kappa})^\vee$ with $(M_{\kappa})^\vee = \mathscr{O}_{\overline{B}}[-w_0\kappa] \defeq \opn{Hom}_T(-w_0\kappa, \mathscr{O}_{\overline{B}})$. 
\end{lemma}

Let $\pi : G \rightarrow G/\overline{B}$ be the canonical $\overline{B}$-torsor. It is a $G$-equivariant torsor via the left $G$-action. 

\begin{proposition} 
 The $\overline{B}$-torsor structure on $G$ induces a decomposition, 
$$\pi_* \mathcal{O}_G = \oplus_{\kappa \in X^*(T)} \mathcal{O}_G[-w_0\kappa] =  \oplus_{\kappa \in X^*(T)}  \bigcup_{r \geq 0} \mathrm{Fil}_r \mathcal{O}_G [-w_0\kappa]$$
where $\mathcal{O}_G[-w_0\kappa] = F(M_{\kappa}^\vee)$ and $\mathrm{Fil}_r \mathcal{O}_G [-w_0\kappa] = F(\mathrm{Fil}_{r} M_{\kappa}^\vee)$. We have a $G$-equivariant connection  $\nabla : \pi_*\mathcal{O}_G \rightarrow \pi_*\mathcal{O}_G \otimes \Omega^1_{\opn{FL}/\Q}$ inducing a connection:
$$\nabla_\kappa : \mathcal{O}_G[-w_0\kappa] \rightarrow \mathcal{O}_G[-w_0\kappa] \otimes \Omega^1_{\opn{FL}/\Q}.$$
Moreover, we have $\nabla_{\kappa} (  \mathrm{Fil}_r \mathcal{O}_G [-w_0\kappa]) \subseteq \mathrm{Fil}_{r+1} \mathcal{O}_G [-w_0\kappa] \otimes  \Omega^1_{\opn{FL}/\Q}$ for all $r \geq 0$.
\end{proposition}
\begin{proof}  We have  an isomorphism of $G$-equivariant sheaves $\pi_* \mathcal{O}_G = F((\oscr_{\overline{B}}, \star_l))$, where $\oscr_{\overline{B}}$ is equipped with the $\star_l$-action of $\overline{B}$. The $\star_r$-action of $\overline{B}$ corresponds to  the $\overline{B}$-torsor structure. All the properties can be read from the properties of $\oscr_{\overline{B}}$. 
\end{proof}

\subsection{Shimura varieties} 

Let $(G,X)$ be a Shimura datum, and let $\{S_K\}_{K \subseteq G(\mathbb{A}_f)}$ be the tower of Shimura varieties, defined over the reflex field $E$. We let $P=P_{\mu}^{\opn{std}}$ be the parabolic attached to (a representative of) the  cocharacter $\mu$ of the Shimura datum and let $\opn{FL} = G/P$ be the flag variety, defined over $E$. We let $Z(G)$ be the center of $G$ and denote by $Z_s(G)$ the largest subtorus which is $\mathbb{R}$-split but has no subtorus split over $\mathbb{Q}$. Set $G^c = G/Z_s(G)$, $P^c = P/Z_s(G)$, and define $M^c= M/Z_s(G)$ where $M$ is the Levi of $P$. Note $\opn{FL} = G^c/P^c$.

Let $G_{\opn{dR},K} \rightarrow S_K$ denote the de Rham $G^c$-torsor (see \cite[\S III.3]{MR1044823}) and $P_{\opn{dR}, K} \rightarrow S_K$ its $P^c$-reduction. We define $M_{\opn{dR}, K} = P_{\opn{dR}, K} \times^{P^c} M^c$ to be the push out of the $P^c$-torsor $P_{\opn{dR}, K}$ to an $M^c$-torsor $M_{\opn{dR}, K}$. We have a diagram: 
\begin{eqnarray*}
\xymatrix{ & \ar[ld]_{p} G_{\opn{dR},K} \ar[rd]^{q} &  \\
S_K & & \opn{FL}}
\end{eqnarray*}
characterised by the property that the pullback of the $P^c$-torsor $G^c \rightarrow \opn{FL}$ via $q$ is $P_{\opn{dR}, K} \rightarrow G_{\opn{dR}, K}$. 
We obtain a functor: 
\begin{eqnarray*}
\opn{VB}_K \colon \mathrm{QCoh}_{G^c}(\opn{FL}) & \rightarrow & \mathrm{QCoh}(S_K) \\
\mathcal{V} & \rightarrow & \mathrm{H}^0(G^c, p_* q^* \mathcal{V})
\end{eqnarray*}

\begin{lemma} 
We have $\opn{VB}_K(\Omega^1_{\opn{FL}/E}) = \Omega^1_{S_K/E}$ and $\opn{VB}_K( \mathcal{D}_{\opn{FL}}) = \mathcal{D}_{S_{K}}$. 
\end{lemma}
\begin{proof} 
The first statement $\opn{VB}_K(\Omega^1_{\opn{FL}/E}) = \Omega^1_{S_K/E}$ follows from Kodaira--Spencer theory. Let us  recall the argument. 
Recall that  $\mathfrak{g}$ is the  Lie algebra of $G$, and $\mathfrak{g}^{ 0}$ is the associated vector bundle with flat connection over $\opn{FL}$. It carries a filtration $\mathfrak{n}^{0} \subseteq \mathfrak{p}^{0} \subseteq \mathfrak{g}^{0}$, where $\mathfrak{m}^0 = \mathfrak{p}^{0}/\mathfrak{n}^{0}$ is associated to the Levi $\mathfrak{m}$ and $\mathfrak{g}^{0}/\mathfrak{p}^{0}$ is the tangent sheaf. We explain how we can recover the isomorphism $\mathfrak{g}^0/\mathfrak{p}^0 \xrightarrow{\sim} \mathcal{T}_{\opn{FL}}$ from the connection. 

Passing to dual vector bundles, and using Griffiths transversality and the connection, we obtain a map  $\mathfrak{g}^{0,\vee}/\mathfrak{n}^{0,\vee} \rightarrow \mathfrak{m}^{0,\vee} \otimes_{\mathcal{O}_{\opn{FL}}} \Omega^1_{\opn{FL}/E} $ or equivalently a map $\mathfrak{g}^{0,\vee}/\mathfrak{n}^{0,\vee} \otimes_{\mathcal{O}_{\opn{FL}}}\mathfrak{m}^0 \rightarrow \Omega^1_{\opn{FL}/E} $. This map factors through the isomorphism $\mathfrak{g}^{0,\vee}/\mathfrak{n}^{0,\vee} \xrightarrow{\sim} \Omega^1_{\opn{FL}/E}$ via the adjoint action of $\mathfrak{m}$ on $\mathfrak{g}/\mathfrak{p}$ which induces a map $\mathfrak{g}^{0,\vee}/\mathfrak{n}^{0,\vee} \otimes_{\mathcal{O}_{\opn{FL}}} \mathfrak{m}^0 \rightarrow \mathfrak{g}^{0,\vee}/\mathfrak{n}^{0,\vee}$. The vector bundle $\opn{VB}_K(\mathfrak{g}^{0,\vee})$ carries an integrable connection, and we thus get a map 
\[
\opn{VB}_K(\mathfrak{g}^{0,\vee}/\mathfrak{n}^{0,\vee}) \otimes_{\mathcal{O}_{S_K}} \opn{VB}_K( \mathfrak{m}^0)\rightarrow \Omega^1_{S_K/E}
\]
which factors through the isomorphism $\opn{VB}_K(\mathfrak{g}^{0,\vee}/\mathfrak{n}^{0,\vee}) \cong \Omega^1_{S_K/E}$. 

We now turn to the identification $\opn{VB}_K( \mathcal{D}_{\opn{FL}}) = \mathcal{D}_{S_{K}}$. We first observe that by the first point we have a map $\mathcal{O}_{S_K} \oplus \mathcal{T}_{S_K} \rightarrow \opn{VB}_K( \mathcal{D}_{\opn{FL}}) $. We will show that $\opn{VB}_K( \mathcal{D}_{\opn{FL}})$ carries an algebra structure, and is generated by $\mathcal{O}_{S_K}$ and $\mathcal{T}_{S_K}$ subject to the usual relations. We recall that if $f \colon X \rightarrow Y$ is an \'etale map of $E$-schemes then the pullback map $ f^* \Omega^1_{Y/E} \rightarrow \Omega^1_{X/E}$ is an isomorphism. Similarly, the natural map $\mathcal{D}_X \rightarrow f^* \mathcal{D}_Y$ is an isomorphism. 

The $G^c$-torsor $G_{\opn{dR},K} \rightarrow S_K$ has sections \'etale locally. Let $U \rightarrow S_K$ be an \'etale map and let $s : U \rightarrow G_{\opn{dR}, K}$ be a section. By a deformation theory argument, one can choose $s$ so that the map $q \circ s : U \rightarrow \opn{FL}$ is \'etale. We deduce that $\opn{VB}_K(\mathcal{D}_{\opn{FL}/E})\vert_U 
 = (q\circ s)^* \mathcal{D}_{\opn{FL}/E} = \mathcal{D}_{U/E}$. 
 We therefore obtain an algebra structure on $\opn{VB}_K(\mathcal{D}_{\opn{FL}/E})\vert_U $. This algebra structure glues using the $G^c$-equivariant action.  Indeed, let $s' : U \rightarrow G_{\opn{dR}, K}$ be another section. Then there exists $g \in G^c(U)$ such that $s' = g \circ s$. Let $\opn{act} : G^c \times \opn{FL} \rightarrow \opn{FL}$ be the action map and let $p : G^c \times \opn{FL} \rightarrow \opn{FL}$ be the projection. We have a map $U \stackrel{(g, q \circ s)}\rightarrow G^c \times \opn{FL}$ and $\opn{act} ( (g, q \circ s)) = q \circ s'$, while $p ( (g, q\circ s)) = q \circ s$. The equivariant structure on 
$\mathcal{D}_{\opn{FL}/E}$ is an isomorphism of algebras (satisfying a certain cocycle condition) $\opn{act}^* \mathcal{D}_{\opn{FL}/E} \xrightarrow{\sim} p^* \mathcal{D}_{\opn{FL}/E}$. We thus get an isomorphism: $ (g, q\circ s)^* \opn{act}^* \mathcal{D}_{\opn{FL}/E} \xrightarrow{\sim}  (g, q \circ s)^* p^* \mathcal{D}_{\opn{FL}/E}$. 
\end{proof}

We immediately obtain the following corollary:

\begin{corollary} 
The functor $\opn{VB}_K$ induces a functor $\opn{VB}_K \colon \mathcal{D}_{\opn{FL}}\mathrm{-Mod}_{G^c} \rightarrow  \mathcal{D}_{S_K}\mathrm{-Mod}$. 
\end{corollary}

Using the equivalences of categories in \S \ref{PrelimsOnFLAppendix}, we have functors (that we keep denoting by $\opn{VB}_K$):
\begin{eqnarray*}
\mathrm{Rep}(P^c) & \rightarrow & \mathrm{QCoh}(S_K) \\
\mathrm{Rep}((\mathfrak{g}^c, P^c)) & \rightarrow & \mathcal{D}_{S_K}\mathrm{-Mod} \\
\end{eqnarray*}

Let $S_{K,\Sigma}^{\opn{tor}}$ be a toroidal compactification, where, as usual, $\Sigma$ denotes a cone decomposition. These functors can be extended to functors $\opn{VB}^{\opn{can}}_K$ : 
\begin{eqnarray*}
\mathrm{Rep}(P^c) & \rightarrow & \mathrm{QCoh}(S^{\opn{tor}}_{K,\Sigma}) \\
\mathrm{Rep}((\mathfrak{g}^c, P^c)) & \rightarrow & \mathcal{D}_{S^{\opn{tor}}_{K,\Sigma}}\mathrm{-Mod} \\
\end{eqnarray*}
where $\mathcal{D}_{S^{\opn{tor}}_{K,\Sigma}}$ is the sheaf of logarithmic differential operators.

\subsection{Application to modular curves} \label{ApplicationToModularCurvesAppendix}

We now specialise to the setting of modular curves, where $(G,X) = (\mathrm{GL}_2, \mathcal{H})$ with $\mathcal{H}$ the upper and lower half-plane. We change notation and let $X_K = S_{K,\Sigma}^{\opn{tor}}$. We take a representative of $P$ to be the lower triangular Borel $\overline{B}$ with Levi $M=T$ the diagonal torus. 

\subsubsection{Modular forms}
For any $\kappa \in X^*(T)$, we let $\omega^{\kappa}_K$ or simply $\omega^\kappa$ be $\opn{VB}^{\opn{can}}_K(w_0\kappa)$. The space of weight $\kappa$, level $K$ modular forms is $\mathrm{H}^0(X_K, \omega^\kappa)$. 

\begin{remark} 
Let $E \rightarrow X_K$ be the universal semi-abelian scheme. Let $\omega_E$ be the conormal sheaf, and $\mathrm{Lie}(E)$ the Lie algebra sheaf dual to $\omega_E$. We adopt similar notations for the dual semi-abelian scheme $E^D$. We let $\mathcal{H}_E = \mathcal{H}_{1}^{\opn{dR}}(E)$ be the relative log-de Rham homology of the Kuga-Sato compactification of $E$. 

Let $\opn{St}$ be the standard representation of $G$. By construction $\opn{VB}_K^{\opn{can}}(\opn{St}) = \mathcal{H}_E$ equipped with the Gauss-Manin connection. The $\overline{B}$-filtration $0 \rightarrow \Q((-1;1)) \rightarrow \opn{St} \rightarrow \Q((1;0))\rightarrow 0$ gives the Hodge filtration:
$$0 \rightarrow \omega_{E^D} \rightarrow \mathcal{H}_{E} \rightarrow \mathrm{Lie}(E) \rightarrow 0.$$
We deduce that $\omega_{E^D} = \omega^{(1;0)}$ and $\omega_E = \omega^{(1; -1)}$. 
\end{remark}

Let $\pi \colon M_{\opn{dR},K} \rightarrow X_K$ be the $T$-torsor over $X_K$. We have $\pi_* \mathcal{O}_{M_{\opn{dR},K}} = \oplus_{\kappa \in X^*(T)} \omega^\kappa$, and hence we deduce that 
\[
\mathrm{H}^0(M_{\opn{dR},K}, \mathcal{O}_{M_{\opn{dR},K}}) = \bigoplus_{\kappa \in X^*(T)} \mathrm{H}^0(X_K, \omega^\kappa)
\]
and that $\mathrm{Hom}_{T}(-w_0\kappa, \mathrm{H}^0(M_{\opn{dR},K}, \mathcal{O}_{M_{\opn{dR},K}})) = \mathrm{H}^0(X_K, \omega^\kappa)$.

\subsubsection{Nearly holomorphic modular forms}

We now discuss nearly holomorphic modular forms. For an integer $r \geq 0$ and $\kappa \in X^*(T)$, let $\mathcal{H}_{\kappa,r} = \opn{VB}^{\opn{can}}_{K} ( \mathrm{Fil}_r(M_{\kappa}^\vee))$ and $\mathcal{H}_{\kappa} = \opn{VB}^{\opn{can}}_K(M_{\kappa}^\vee)$. We have a connection $\nabla_\kappa : \mathcal{H}_{\kappa} \rightarrow \mathcal{H}_{\kappa} \otimes \Omega^1_{X_K/\Q}(\mathrm{log}\,D)$, where $D$ denotes the boundary divisor of $X_K$. This connection satisfies $\nabla_\kappa ( \mathcal{H}_{\kappa, r} ) \subseteq  \mathcal{H}_{\kappa, r+1} \otimes \Omega^1_{X_K/\Q}(\mathrm{log}\,D)$.

\begin{remark} 
We have $\mathcal{H}_{\kappa,r} = \mathrm{Sym}^r \mathcal{H}_{E} \otimes \omega^{\kappa-r((1;0))}$. Furthermore, we also have $\Omega^1_{X_K/\Q}(\mathrm{log}\,D) \cong \omega^{2\rho}$ by the Kodaira-Spencer isomorphism. 
\end{remark}

The space of nearly holomorphic modular forms of weight $\kappa$, degree $r$ and level $K$ is precisely the space $\mathrm{H}^0(X_K, \mathcal{H}_{\kappa,r})$. Using the Kodaira--Spencer isomorphism, we have a map 
\[
\nabla_\kappa : \mathrm{H}^0(X_K, \mathcal{H}_{\kappa,r}) \rightarrow \mathrm{H}^0(X_K, \mathcal{H}_{\kappa + 2\rho,r+1}).
\]
It is also possible to describe nearly holomorphic modular forms as functions on the torsor $P_{\opn{dR},K}$. Let $\pi' : P_{\opn{dR},K} \rightarrow X_{K}$ be the $\overline{B}$-torsor over $X_K$. 

\begin{proposition} \label{OPdrKPropAppendix}
The following properties are satisfied:
\begin{enumerate}

\item The sheaf $\pi'_* \mathcal{O}_{P_{\opn{dR},K}}$ has an action of $\overline{B}$ and the action of the unipotent radical $\overline{U}$ yields a filtration $\mathrm{Fil}_r \pi'_* \mathcal{O}_{P_{\opn{dR},K}} = \pi'_* \mathcal{O}_{P_{\opn{dR},K}}[\overline{\mathfrak{n}}^{r+1}]$. 
\item The action of $T$ yields a decomposition 
\[
\pi'_* \mathcal{O}_{P_{\opn{dR},K}} = \bigoplus_{\kappa \in X^*(T)} \mathcal{H}_\kappa
\]
compatible with the filtration. 
\item  We have a connection $\nabla : \pi'_* \mathcal{O}_{P_{\opn{dR},K}} \rightarrow \pi'_* \mathcal{O}_{P_{\opn{dR},K}} \otimes_{\mathcal{O}_{X_K}} \Omega^1_{X_K/\Q}(\mathrm{log} \, D)$ which commutes with the action of $T$,
\item We have  $\mathrm{Hom}_{\overline{B}} ( 2\rho, \pi'_* \mathcal{O}_{P_{\opn{dR},K}}) = \Omega^1_{X_K/\Q}(\mathrm{log}\,D)$.
\end{enumerate}
\end{proposition}
\begin{proof} The $(\mathfrak{g},\overline{B})$-module $\oscr_{\overline{B}}$ corresponds to the tautological torsor $G \rightarrow G/\overline{B}$ which corresponds to $P_{\opn{dR}, K}$. \end{proof}

We now pass to the limit over $K$ and let $P_{\opn{dR}} = \lim_K P_{\opn{dR},K}$ and $X = \lim_K X_K$.

\begin{prop}  \label{PropNearlyOverconvergent}
The following properties are satisfied:
\begin{enumerate}
\item The space $ \mathrm{H}^0(P_{\opn{dR}}, \mathcal{O}_{P_{dR}})$ has commuting actions of $\overline{B}$ and $G(\mathbb{A}_f)$.
\item The $\overline{U}$-action on $ \mathrm{H}^0(P_{\opn{dR}}, \mathcal{O}_{P_{\opn{dR}}})$ yields an increasing filtration 
\[
\mathrm{Fil}_\bullet \mathrm{H}^0(P_{\opn{dR}}, \mathcal{O}_{P_{\opn{dR}}}) \subset \mathrm{H}^0(P_{\opn{dR}}, \mathcal{O}_{P_{\opn{dR}}}).
\] 
\item For each $\kappa \in X^*(T)$, we have $$\mathrm{Hom}_{T}(-w_0\kappa,  \mathrm{H}^0(P_{\opn{dR}}, \mathcal{O}_{P_{\opn{dR}}})) = \mathrm{H}^0(X, \mathcal{H}_{\kappa})$$ and $$\mathrm{Hom}_{T}(-w_0\kappa,  \mathrm{Fil}_r\mathrm{H}^0(P_{\opn{dR}}, \mathcal{O}_{P_{\opn{dR}}})) = \mathrm{H}^0(X, \mathcal{H}_{\kappa, r})$$
\item We have a $G(\mathbb{A}_f) \times \overline{B}$-equivariant isomorphism $\mathcal{O}_{P_{\opn{dR}}}\otimes_{\mathcal{O}_X} \Omega^1_{X/\Q} \simeq \mathcal{O}_{ P_{\opn{dR}}} \{2\rho\}$, where the latter means we twist the action of $\overline{B}$ by the character $2\rho$.
\item We have a $T \times G(\mathbb{A}_f)$-equivariant derivation  
\[
\nabla :  \mathrm{H}^0(P_{\opn{dR}}, \mathcal{O}_{P_{\opn{dR}}})  \rightarrow \mathrm{H}^0(P_{\opn{dR}}, \mathcal{O}_{P_{\opn{dR}}}) \otimes_\Q \Q\{2\rho\}
\]
which satisfies $\nabla (\mathrm{Fil}_r \mathrm{H}^0(P_{\opn{dR}}, \mathcal{O}_{P_{\opn{dR}}})) \subseteq \mathrm{Fil}_{r+1} \mathrm{H}^0(P_{\opn{dR}}, \mathcal{O}_{P_{\opn{dR}}}) \otimes_\Q \Q\{2\rho\}$ for all $r \geq 0$.
\end{enumerate}
\end{prop}
\begin{proof}
This follows immediately from Proposition \ref{OPdrKPropAppendix}.
\end{proof}

\section{Glossary of notation} \label{GlossaryNotAppendix}

In the following table, the reader will find a list of some objects appearing frequently in this article, together with how they are denoted in \S 2 (where there is no assumption on the levels in order to state our main result in its most general form) and in the rest of the article (i.e. at hyperspecial level).

\vspace{0.3cm}
\begin{center}
    \textbf{Comparison of notation between \S \ref{PadicMFsMainresultSec} and the rest of the article}
\end{center}
\vspace{0.3cm}

\begin{center}
    \begin{tabular}{c|c|c}
     Object description & General level (\S \ref{PadicMFsMainresultSec}) & Hyperspecial level (\S \ref{SectionNOCMFs}, \S \ref{GMinterpolationSection}, \S \ref{AdditionalStructuresSection})  \\
     \hline 
     Level subgroup at $p$  & $K_p \subset \opn{GL}_2(\mbb{Q}_p)$ compact open & $K_p = \opn{GL}_2(\mbb{Z}_p)$ (omitted from notation) \\
     Level subgroup away from $p$  & $K^p$ & $K^p$ (omitted from notation) \\
     Level of ordinary locus  & $K_{p, P} K^p$ & $(P')^{\opn{int}}K^p$ (omitted from notation) \\
     Level of Igusa tower  & $U_{K_{p, P}}K^p$ & $U_{P'}^{\opn{int}}K^p$ (omitted from notation) \\
     Torsor group of the Igusa tower  & $M_{K_{p, P}}$ & $T(\mbb{Z}_p)$ \\
     The ordinary locus  & $\mathfrak{IG}_{K_{p, P}K^p}$ & $\mathfrak{IG}_{(P')^{\opn{int}}K^p} = \mathfrak{X}^{\opn{ord}}_{\opn{GL}_2(\mbb{Z}_p)K^p} = \mathfrak{X}_{\opn{ord}}$ \\
     The Igusa tower  & $\mathfrak{IG}_{U_{K_{p, P}}K^p}$ & $\mathfrak{IG}_{U_{P'}^{\opn{int}}K^p} = \mathfrak{IG}_{\infty}$ \\
     $p$-adic modular forms  & $\mathscr{M}_{U_{K_{p, P}}}$ & $\mathscr{M}$ \\
     Nearly overconvergent modular forms  & $\mathscr{N}^{\dagger}_{U_{K_{p, P}}}$ & $\mathscr{N}^{\dagger}$ \\
     Continuous functions  & $C_{\opn{cont}}(U_{K_{p, P}}(-1)^{\vee}, \mbb{Z}_p)$  & $C_{\opn{cont}}(\mbb{Z}_p, \mbb{Z}_p)$ \\
     Locally analytic functions  & $C^{\opn{la}}(U_{K_{p, P}}(-1)^{\vee}, \mbb{Q}_p)$ & $C^{\opn{la}}(\mbb{Z}_p, \mbb{Q}_p)$ \\
     Atkin--Serre operator  & $\Theta_{U_{K_{p, P}}}$ & $\Theta$ or $\theta$
    \end{tabular}
\end{center}

\vspace{0.3cm}

\begin{center}
    \textbf{Diagram of torsors and (nearly) overconvergent forms}
\end{center}

\vspace{0.3cm}

For $s \gg r \gg 1$, one has the commutative diagram:
\[
\begin{tikzcd}
                                                     & {\tilde{\mathcal{F}}_{r, s}} \arrow[rr] \arrow[rd] &                                                                         & {\tilde{\mathcal{F}}_{r, s}^{\opn{AI}}} \arrow[ld] \arrow[r] & P_{\opn{dR}}^{\opn{an}} \arrow[d] \\
\mathcal{IG}_{\infty} \arrow[d] \arrow[ru] \arrow[r] & \mathcal{F}_s \arrow[r] \arrow[rd]                 & \mathcal{F}_r \times_{\mathcal{X}_r} \mathcal{X}_s \arrow[d] \arrow[rr] &                                                              & M_{\opn{dR}}^{\opn{an}} \arrow[d] \\
\mathcal{X}_{\opn{ord}} \arrow[rr]                   &                                                    & \mathcal{X}_s \arrow[rr]                                                &                                                              & \mathcal{X}                      
\end{tikzcd}
\]
and spaces: 
\begin{itemize}
    \item $\mathscr{M} = \opn{H}^0(\mathcal{IG}_{\infty}, \mathcal{O}_{\mathcal{IG}_{\infty}})$ ($p$-adic modular forms)
    \item $\mathscr{M}^{\dagger} = \varinjlim_s \opn{H}^0(\mathcal{F}_s, \mathcal{O}_{\mathcal{F}_s})$ (overconvergent modular forms)
    \item $\mathscr{N}^{\dagger, \opn{AI}} = \varinjlim_{r, s} \opn{H}^0(\tilde{\mathcal{F}}^{\opn{AI}}_{r, s}, \mathcal{O}_{\tilde{\mathcal{F}}^{\opn{AI}}_{r, s}})$ (nearly overconvergent modular forms as in \cite{AI_LLL})
    \item $\mathscr{N}^{\dagger} = \varinjlim_{r, s} \opn{H}^0(\tilde{\mathcal{F}}_{r, s}, \mathcal{O}_{\tilde{\mathcal{F}}_{r, s}})$ (nearly overconvergent modular forms)
    \item $\mathscr{N}^{\opn{hol}} = \opn{H}^0( P_{\opn{dR}}^{\opn{an}}, \mathcal{O}_{P_{\opn{dR}}^{\opn{an}}})$ (nearly holomorphic modular forms)
    \item $\mathscr{M}^{\opn{hol}} = \opn{H}^0( M_{\opn{dR}}^{\opn{an}}, \mathcal{O}_{M_{\opn{dR}}^{\opn{an}}})$ (holomorphic modular forms)
\end{itemize}
fitting into the diagram:
\[
\begin{tikzcd}
            & \mathscr{M}^{\dagger} \arrow[d, hook] \arrow[ld, hook]      & \mathscr{M}^{\opn{hol}} \arrow[l, hook] \arrow[d, hook]  \\
\mathscr{M} & {\mathscr{N}^{\dagger, \opn{AI}}} \arrow[d, hook] \arrow[l] & \mathscr{N}^{\opn{hol}} \arrow[l, hook] \arrow[ld, hook] \\
            & \mathscr{N}^{\dagger} \arrow[lu]                            &                                                         
\end{tikzcd}
\]

\providecommand{\bysame}{\leavevmode\hbox to3em{\hrulefill}\thinspace}
\providecommand{\MR}{\relax\ifhmode\unskip\space\fi MR }
\providecommand{\MRhref}[2]{%
  \href{http://www.ams.org/mathscinet-getitem?mr=#1}{#2}
}
\providecommand{\href}[2]{#2}

\Addresses

\end{document}